\documentclass[a4paper]{amsart}
\usepackage{amsmath}
\usepackage{amsaddr}
\usepackage{a4wide}
\usepackage{bm}
\usepackage{amsthm}
\usepackage{amssymb}
\usepackage{stmaryrd}
\usepackage[colorlinks=true]{hyperref}
\newtheorem{proposition}{Proposition}
\newtheorem{definition}{Definition} 
\newtheorem{example}{Example} 
\newtheorem{remark}{Remark} 
\newtheorem{theorem}{Theorem}
\newtheorem{corollary}{Corollary} 
\newcommand{\interp}[1] {\left\llbracket #1 \right\rrbracket}
\newcommand{\annoted}[3]{\overbrace{#3}^{#2}\left.\vphantom{#3}\right\rbrace{\scriptstyle #1}}

\usepackage{tikz}
\usetikzlibrary{backgrounds, shapes.geometric, decorations.pathreplacing}

\usetikzlibrary{decorations.markings}

\title{Strong shift equivalence as a category notion}

\pgfdeclarelayer{edgelayer}\pgfdeclarelayer{nodelayer}\pgfsetlayers{background,edgelayer,nodelayer,main}
\tikzstyle{none}=[]
\tikzstyle{default_label}=[label position=right]
\tikzstyle{monoid}=[inner sep=0.5pt,minimum width=.3cm,minimum height=.3cm,draw=black,shape=diamond,fill=white,font=\footnotesize]
\tikzstyle{comonoid}=[inner sep=0.5pt,minimum width=.3cm,minimum height=.3cm,draw=black,shape=diamond,fill=white,font=\footnotesize]
\tikzstyle{bmonoid}=[inner sep=0.5pt,minimum width=.3cm,minimum height=.3cm,draw=black,shape=diamond,fill=black,font=\footnotesize]
\tikzstyle{bcomonoid}=[inner sep=0.5pt,minimum width=.3cm,minimum height=.3cm,draw=black,shape=diamond,fill=black,font=\footnotesize]
\tikzstyle{anymonoid}=[inner sep=0.5pt,minimum width=.3cm,minimum height=.3cm,draw=black,shape=diamond,font=\footnotesize,pattern=north west lines,pattern color=black,preaction={fill=white}]
\tikzstyle{anycomonoid}=[inner sep=0.5pt,minimum width=.3cm,minimum height=.3cm,draw=black,shape=diamond,font=\footnotesize,pattern=north west lines,pattern color=black,preaction={fill=white}]
\tikzstyle{hmonoid}=[inner sep=0.5pt,minimum width=.2cm,minimum height=.2cm,draw=black,shape=circle,fill=black,font=\footnotesize]
\tikzstyle{gmonoid}=[inner sep=0.5pt,minimum width=.2cm,minimum height=.2cm,draw=black,shape=circle,fill=white,font=\footnotesize]
\tikzstyle{square}=[shape=square,fill=white]
\makeatletter\pgfkeys{/tikz/stretch/.code 2 args=\tikz@stretch{#1}{#2}}\def\tikz@stretch#1#2{\pgfpointtransformed{\pgfpoint{0cm}{0cm}}\pgf@xa=\pgf@x\pgf@ya=\pgf@y\pgfpointtransformed{\pgfpoint{#1cm}{#2cm}}\advance\pgf@x by-\pgf@xa\ifdim\pgf@x<0pt\pgf@x=-\pgf@x\fi\advance\pgf@y by-\pgf@ya\ifdim\pgf@y<0pt\pgf@y=-\pgf@y\fi\pgfkeysalso{/tikz/minimum width/.expanded=\the\pgf@x}\pgfkeysalso{/tikz/minimum height/.expanded=\the\pgf@y}}\makeatother

\begin{document}
\author{Emmanuel Jeandel}
\address{Universit\'e de Lorraine, CNRS, Inria, LORIA, F 54000 Nancy, France}
\email{emmanuel.jeandel@loria.fr}

\begin{abstract}
  In this paper, we present a completely radical way to investigate the main problem of symbolic dynamics, the conjugacy problem, by proving that this problem actually relates to a natural question in category theory regarding the theory of traced bialgebras.

As a consequence of this theory, we obtain a systematic way of obtaining new invariants for the conjugacy problem by looking at existing bialgebras in the literature.
\end{abstract}
\maketitle

\section*{Introduction}

Possibly the most important question in symbolic dynamics \cite{LindMarcus} is the decidability of the conjugacy problem: decide if two symbolic dynamic systems (more precisely  subshifts of finite  type) are isomorphic \cite{BoyleOpen}.
Subshifts of finite type can be interpreted as biinfinite walks in a finite graph, or equivalently as biinfinite words over a finite alphabet forbidding a given finite set of words. In this article, we will be interested in the matrix representation, where a subshift of finite type is just a matrix of nonnegative integers (but the graph representation will be central to understand the results and the intuition). In this context, two matrices $M,N$ are \emph{strong shift equivalent} (represent isomorphic subshifts) if and only if $M \equiv N$ where $\equiv$ is the smallest equivalence relation s.t. $RS \equiv SR$ for all nonnegative, possibly nonsquare, matrices $R,S$.

The goal of this article is to investigate strong shift equivalence from the categorical point of view, using the concept of \emph{props}.
Props is a relatively old concept from MacLane's articles on categorical algebra, and is a natural formalism in which one can express easily circuits, i.e. diagrams made of boxes linked by wires. It has proven quite successful recently in categorical quantum mechanics, in particular with the introduction of the ZX-calculus \cite{coecke2011interacting}, a graphical language for quantum circuits (and more), and in linear algebra, by providing a new graphical way to reason about matrices \cite{bonchi2017interacting,Zanasi}. Similarities between these two languages are given in \cite{ZXZHetc}.

On the center of the representation of linear algebra by diagrams is the basic idea that the bialgebra prop is exactly the prop of nonnegative integer matrices, probably observed first in \cite{pirash}.
As symbolic dynamics is interested in a particular equivalence relation  on nonnegative integer matrices, this is a good start to think about them categorically. As the equation $RS \equiv SR$  reminds of the trace, it is therefore not surprising that there is a strong link between the traced bialgebra prop (matrices with a notion of traces) and strong shift equivalence.

The link is not completely immediate due to the fact that the classical notion of trace  in a category is a notion of a \emph{partial trace}.
We will prove in this article that actually the traced bialgebra prop does not technically correspond to strong shift equivalence, but to another well studied notion of equivalence of matrices, \emph{flow equivalence} \cite{parry75}.
To obtain a full categorical formulation of strong shift equivalence, we will need to consider bialgebras \emph{with a distinguished morphism}.

It is important to note that the link between flow equivalence and bicommutative bialgebras is not new, and has been investigated previously by David Hillman \cite{Hillman}, but for some reasons has not been investigated thoroughly since. 

\vspace{5mm}

This categorical manipulation of the concepts of symbolic dynamics is not just an intellectual game but serves some purpose, as it has actually the following consequence: Suppose one knows a category which contains a traced bialgebra (billions examples abound in the literature, the most prominent examples being finite dimensional Hopf algebras). Then one can interpret matrices in this category in such a way that two matrices that are strongly shift equivalent will have the same interpretation. This gives a new way to obtain \emph{invariants} for strong shift equivalence (i.e. functions from matrices to some set $S$ s.t. two matrices that are strong shift equivalents are equal by the function), with the added property that our approach is complete: there exist at least one category for which having the same interpretation is a necessary and sufficient condition for strong shift equivalence.

\vspace{5mm}

The article is organized as follows. In the first section, we recall the classical definitions of a prop, and gives in particular the graphical interpretation of the prop of matrices. In the second section, we introduce traced props, and show the main theorem: the traced completion of matrices with noninteger coefficients corresponds to matrices quotiented by flow equivalence, and the traced completion of matrices with coefficients in $\mathbb{Z}_+[t]$ corresponds to matrices quotiented by strong shift equivalence.

In the last section, we explain how we can recover classical and new invariants for symbolic dynamics by exploiting the idea of interpreting matrices in traced bialgebra props. This section is developed from the  point of view of someone who knowns no invariant of strong shift equivalence: we look at existing and well known bialgebras and see which invariant we get from them. The fact one obtains this way fairly well known invariants from symbolic dynamics acts as a proof that the whole approach is successful.

\vspace{5mm}

The perfect reader for this article is someone familiar both with category theory (esp. the categorical approach to universal algebra developed by Lawvere) and symbolic dynamics. It should be accessible however to a reader with only a passing familiarity in at least one of the two domains.

\section{prop and finitely presented props}

\subsection{props}
A prop\footnote{prop was originally an acronym from PROducts and Permutations} is one of the categorical ways to represent circuits. Morphisms in a prop (may) represent circuits: we have two composition laws on circuits: parallel ($\otimes$) and sequential ($\circ$) compositions that should satisfy some obvious properties.

\begin{definition}[prop]
\label{def:prop}  
                A \textbf{prop} $\textbf{P}$ is the data, for each pair $(n,m) \in \mathbb{N}^2$ of a set\footnote{Technically we only consider locally small props, i.e. props where morphisms form a set, but the distinction is irrelevant for most of the paper.}
                $\textbf{P}[n,m]$, called the set of morphisms. An
                element $f \in \textbf{P}[n,m]$  is usually written
                $f:n\to m$. These sets are linked by the following
                operators:\vspace{0.25cm}
                \begin{itemize}
                        \item A \textbf{composition} $\circ: \textbf{P}[b,c]\times \textbf{P}[a,b]\to \textbf{P}[a,c]$ satisfying: $(f\circ g)\circ h=f\circ (g\circ h)$.\vspace{0.25cm}
                        \item A \textbf{tensor product} $\otimes: \textbf{P}[a,b]\times \textbf{P}[c,d]\to \textbf{P}[a+c,b+d]$, satisfying: $(f\otimes g)\otimes h=f\otimes (g\otimes h)$ and $(f\circ g)\otimes (h\circ k)=(f\otimes h)\circ (g\otimes k)$.\vspace{0.25cm}
                        \item An \textbf{empty morphism} $1:0\to 0$
                          such that $f\otimes 1= 1\otimes f=f$ for all
                          $f:a\to b$.\vspace{0.25cm}
                        \item An \textbf{identity} $id: 1\to 1$ such that
                          $f\circ id^{\otimes a}=id^{\otimes b} \circ
                          f=f$ for all $f:a\to b$. With the convention $id^{\otimes 0}=1$.
\vspace{0.25cm}
                        \item A \textbf{symmetry} $\sigma: 2\to 2$
                          satisfying: $\sigma^2=id^{\otimes 2}$ and such that, $\sigma_{a}\circ (f\otimes id)=(id\otimes f)\circ \sigma_b$, for all $f:a\to b$, where $\sigma_{n+1}=(1^{\otimes n}\otimes \sigma)\circ (1\otimes \sigma_n)$, with $\sigma_0 = id$.
                \end{itemize}

A prop functor $F$ (i.e. a morphism of props) from $P$ to $Q$ is a set of maps, again denoted $F$,  from $P[n,m]$ to $Q[n,m]$ such that  $F(f  \circ g) = F(f)  \circ F(g), F(f \otimes g) = F(f) \otimes F(g)$, and $F(id_P) = id_Q, F(0_P) = 0_Q, F(\sigma_P) = \sigma_Q$.
                \end{definition}
        
In the language of categories \cite{MacLane}, a prop is a small strict symmetric monoidal category whose monoid of object is spanned by a unique object.

    Props admit a nice diagrammatical representation that gives a topological interpretation to the axioms \cite{selinger2010survey}. A morphism $f:n\to m$ is represented as a box with $n$ inputs and $m$ outputs, with inputs on the left and outputs on the right. Composition is represented by plugging the boxes, and the tensor product by drawing the boxes side by side. The identity is represented by a single wire, the empty morphism by an empty diagram and the symmetry by wire crossing:
          
        \begin{center}
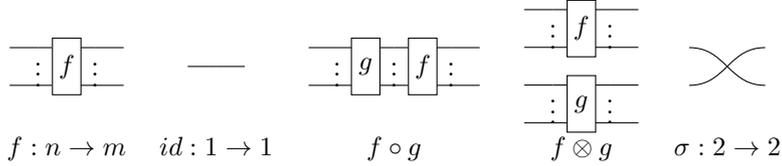
                \begin{tabular}{ccccc}
\begin{tikzpicture}[baseline=(current bounding box.center)]
\begin{pgfonlayer}{nodelayer}
\node [] (0) at (-0.000,-0.250) {};
\node [] (1) at (-0.375,-0.250) {};
\node [fill=white,label={[label position=center]$\vdots$}] (2) at (-0.375,-0.500) {};
\node [] (3) at (-0.000,-0.750) {};
\node [] (4) at (-0.375,-0.750) {};
\node [shape=rectangle,fill=white,draw,label={[label position=center]$f$},inner xsep=0.1875cm,inner ysep=0.375cm] (5) at (-0.750,-0.500) {};
\node [] (6) at (-0.750,-0.250) {};
\node [] (7) at (-0.750,-0.750) {};
\node [] (8) at (-0.750,-0.250) {};
\node [] (9) at (-0.750,-0.750) {};
\node [] (10) at (-1.125,-0.250) {};
\node [] (11) at (-1.500,-0.250) {};
\node [fill=white,label={[label position=center]$\vdots$}] (12) at (-1.125,-0.500) {};
\node [] (13) at (-1.125,-0.750) {};
\node [] (14) at (-1.500,-0.750) {};
\end{pgfonlayer}
\begin{pgfonlayer}{edgelayer}
\draw[] (0.center) to[out=180, in=  0] (1.center);
\draw[] (3.center) to[out=180, in=  0] (4.center);
\draw[] (1.center) to[out=180, in=  0] (8.center);
\draw[] (4.center) to[out=180, in=  0] (9.center);
\draw[] (6.center) to[out=180, in=  0] (10.center);
\draw[] (10.center) to[out=180, in=  0] (11.center);
\draw[] (7.center) to[out=180, in=  0] (13.center);
\draw[] (13.center) to[out=180, in=  0] (14.center);
\end{pgfonlayer}
\end{tikzpicture}
 & \begin{tikzpicture}[baseline=(current bounding box.center)]
\begin{pgfonlayer}{nodelayer}
\node [] (0) at (-0.000,-0.250) {};
\node [] (1) at (-0.375,-0.250) {};
\node [] (2) at (-0.750,-0.250) {};
\end{pgfonlayer}
\begin{pgfonlayer}{edgelayer}
\draw[] (0.center) to[out=180, in=  0] (1.center);
\draw[] (1.center) to[out=180, in=  0] (2.center);
\end{pgfonlayer}
\end{tikzpicture}
 & \begin{tikzpicture}[baseline=(current bounding box.center)]
\begin{pgfonlayer}{nodelayer}
\node [] (0) at (-0.000,-0.250) {};
\node [] (1) at (-0.375,-0.250) {};
\node [fill=white,label={[label position=center]$\vdots$}] (2) at (-0.375,-0.500) {};
\node [] (3) at (-0.000,-0.750) {};
\node [] (4) at (-0.375,-0.750) {};
\node [shape=rectangle,fill=white,draw,label={[label position=center]$f$},inner xsep=0.1875cm,inner ysep=0.375cm] (5) at (-0.750,-0.500) {};
\node [] (6) at (-0.750,-0.250) {};
\node [] (7) at (-0.750,-0.750) {};
\node [] (8) at (-0.750,-0.250) {};
\node [] (9) at (-0.750,-0.750) {};
\node [] (10) at (-1.125,-0.250) {};
\node [fill=white,label={[label position=center]$\vdots$}] (11) at (-1.125,-0.500) {};
\node [] (12) at (-1.125,-0.750) {};
\node [shape=rectangle,fill=white,draw,label={[label position=center]$g$},inner xsep=0.1875cm,inner ysep=0.375cm] (13) at (-1.500,-0.500) {};
\node [] (14) at (-1.500,-0.250) {};
\node [] (15) at (-1.500,-0.750) {};
\node [] (16) at (-1.500,-0.250) {};
\node [] (17) at (-1.500,-0.750) {};
\node [] (18) at (-1.875,-0.250) {};
\node [] (19) at (-2.250,-0.250) {};
\node [fill=white,label={[label position=center]$\vdots$}] (20) at (-1.875,-0.500) {};
\node [] (21) at (-1.875,-0.750) {};
\node [] (22) at (-2.250,-0.750) {};
\end{pgfonlayer}
\begin{pgfonlayer}{edgelayer}
\draw[] (0.center) to[out=180, in=  0] (1.center);
\draw[] (3.center) to[out=180, in=  0] (4.center);
\draw[] (1.center) to[out=180, in=  0] (8.center);
\draw[] (4.center) to[out=180, in=  0] (9.center);
\draw[] (6.center) to[out=180, in=  0] (10.center);
\draw[] (7.center) to[out=180, in=  0] (12.center);
\draw[] (10.center) to[out=180, in=  0] (16.center);
\draw[] (12.center) to[out=180, in=  0] (17.center);
\draw[] (14.center) to[out=180, in=  0] (18.center);
\draw[] (18.center) to[out=180, in=  0] (19.center);
\draw[] (15.center) to[out=180, in=  0] (21.center);
\draw[] (21.center) to[out=180, in=  0] (22.center);
\end{pgfonlayer}
\end{tikzpicture}
  & \begin{tikzpicture}[baseline=(current bounding box.center)]
\begin{pgfonlayer}{nodelayer}
\node [] (0) at (-0.000,-0.250) {};
\node [] (1) at (-0.375,-0.250) {};
\node [fill=white,label={[label position=center]$\vdots$}] (2) at (-0.375,-0.500) {};
\node [] (3) at (-0.000,-0.750) {};
\node [] (4) at (-0.375,-0.750) {};
\node [shape=rectangle,fill=white,draw,label={[label position=center]$f$},inner xsep=0.1875cm,inner ysep=0.375cm] (5) at (-0.750,-0.500) {};
\node [] (6) at (-0.750,-0.250) {};
\node [] (7) at (-0.750,-0.750) {};
\node [] (8) at (-0.750,-0.250) {};
\node [] (9) at (-0.750,-0.750) {};
\node [] (10) at (-1.125,-0.250) {};
\node [] (11) at (-1.500,-0.250) {};
\node [fill=white,label={[label position=center]$\vdots$}] (12) at (-1.125,-0.500) {};
\node [] (13) at (-1.125,-0.750) {};
\node [] (14) at (-1.500,-0.750) {};
\node [] (15) at (-0.000,-1.250) {};
\node [] (16) at (-0.375,-1.250) {};
\node [fill=white,label={[label position=center]$\vdots$}] (17) at (-0.375,-1.500) {};
\node [] (18) at (-0.000,-1.750) {};
\node [] (19) at (-0.375,-1.750) {};
\node [shape=rectangle,fill=white,draw,label={[label position=center]$g$},inner xsep=0.1875cm,inner ysep=0.375cm] (20) at (-0.750,-1.500) {};
\node [] (21) at (-0.750,-1.250) {};
\node [] (22) at (-0.750,-1.750) {};
\node [] (23) at (-0.750,-1.250) {};
\node [] (24) at (-0.750,-1.750) {};
\node [] (25) at (-1.125,-1.250) {};
\node [] (26) at (-1.500,-1.250) {};
\node [fill=white,label={[label position=center]$\vdots$}] (27) at (-1.125,-1.500) {};
\node [] (28) at (-1.125,-1.750) {};
\node [] (29) at (-1.500,-1.750) {};
\end{pgfonlayer}
\begin{pgfonlayer}{edgelayer}
\draw[] (0.center) to[out=180, in=  0] (1.center);
\draw[] (3.center) to[out=180, in=  0] (4.center);
\draw[] (1.center) to[out=180, in=  0] (8.center);
\draw[] (4.center) to[out=180, in=  0] (9.center);
\draw[] (6.center) to[out=180, in=  0] (10.center);
\draw[] (10.center) to[out=180, in=  0] (11.center);
\draw[] (7.center) to[out=180, in=  0] (13.center);
\draw[] (13.center) to[out=180, in=  0] (14.center);
\draw[] (15.center) to[out=180, in=  0] (16.center);
\draw[] (18.center) to[out=180, in=  0] (19.center);
\draw[] (16.center) to[out=180, in=  0] (23.center);
\draw[] (19.center) to[out=180, in=  0] (24.center);
\draw[] (21.center) to[out=180, in=  0] (25.center);
\draw[] (25.center) to[out=180, in=  0] (26.center);
\draw[] (22.center) to[out=180, in=  0] (28.center);
\draw[] (28.center) to[out=180, in=  0] (29.center);
\end{pgfonlayer}
\end{tikzpicture}
 & \begin{tikzpicture}[baseline=(current bounding box.center)]
\begin{pgfonlayer}{nodelayer}
\node [] (0) at (-0.000,-0.250) {};
\node [] (1) at (-0.000,-0.750) {};
\node [] (2) at (-0.500,-0.500) {};
\node [] (3) at (-1.000,-0.250) {};
\node [] (4) at (-1.000,-0.750) {};
\end{pgfonlayer}
\begin{pgfonlayer}{edgelayer}
\draw[] (1.center) to[out=180, in=-44.9] (2.center) to[out=135, in=  0] (3.center);
\draw[] (0.center) to[out=180, in= 45] (2.center) to[out=-135, in=  0] (4.center);
\end{pgfonlayer}
\end{tikzpicture}
 \\[0.5cm]
                        $f:n\to m$ & $id:1\to 1$ & $f\circ g$ & $f\otimes g$ & $\sigma : 2\to 2$
                \end{tabular}
        \end{center}

        These notations are well chosen, as the equations defining a prop correspond to equalities of diagrams that are obvious to the eye~\cite{selinger2010survey,joyal1991geometry}, so we can equivalently work with equations or with diagrams. To give one example, the symmetry axioms express that the boxes can move through wires:
        
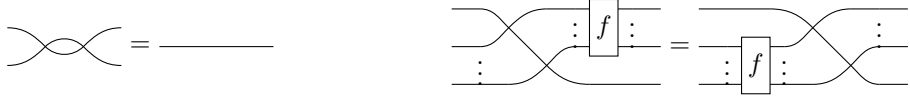
\begin{center}
        $\begin{tikzpicture}[baseline=(current bounding box.center)]
\begin{pgfonlayer}{nodelayer}
\node [] (0) at (-0.000,-0.500) {};
\node [] (1) at (-0.500,-0.500) {};
\node [] (2) at (-1.000,-0.500) {};
\node [] (3) at (-1.500,-0.500) {};
\node [label={[label position=center]=}] (4) at (-1.750,-0.500) {};
\node [] (5) at (-2.000,-0.250) {};
\node [] (6) at (-2.000,-0.750) {};
\node [] (7) at (-2.500,-0.500) {};
\node [] (8) at (-3.000,-0.500) {};
\node [] (9) at (-3.500,-0.250) {};
\node [] (10) at (-3.500,-0.750) {};
\end{pgfonlayer}
\begin{pgfonlayer}{edgelayer}
\draw[] (0.center) to[out=180, in=  0] (1.center) to[out=180, in=  0] (2.center) to[out=180, in=  0] (3.center);
\draw[] (5.center) to[out=180, in= 45] (7.center) to[out=-135, in=-44.9] (8.center) to[out=135, in=  0] (9.center);
\draw[] (6.center) to[out=180, in=-44.9] (7.center) to[out=135, in= 45] (8.center) to[out=-135, in=  0] (10.center);
\end{pgfonlayer}
\end{tikzpicture}
 \qquad\qquad\qquad \begin{tikzpicture}[baseline=(current bounding box.center)]
\begin{pgfonlayer}{nodelayer}
\node [] (0) at (-0.000,-0.250) {};
\node [] (1) at (-0.375,-0.250) {};
\node [fill=white,label={[label position=center]$\vdots$}] (2) at (-0.375,-0.500) {};
\node [] (3) at (-0.000,-0.750) {};
\node [] (4) at (-0.375,-0.750) {};
\node [] (5) at (-0.000,-1.250) {};
\node [] (6) at (-0.375,-1.250) {};
\node [] (7) at (-0.750,-0.250) {};
\node [] (8) at (-0.750,-1.000) {};
\node [] (9) at (-1.250,-0.500) {};
\node [] (10) at (-1.250,-1.250) {};
\node [] (11) at (-1.812,-0.250) {};
\node [] (12) at (-2.750,-0.250) {};
\node [] (13) at (-1.625,-0.750) {};
\node [fill=white,label={[label position=center]$\vdots$}] (14) at (-1.625,-1.000) {};
\node [] (15) at (-1.625,-1.250) {};
\node [shape=rectangle,fill=white,draw,label={[label position=center]$f$},inner xsep=0.1875cm,inner ysep=0.375cm] (16) at (-2.000,-1.000) {};
\node [] (17) at (-2.000,-0.750) {};
\node [] (18) at (-2.000,-1.250) {};
\node [] (19) at (-2.000,-0.750) {};
\node [] (20) at (-2.000,-1.250) {};
\node [] (21) at (-2.375,-0.750) {};
\node [] (22) at (-2.750,-0.750) {};
\node [fill=white,label={[label position=center]$\vdots$}] (23) at (-2.375,-1.000) {};
\node [] (24) at (-2.375,-1.250) {};
\node [] (25) at (-2.750,-1.250) {};
\node [label={[label position=center]=}] (26) at (-3.000,-0.750) {};
\node [] (27) at (-3.250,-0.250) {};
\node [] (28) at (-3.625,-0.250) {};
\node [fill=white,label={[label position=center]$\vdots$}] (29) at (-3.625,-0.500) {};
\node [] (30) at (-3.250,-0.750) {};
\node [] (31) at (-3.625,-0.750) {};
\node [shape=rectangle,fill=white,draw,label={[label position=center]$f$},inner xsep=0.1875cm,inner ysep=0.375cm] (32) at (-4.000,-0.500) {};
\node [] (33) at (-4.000,-0.250) {};
\node [] (34) at (-4.000,-0.750) {};
\node [] (35) at (-4.000,-0.250) {};
\node [] (36) at (-4.000,-0.750) {};
\node [] (37) at (-4.375,-0.250) {};
\node [fill=white,label={[label position=center]$\vdots$}] (38) at (-4.375,-0.500) {};
\node [] (39) at (-4.375,-0.750) {};
\node [] (40) at (-3.250,-1.250) {};
\node [] (41) at (-4.188,-1.250) {};
\node [] (42) at (-4.750,-0.250) {};
\node [] (43) at (-4.750,-1.000) {};
\node [] (44) at (-5.250,-0.500) {};
\node [] (45) at (-5.250,-1.250) {};
\node [] (46) at (-5.625,-0.250) {};
\node [] (47) at (-6.000,-0.250) {};
\node [] (48) at (-5.625,-0.750) {};
\node [] (49) at (-6.000,-0.750) {};
\node [fill=white,label={[label position=center]$\vdots$}] (50) at (-5.625,-1.000) {};
\node [] (51) at (-5.625,-1.250) {};
\node [] (52) at (-6.000,-1.250) {};
\end{pgfonlayer}
\begin{pgfonlayer}{edgelayer}
\draw[] (0.center) to[out=180, in=  0] (1.center);
\draw[] (3.center) to[out=180, in=  0] (4.center);
\draw[] (5.center) to[out=180, in=  0] (6.center);
\draw[] (1.center) to[out=180, in=  0] (7.center);
\draw[] (4.center) to[out=180, in= 45] (8.center) to[out=-135, in=  0] (10.center);
\draw[] (6.center) to[out=180, in=-44.9] (8.center) to[out=135, in=-44.9] (9.center) to[out=135, in=  0] (11.center);
\draw[] (11.center) to[out=180, in=  0] (12.center);
\draw[] (7.center) to[out=180, in= 45] (9.center) to[out=-135, in=  0] (13.center);
\draw[] (10.center) to[out=180, in=  0] (15.center);
\draw[] (13.center) to[out=180, in=  0] (19.center);
\draw[] (15.center) to[out=180, in=  0] (20.center);
\draw[] (17.center) to[out=180, in=  0] (21.center);
\draw[] (21.center) to[out=180, in=  0] (22.center);
\draw[] (18.center) to[out=180, in=  0] (24.center);
\draw[] (24.center) to[out=180, in=  0] (25.center);
\draw[] (27.center) to[out=180, in=  0] (28.center);
\draw[] (30.center) to[out=180, in=  0] (31.center);
\draw[] (28.center) to[out=180, in=  0] (35.center);
\draw[] (31.center) to[out=180, in=  0] (36.center);
\draw[] (33.center) to[out=180, in=  0] (37.center);
\draw[] (34.center) to[out=180, in=  0] (39.center);
\draw[] (40.center) to[out=180, in=  0] (41.center);
\draw[] (37.center) to[out=180, in=  0] (42.center);
\draw[] (39.center) to[out=180, in= 45] (43.center) to[out=-135, in=  0] (45.center);
\draw[] (41.center) to[out=180, in=-44.9] (43.center) to[out=135, in=-44.9] (44.center) to[out=135, in=  0] (46.center);
\draw[] (46.center) to[out=180, in=  0] (47.center);
\draw[] (42.center) to[out=180, in= 45] (44.center) to[out=-135, in=  0] (48.center);
\draw[] (48.center) to[out=180, in=  0] (49.center);
\draw[] (45.center) to[out=180, in=  0] (51.center);
\draw[] (51.center) to[out=180, in=  0] (52.center);
\end{pgfonlayer}
\end{tikzpicture}
 $
\end{center}

Here are a few examples:
\begin{example}
  Let $X$ be a set. $Fun_X$ is the prop where the morphisms $n \mapsto m$ are exactly the functions from $X^n$ to $X^m$. The composition $\circ$ is the classical composition of functions, the tensor product $\otimes$ is the cartesian product.
\end{example}

\begin{example}
  Let $R$ be a semiring and $d$ an integer. $Mat_R^d$ is the prop where the morphisms $n \mapsto m$ are the matrices with coefficients in $R$ of size $d^m \times d^n$. The composition $\circ$ is the product of matrices, the tensor product $\otimes$ is the Kronecker product of matrices.
\end{example}

\subsection{Presentations}

An interesting way to give props is with generators and relations. 
In a similar way to universal algebra, a formal definition can be
obtained by first considering the free prop given by the generators
(consider all circuits obtained from the generators using the
compositions $\circ$, $\otimes$), and by then  quotienting the category by the equations. See \cite{baez17} for the formal definition.
We only only give the universal property of such a prop:

\begin{definition}
\label{defn:universalprop}  
  Let $G$ be a set of generators and $E$ be a set of equations. The prop $P$ presented by $(G,E)$ is (one of) the prop satisfying the following properties:
  \begin{itemize}
  \item $P$ contains the generators $G$ and the generators satisfy the equations of $E$
  \item For any other prop $Q$ that has this property, there is a unique prop functor from $P$ to $Q$ that sends the generators to the generators.
  \end{itemize}

  The existence of such a prop is given in \cite{baez17}. The uniqueness (upto prop isomorphism) is implied by the universal property.
\end{definition}  

We will only give here a few examples that we hope are illuminating.

\begin{example}
  In the prop with no generators and no equations, the only thing we can do is use the identity morphism $id$ and the symmetry $\sigma$ to obtain diagrams.

  Thus every diagram looks like the following one:
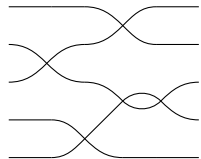
\begin{center}
\begin{tikzpicture}[baseline=(current bounding box.center)]
\begin{pgfonlayer}{nodelayer}
\node [] (0) at (-0.000,-0.250) {};
\node [] (1) at (-0.562,-0.250) {};
\node [] (2) at (-0.000,-0.750) {};
\node [] (3) at (-0.562,-0.750) {};
\node [] (4) at (-0.000,-1.250) {};
\node [] (5) at (-0.000,-1.750) {};
\node [] (6) at (-0.500,-1.500) {};
\node [] (7) at (-0.000,-2.250) {};
\node [] (8) at (-0.563,-2.250) {};
\node [] (9) at (-1.000,-0.500) {};
\node [] (10) at (-1.000,-1.500) {};
\node [] (11) at (-1.000,-2.250) {};
\node [] (12) at (-1.500,-0.250) {};
\node [] (13) at (-1.500,-0.750) {};
\node [] (14) at (-1.500,-1.250) {};
\node [] (15) at (-1.500,-2.000) {};
\node [] (16) at (-1.938,-0.250) {};
\node [] (17) at (-2.500,-0.250) {};
\node [] (18) at (-2.000,-1.000) {};
\node [] (19) at (-2.500,-0.750) {};
\node [] (20) at (-2.500,-1.250) {};
\node [] (21) at (-1.938,-1.750) {};
\node [] (22) at (-2.500,-1.750) {};
\node [] (23) at (-1.938,-2.250) {};
\node [] (24) at (-2.500,-2.250) {};
\end{pgfonlayer}
\begin{pgfonlayer}{edgelayer}
\draw[] (0.center) to[out=180, in=  0] (1.center);
\draw[] (2.center) to[out=180, in=  0] (3.center);
\draw[] (7.center) to[out=180, in=  0] (8.center);
\draw[] (8.center) to[out=180, in=  0] (11.center);
\draw[] (3.center) to[out=180, in=-44.9] (9.center) to[out=135, in=  0] (12.center);
\draw[] (1.center) to[out=180, in= 45] (9.center) to[out=-135, in=  0] (13.center);
\draw[] (4.center) to[out=180, in= 45] (6.center) to[out=-135, in=-44.9] (10.center) to[out=135, in=  0] (14.center);
\draw[] (12.center) to[out=180, in=  0] (16.center);
\draw[] (16.center) to[out=180, in=  0] (17.center);
\draw[] (14.center) to[out=180, in=-44.9] (18.center) to[out=135, in=  0] (19.center);
\draw[] (13.center) to[out=180, in= 45] (18.center) to[out=-135, in=  0] (20.center);
\draw[] (11.center) to[out=180, in=-44.9] (15.center) to[out=135, in=  0] (21.center);
\draw[] (21.center) to[out=180, in=  0] (22.center);
\draw[] (5.center) to[out=180, in=-44.9] (6.center) to[out=135, in= 45] (10.center) to[out=-135, in= 45] (15.center) to[out=-135, in=  0] (23.center);
\draw[] (23.center) to[out=180, in=  0] (24.center);
\end{pgfonlayer}
\end{tikzpicture}

\end{center}
and it is easy to see that this prop is (isomorphic to) the prop $Sym$ of invertible maps, i.e. $Sym[n,m]$ is empty if $n\not=m$ and $Sym[n,n]$ is the set of 
all invertible functions from $[1,\dots, n]$ to itself, with $\circ$ as function composition and $\otimes$ is disjoint union.
\end{example}

\begin{example}
  Consider the prop with two generators $\mu: 2 \rightarrow 1$ and $\eta: 0 \rightarrow 1$ that satisfy the equations: ${\mu\circ (\eta\otimes \mathrm{id})=\mu\circ (\mathrm{id}\otimes \eta)=\mathrm{id}
 }$ , 
  ${\mu\circ (\mu\otimes \mathrm{id})=\mu\circ (\mathrm{id}\otimes \mu)
 }$ and ${\mu\circ \sigma=\mu
 }$.

If we depict the generators as
{\begin{tikzpicture}[baseline=(current bounding box.center)]
\begin{pgfonlayer}{nodelayer}
\node [] (0) at (-0.000,-0.250) {};
\node [style=monoid] (1) at (-0.500,-0.250) {};
\node [] (2) at (-1.000,-0.125) {};
\node [] (3) at (-1.000,-0.375) {};
\end{pgfonlayer}
\begin{pgfonlayer}{edgelayer}
\draw[] (0.center) to[out=180, in=  0] (1.center);
\draw[] (1.center) to[out=153, in=  0] (2.center);
\draw[] (1.center) to[out=-153, in=  0] (3.center);
\end{pgfonlayer}
\end{tikzpicture}
 } and
{\begin{tikzpicture}[baseline=(current bounding box.center)]
\begin{pgfonlayer}{nodelayer}
\node [] (0) at (-0.000,-0.250) {};
\node [style=monoid] (1) at (-0.500,-0.250) {};
\end{pgfonlayer}
\begin{pgfonlayer}{edgelayer}
\draw[] (0.center) to[out=180, in=  0] (1.center);
\end{pgfonlayer}
\end{tikzpicture}
 }, the equations become:

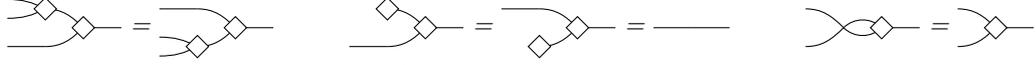
\hfill    {\begin{tikzpicture}[baseline=(current bounding box.center)]
\begin{pgfonlayer}{nodelayer}
\node [] (0) at (-0.000,-0.500) {};
\node [style=monoid] (1) at (-0.500,-0.500) {};
\node [] (2) at (-1.000,-0.250) {};
\node [] (3) at (-1.500,-0.250) {};
\node [style=monoid] (4) at (-1.000,-0.750) {};
\node [] (5) at (-1.500,-0.625) {};
\node [] (6) at (-1.500,-0.875) {};
\node [label={[label position=center]=}] (7) at (-1.750,-0.500) {};
\node [] (8) at (-2.000,-0.500) {};
\node [style=monoid] (9) at (-2.500,-0.500) {};
\node [style=monoid] (10) at (-3.000,-0.250) {};
\node [] (11) at (-3.500,-0.125) {};
\node [] (12) at (-3.500,-0.375) {};
\node [] (13) at (-3.000,-0.750) {};
\node [] (14) at (-3.500,-0.750) {};
\end{pgfonlayer}
\begin{pgfonlayer}{edgelayer}
\draw[] (0.center) to[out=180, in=  0] (1.center);
\draw[] (1.center) to[out=135, in=  0] (2.center);
\draw[] (2.center) to[out=180, in=  0] (3.center);
\draw[] (1.center) to[out=-135, in=  0] (4.center);
\draw[] (4.center) to[out=153, in=  0] (5.center);
\draw[] (4.center) to[out=-153, in=  0] (6.center);
\draw[] (8.center) to[out=180, in=  0] (9.center);
\draw[] (9.center) to[out=135, in=  0] (10.center);
\draw[] (10.center) to[out=153, in=  0] (11.center);
\draw[] (10.center) to[out=-153, in=  0] (12.center);
\draw[] (9.center) to[out=-135, in=  0] (13.center);
\draw[] (13.center) to[out=180, in=  0] (14.center);
\end{pgfonlayer}
\end{tikzpicture}
 } \hfill
{\begin{tikzpicture}[baseline=(current bounding box.center)]
\begin{pgfonlayer}{nodelayer}
\node [] (0) at (-0.000,-0.500) {};
\node [] (1) at (-0.500,-0.500) {};
\node [] (2) at (-1.000,-0.500) {};
\node [label={[label position=center]=}] (3) at (-1.250,-0.500) {};
\node [] (4) at (-1.500,-0.500) {};
\node [style=monoid] (5) at (-2.000,-0.500) {};
\node [] (6) at (-2.500,-0.250) {};
\node [] (7) at (-3.000,-0.250) {};
\node [style=monoid] (8) at (-2.500,-0.750) {};
\node [label={[label position=center]=}] (9) at (-3.250,-0.500) {};
\node [] (10) at (-3.500,-0.500) {};
\node [style=monoid] (11) at (-4.000,-0.500) {};
\node [style=monoid] (12) at (-4.500,-0.250) {};
\node [] (13) at (-4.500,-0.750) {};
\node [] (14) at (-5.000,-0.750) {};
\end{pgfonlayer}
\begin{pgfonlayer}{edgelayer}
\draw[] (0.center) to[out=180, in=  0] (1.center);
\draw[] (1.center) to[out=180, in=  0] (2.center);
\draw[] (4.center) to[out=180, in=  0] (5.center);
\draw[] (5.center) to[out=135, in=  0] (6.center);
\draw[] (6.center) to[out=180, in=  0] (7.center);
\draw[] (5.center) to[out=-135, in=  0] (8.center);
\draw[] (10.center) to[out=180, in=  0] (11.center);
\draw[] (11.center) to[out=135, in=  0] (12.center);
\draw[] (11.center) to[out=-135, in=  0] (13.center);
\draw[] (13.center) to[out=180, in=  0] (14.center);
\end{pgfonlayer}
\end{tikzpicture}
 } \hfill
{\begin{tikzpicture}[baseline=(current bounding box.center)]
\begin{pgfonlayer}{nodelayer}
\node [] (0) at (-0.000,-0.500) {};
\node [style=monoid] (1) at (-0.500,-0.500) {};
\node [] (2) at (-1.000,-0.250) {};
\node [] (3) at (-1.000,-0.750) {};
\node [label={[label position=center]=}] (4) at (-1.250,-0.500) {};
\node [] (5) at (-1.500,-0.500) {};
\node [style=monoid] (6) at (-2.000,-0.500) {};
\node [] (7) at (-2.500,-0.500) {};
\node [] (8) at (-3.000,-0.250) {};
\node [] (9) at (-3.000,-0.750) {};
\end{pgfonlayer}
\begin{pgfonlayer}{edgelayer}
\draw[] (0.center) to[out=180, in=  0] (1.center);
\draw[] (1.center) to[out=135, in=  0] (2.center);
\draw[] (1.center) to[out=-135, in=  0] (3.center);
\draw[] (5.center) to[out=180, in=  0] (6.center);
\draw[] (6.center) to[out=-135, in=-44.9] (7.center) to[out=135, in=  0] (8.center);
\draw[] (6.center) to[out=135, in= 45] (7.center) to[out=-135, in=  0] (9.center);
\end{pgfonlayer}
\end{tikzpicture}
 } \hfill

Using the first equation, one can ``define'' a generalized  version of the diamond of type $k \rightarrow 1$ inputs by composing $k-1$ diamonds $2 \rightarrow 1$. The first equation essentially states that the order of composition is irrelevant.

Now it can be proven easily that every morphism can be represented by a diagram of the following form:
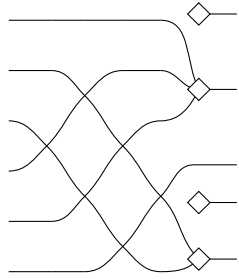
\begin{center}
\begin{tikzpicture}[baseline=(current bounding box.center)]
\begin{pgfonlayer}{nodelayer}
\node [] (0) at (-0.000,-0.250) {};
\node [style=monoid] (1) at (-0.500,-0.250) {};
\node [] (2) at (-0.000,-1.250) {};
\node [style=monoid] (3) at (-0.500,-1.250) {};
\node [] (4) at (-0.000,-2.250) {};
\node [] (5) at (-0.563,-2.250) {};
\node [] (6) at (-0.000,-2.750) {};
\node [style=monoid] (7) at (-0.500,-2.750) {};
\node [] (8) at (-0.000,-3.500) {};
\node [style=monoid] (9) at (-0.500,-3.500) {};
\node [] (10) at (-1.000,-0.333) {};
\node [] (11) at (-1.000,-1.000) {};
\node [] (12) at (-1.000,-1.667) {};
\node [] (13) at (-1.000,-2.667) {};
\node [] (14) at (-1.000,-3.667) {};
\node [] (15) at (-1.500,-0.333) {};
\node [] (16) at (-1.500,-1.000) {};
\node [] (17) at (-1.500,-2.000) {};
\node [] (18) at (-1.500,-3.333) {};
\node [] (19) at (-2.000,-0.333) {};
\node [] (20) at (-2.000,-1.333) {};
\node [] (21) at (-2.000,-2.667) {};
\node [] (22) at (-2.000,-3.667) {};
\node [] (23) at (-2.438,-0.333) {};
\node [] (24) at (-3.000,-0.333) {};
\node [] (25) at (-2.438,-1.000) {};
\node [] (26) at (-3.000,-1.000) {};
\node [] (27) at (-2.500,-2.000) {};
\node [] (28) at (-3.000,-1.667) {};
\node [] (29) at (-3.000,-2.333) {};
\node [] (30) at (-2.438,-3.000) {};
\node [] (31) at (-3.000,-3.000) {};
\node [] (32) at (-2.438,-3.667) {};
\node [] (33) at (-3.000,-3.667) {};
\end{pgfonlayer}
\begin{pgfonlayer}{edgelayer}
\draw[] (0.center) to[out=180, in=  0] (1.center);
\draw[] (2.center) to[out=180, in=  0] (3.center);
\draw[] (4.center) to[out=180, in=  0] (5.center);
\draw[] (6.center) to[out=180, in=  0] (7.center);
\draw[] (8.center) to[out=180, in=  0] (9.center);
\draw[] (3.center) to[out=117, in=  0] (10.center);
\draw[] (3.center) to[out=180, in=  0] (11.center);
\draw[] (3.center) to[out=-116, in=  0] (12.center);
\draw[] (9.center) to[out=-135, in=  0] (14.center);
\draw[] (10.center) to[out=180, in=  0] (15.center);
\draw[] (11.center) to[out=180, in=  0] (16.center);
\draw[] (15.center) to[out=180, in=  0] (19.center);
\draw[] (5.center) to[out=180, in= 45] (13.center) to[out=-135, in= 45] (18.center) to[out=-135, in=  0] (22.center);
\draw[] (19.center) to[out=180, in=  0] (23.center);
\draw[] (23.center) to[out=180, in=  0] (24.center);
\draw[] (9.center) to[out=135, in=-44.9] (13.center) to[out=135, in=-44.9] (17.center) to[out=135, in=-44.9] (20.center) to[out=135, in=  0] (25.center);
\draw[] (25.center) to[out=180, in=  0] (26.center);
\draw[] (14.center) to[out=180, in=-44.9] (18.center) to[out=135, in=-44.9] (21.center) to[out=135, in=-44.9] (27.center) to[out=135, in=  0] (28.center);
\draw[] (16.center) to[out=180, in= 45] (20.center) to[out=-135, in= 45] (27.center) to[out=-135, in=  0] (29.center);
\draw[] (12.center) to[out=180, in= 45] (17.center) to[out=-135, in= 45] (21.center) to[out=-135, in=  0] (30.center);
\draw[] (30.center) to[out=180, in=  0] (31.center);
\draw[] (22.center) to[out=180, in=  0] (32.center);
\draw[] (32.center) to[out=180, in=  0] (33.center);
\end{pgfonlayer}
\end{tikzpicture}

\end{center}

It is then easy to see that what we obtain is the prop $Fun$ 
where $Fun[n,m]$ is the set of  all functions from $[1,\dots, n]$ to $[1,\dots, m]$, with $\circ$ as function composition and $\otimes$ is disjoint union.
The picture above represents in particular the function from $[1,2,3,4,5,6]$ to $[1,2,3,4,5]$ given by $1\mapsto 2\ \  2\mapsto 5\ \  3\mapsto 5\ \  4 \mapsto 2\ \  5\mapsto 2\ \  6\mapsto 3$.

\end{example}

Now, to prefigure what we will be doing later on, remark that the equations satisfied by $\mu$ and  $\eta$ are exactly the equations of a commutative monoid. This means that the diagrams of the prop $Fun$ can be interpreted (i.e., there is a functor) in any prop that contains two morphisms satisfying the monoid equations (this is of course by the universal property of a prop given by generators and relations).

As an example, remember the prop $Fun_X$, the prop where the morphisms $n \mapsto m$ are exactly the functions from $X^n$ to $ X^m$, in the particular case $X = \mathbb{R}$.
This prop contains a morphism $\mu: \mathbb{R}^2 \to \mathbb{R}$ defined by $\mu(x,y) = x+y$ and a morphism $\eta: \mathbb{R}^0 \to \mathbb{R}$ (a constant) define by $\eta = 0$ and these two morphism satisfy the monoid equations.
This means every diagram in the prop $Fun$ can be interpreted in the prop $Fun_X$. The previous diagram will correspond to the map from $\mathbb{R}^6 \to \mathbb{R}^5$ defined by
$(x_1, x_2, x_3, x_4, x_5, x_6) \mapsto (0,x_1+x_4+x_5, x_6, 0, x_2+x_3)$

In the case where we start from a prop with a complicated set of generators of equations, we can use this method to transform diagrams in a prop we do not understand into diagrams in a prop we do understand. This might be used for example to prove that two different diagrams are nonequal (interpret the diagrams in another prop in which it is clear that they are nonequal) and will be the focus of the last section.

The most important example of a finitely presented prop is the following one, on which we will base all works in this paper:
\begin{definition}
\label{def:matricesintegers}  The prop MAT is the prop with the following presentation:
  \begin{itemize}
  \item 4 generators $\mu: 2 \rightarrow 1$, $\eta: 0 \rightarrow 1$, $\Delta: 1 \rightarrow 2$, $\epsilon: 1 \rightarrow 0$ and 
  \item 10 equations:
\[\begin{array}{lll}    
    {\mu\circ (\eta\otimes \mathrm{id})=\mu\circ (\mathrm{id}\otimes \eta)=\mathrm{id}
 } &    {\mu\circ (\mu\otimes \mathrm{id})=\mu\circ (\mathrm{id}\otimes \mu)
 } &     {\mu\circ \sigma=\mu
 }\\
    {(\epsilon\otimes \mathrm{id})\circ \Delta=(\mathrm{id}\otimes \epsilon)\circ \Delta=\mathrm{id}
 } &    {(\Delta\otimes \mathrm{id})\circ \Delta=(\mathrm{id}\otimes \Delta)\circ \Delta
 }&    {\sigma\circ \Delta=\Delta
 }\\
    {\Delta\circ \eta=\eta\otimes \eta
 } &    {\epsilon\circ \mu=\epsilon\otimes \epsilon
 }&    {\epsilon\circ \eta=1
 }\\
  &  {(\mu\otimes \mu)\circ (\mathrm{id}\otimes \sigma\otimes \mathrm{id})\circ (\Delta\otimes \Delta)=\Delta\circ \mu
 } \\
\end{array}
\]
  \end{itemize}
  In pictures:
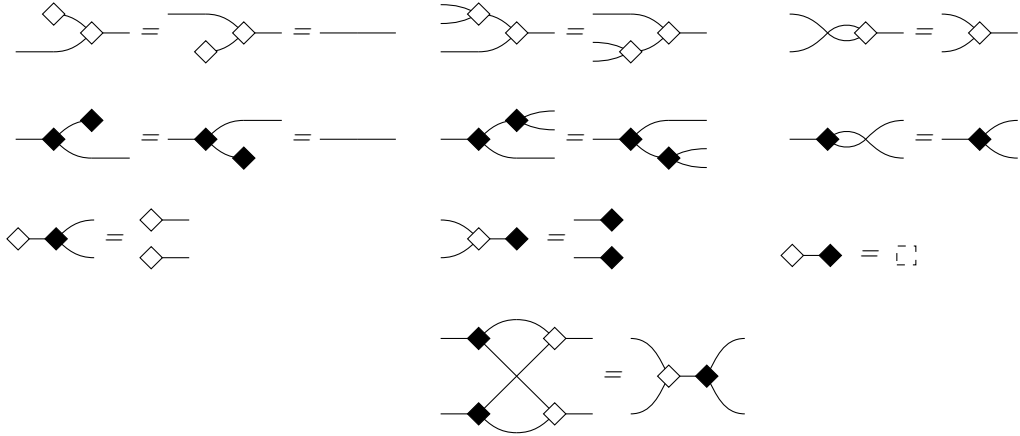
\[\begin{array}{lll}    
    {\begin{tikzpicture}[baseline=(current bounding box.center)]
\begin{pgfonlayer}{nodelayer}
\node [] (0) at (-0.000,-0.500) {};
\node [] (1) at (-0.500,-0.500) {};
\node [] (2) at (-1.000,-0.500) {};
\node [label={[label position=center]=}] (3) at (-1.250,-0.500) {};
\node [] (4) at (-1.500,-0.500) {};
\node [style=monoid] (5) at (-2.000,-0.500) {};
\node [] (6) at (-2.500,-0.250) {};
\node [] (7) at (-3.000,-0.250) {};
\node [style=monoid] (8) at (-2.500,-0.750) {};
\node [label={[label position=center]=}] (9) at (-3.250,-0.500) {};
\node [] (10) at (-3.500,-0.500) {};
\node [style=monoid] (11) at (-4.000,-0.500) {};
\node [style=monoid] (12) at (-4.500,-0.250) {};
\node [] (13) at (-4.500,-0.750) {};
\node [] (14) at (-5.000,-0.750) {};
\end{pgfonlayer}
\begin{pgfonlayer}{edgelayer}
\draw[] (0.center) to[out=180, in=  0] (1.center);
\draw[] (1.center) to[out=180, in=  0] (2.center);
\draw[] (4.center) to[out=180, in=  0] (5.center);
\draw[] (5.center) to[out=135, in=  0] (6.center);
\draw[] (6.center) to[out=180, in=  0] (7.center);
\draw[] (5.center) to[out=-135, in=  0] (8.center);
\draw[] (10.center) to[out=180, in=  0] (11.center);
\draw[] (11.center) to[out=135, in=  0] (12.center);
\draw[] (11.center) to[out=-135, in=  0] (13.center);
\draw[] (13.center) to[out=180, in=  0] (14.center);
\end{pgfonlayer}
\end{tikzpicture}
 } &    {\begin{tikzpicture}[baseline=(current bounding box.center)]
\begin{pgfonlayer}{nodelayer}
\node [] (0) at (-0.000,-0.500) {};
\node [style=monoid] (1) at (-0.500,-0.500) {};
\node [] (2) at (-1.000,-0.250) {};
\node [] (3) at (-1.500,-0.250) {};
\node [style=monoid] (4) at (-1.000,-0.750) {};
\node [] (5) at (-1.500,-0.625) {};
\node [] (6) at (-1.500,-0.875) {};
\node [label={[label position=center]=}] (7) at (-1.750,-0.500) {};
\node [] (8) at (-2.000,-0.500) {};
\node [style=monoid] (9) at (-2.500,-0.500) {};
\node [style=monoid] (10) at (-3.000,-0.250) {};
\node [] (11) at (-3.500,-0.125) {};
\node [] (12) at (-3.500,-0.375) {};
\node [] (13) at (-3.000,-0.750) {};
\node [] (14) at (-3.500,-0.750) {};
\end{pgfonlayer}
\begin{pgfonlayer}{edgelayer}
\draw[] (0.center) to[out=180, in=  0] (1.center);
\draw[] (1.center) to[out=135, in=  0] (2.center);
\draw[] (2.center) to[out=180, in=  0] (3.center);
\draw[] (1.center) to[out=-135, in=  0] (4.center);
\draw[] (4.center) to[out=153, in=  0] (5.center);
\draw[] (4.center) to[out=-153, in=  0] (6.center);
\draw[] (8.center) to[out=180, in=  0] (9.center);
\draw[] (9.center) to[out=135, in=  0] (10.center);
\draw[] (10.center) to[out=153, in=  0] (11.center);
\draw[] (10.center) to[out=-153, in=  0] (12.center);
\draw[] (9.center) to[out=-135, in=  0] (13.center);
\draw[] (13.center) to[out=180, in=  0] (14.center);
\end{pgfonlayer}
\end{tikzpicture}
 } &     {\begin{tikzpicture}[baseline=(current bounding box.center)]
\begin{pgfonlayer}{nodelayer}
\node [] (0) at (-0.000,-0.500) {};
\node [style=monoid] (1) at (-0.500,-0.500) {};
\node [] (2) at (-1.000,-0.250) {};
\node [] (3) at (-1.000,-0.750) {};
\node [label={[label position=center]=}] (4) at (-1.250,-0.500) {};
\node [] (5) at (-1.500,-0.500) {};
\node [style=monoid] (6) at (-2.000,-0.500) {};
\node [] (7) at (-2.500,-0.500) {};
\node [] (8) at (-3.000,-0.250) {};
\node [] (9) at (-3.000,-0.750) {};
\end{pgfonlayer}
\begin{pgfonlayer}{edgelayer}
\draw[] (0.center) to[out=180, in=  0] (1.center);
\draw[] (1.center) to[out=135, in=  0] (2.center);
\draw[] (1.center) to[out=-135, in=  0] (3.center);
\draw[] (5.center) to[out=180, in=  0] (6.center);
\draw[] (6.center) to[out=-135, in=-44.9] (7.center) to[out=135, in=  0] (8.center);
\draw[] (6.center) to[out=135, in= 45] (7.center) to[out=-135, in=  0] (9.center);
\end{pgfonlayer}
\end{tikzpicture}
 }\\
    \\
      {\begin{tikzpicture}[baseline=(current bounding box.center)]
\begin{pgfonlayer}{nodelayer}
\node [] (0) at (-0.000,-0.500) {};
\node [] (1) at (-0.500,-0.500) {};
\node [] (2) at (-1.000,-0.500) {};
\node [label={[label position=center]=}] (3) at (-1.250,-0.500) {};
\node [] (4) at (-1.500,-0.250) {};
\node [] (5) at (-2.000,-0.250) {};
\node [style=bmonoid] (6) at (-2.000,-0.750) {};
\node [style=bmonoid] (7) at (-2.500,-0.500) {};
\node [] (8) at (-3.000,-0.500) {};
\node [label={[label position=center]=}] (9) at (-3.250,-0.500) {};
\node [style=bmonoid] (10) at (-4.000,-0.250) {};
\node [] (11) at (-3.500,-0.750) {};
\node [] (12) at (-4.000,-0.750) {};
\node [style=bmonoid] (13) at (-4.500,-0.500) {};
\node [] (14) at (-5.000,-0.500) {};
\end{pgfonlayer}
\begin{pgfonlayer}{edgelayer}
\draw[] (0.center) to[out=180, in=  0] (1.center);
\draw[] (1.center) to[out=180, in=  0] (2.center);
\draw[] (4.center) to[out=180, in=  0] (5.center);
\draw[] (5.center) to[out=180, in= 45] (7.center);
\draw[] (6.center) to[out=180, in=-44.9] (7.center);
\draw[] (7.center) to[out=180, in=  0] (8.center);
\draw[] (11.center) to[out=180, in=  0] (12.center);
\draw[] (10.center) to[out=180, in= 45] (13.center);
\draw[] (12.center) to[out=180, in=-44.9] (13.center);
\draw[] (13.center) to[out=180, in=  0] (14.center);
\end{pgfonlayer}
\end{tikzpicture}
 } &    {\begin{tikzpicture}[baseline=(current bounding box.center)]
\begin{pgfonlayer}{nodelayer}
\node [] (0) at (-0.000,-0.250) {};
\node [] (1) at (-0.500,-0.250) {};
\node [] (2) at (-0.000,-0.625) {};
\node [] (3) at (-0.000,-0.875) {};
\node [style=bmonoid] (4) at (-0.500,-0.750) {};
\node [style=bmonoid] (5) at (-1.000,-0.500) {};
\node [] (6) at (-1.500,-0.500) {};
\node [label={[label position=center]=}] (7) at (-1.750,-0.500) {};
\node [] (8) at (-2.000,-0.125) {};
\node [] (9) at (-2.000,-0.375) {};
\node [style=bmonoid] (10) at (-2.500,-0.250) {};
\node [] (11) at (-2.000,-0.750) {};
\node [] (12) at (-2.500,-0.750) {};
\node [style=bmonoid] (13) at (-3.000,-0.500) {};
\node [] (14) at (-3.500,-0.500) {};
\end{pgfonlayer}
\begin{pgfonlayer}{edgelayer}
\draw[] (0.center) to[out=180, in=  0] (1.center);
\draw[] (2.center) to[out=180, in=26.6] (4.center);
\draw[] (3.center) to[out=180, in=-26.5] (4.center);
\draw[] (1.center) to[out=180, in= 45] (5.center);
\draw[] (4.center) to[out=180, in=-44.9] (5.center);
\draw[] (5.center) to[out=180, in=  0] (6.center);
\draw[] (8.center) to[out=180, in=26.6] (10.center);
\draw[] (9.center) to[out=180, in=-26.5] (10.center);
\draw[] (11.center) to[out=180, in=  0] (12.center);
\draw[] (10.center) to[out=180, in= 45] (13.center);
\draw[] (12.center) to[out=180, in=-44.9] (13.center);
\draw[] (13.center) to[out=180, in=  0] (14.center);
\end{pgfonlayer}
\end{tikzpicture}
 }&    {\begin{tikzpicture}[baseline=(current bounding box.center)]
\begin{pgfonlayer}{nodelayer}
\node [] (0) at (-0.000,-0.250) {};
\node [] (1) at (-0.000,-0.750) {};
\node [style=bmonoid] (2) at (-0.500,-0.500) {};
\node [] (3) at (-1.000,-0.500) {};
\node [label={[label position=center]=}] (4) at (-1.250,-0.500) {};
\node [] (5) at (-1.500,-0.250) {};
\node [] (6) at (-1.500,-0.750) {};
\node [] (7) at (-2.000,-0.500) {};
\node [style=bmonoid] (8) at (-2.500,-0.500) {};
\node [] (9) at (-3.000,-0.500) {};
\end{pgfonlayer}
\begin{pgfonlayer}{edgelayer}
\draw[] (0.center) to[out=180, in= 45] (2.center);
\draw[] (1.center) to[out=180, in=-44.9] (2.center);
\draw[] (2.center) to[out=180, in=  0] (3.center);
\draw[] (6.center) to[out=180, in=-44.9] (7.center) to[out=135, in= 45] (8.center);
\draw[] (5.center) to[out=180, in= 45] (7.center) to[out=-135, in=-44.9] (8.center);
\draw[] (8.center) to[out=180, in=  0] (9.center);
\end{pgfonlayer}
\end{tikzpicture}
 }\\
      \\
        {\begin{tikzpicture}
\begin{pgfonlayer}{nodelayer}
\node [] (0) at (-0.000,-0.250) {};
\node [style=monoid] (1) at (-0.500,-0.250) {};
\node [] (2) at (-0.000,-0.750) {};
\node [style=monoid] (3) at (-0.500,-0.750) {};
\node [label={[label position=center]=}] (4) at (-1.000,-0.500) {};
\node [] (5) at (-1.250,-0.250) {};
\node [] (6) at (-1.250,-0.750) {};
\node [style=bmonoid] (7) at (-1.750,-0.500) {};
\node [style=monoid] (8) at (-2.250,-0.500) {};
\end{pgfonlayer}
\begin{pgfonlayer}{edgelayer}
\draw[] (0.center) to[out=180, in=  0] (1.center);
\draw[] (2.center) to[out=180, in=  0] (3.center);
\draw[] (5.center) to[out=180, in= 45] (7.center);
\draw[] (6.center) to[out=180, in=-44.9] (7.center);
\draw[] (7.center) to[out=180, in=  0] (8.center);
\end{pgfonlayer}
\end{tikzpicture}
 } &    {\begin{tikzpicture}
\begin{pgfonlayer}{nodelayer}
\node [style=bmonoid] (0) at (-0.250,-0.250) {};
\node [] (1) at (-0.750,-0.250) {};
\node [style=bmonoid] (2) at (-0.250,-0.750) {};
\node [] (3) at (-0.750,-0.750) {};
\node [label={[label position=center]=}] (4) at (-1.000,-0.500) {};
\node [style=bmonoid] (5) at (-1.500,-0.500) {};
\node [style=monoid] (6) at (-2.000,-0.500) {};
\node [] (7) at (-2.500,-0.250) {};
\node [] (8) at (-2.500,-0.750) {};
\end{pgfonlayer}
\begin{pgfonlayer}{edgelayer}
\draw[] (0.center) to[out=180, in=  0] (1.center);
\draw[] (2.center) to[out=180, in=  0] (3.center);
\draw[] (5.center) to[out=180, in=  0] (6.center);
\draw[] (6.center) to[out=135, in=  0] (7.center);
\draw[] (6.center) to[out=-135, in=  0] (8.center);
\end{pgfonlayer}
\end{tikzpicture}
 }&    {\begin{tikzpicture}
\begin{pgfonlayer}{nodelayer}
\node [shape=rectangle,fill=white,draw=,dashed=,inner xsep=0.125cm,inner ysep=0.125cm] (0) at (-0.250,-0.250) {};
\node [label={[label position=center]=}] (1) at (-0.750,-0.250) {};
\node [style=bmonoid] (2) at (-1.250,-0.250) {};
\node [style=monoid] (3) at (-1.750,-0.250) {};
\end{pgfonlayer}
\begin{pgfonlayer}{edgelayer}
\draw[] (2.center) to[out=180, in=  0] (3.center);
\end{pgfonlayer}
\end{tikzpicture}
 }\\
        \\
&    {\begin{tikzpicture}
\begin{pgfonlayer}{nodelayer}
\node [] (0) at (-0.000,-0.500) {};
\node [] (1) at (-0.000,-1.500) {};
\node [style=bmonoid] (2) at (-0.500,-1.000) {};
\node [style=monoid] (3) at (-1.000,-1.000) {};
\node [] (4) at (-1.500,-0.500) {};
\node [] (5) at (-1.500,-1.500) {};
\node [label={[label position=center]=}] (6) at (-1.750,-1.000) {};
\node [] (7) at (-2.000,-0.500) {};
\node [style=monoid] (8) at (-2.500,-0.500) {};
\node [] (9) at (-2.000,-1.500) {};
\node [style=monoid] (10) at (-2.500,-1.500) {};
\node [] (11) at (-3.000,-0.250) {};
\node [] (12) at (-3.000,-1.000) {};
\node [] (13) at (-3.000,-1.750) {};
\node [style=bmonoid] (14) at (-3.500,-0.500) {};
\node [] (15) at (-4.000,-0.500) {};
\node [style=bmonoid] (16) at (-3.500,-1.500) {};
\node [] (17) at (-4.000,-1.500) {};
\end{pgfonlayer}
\begin{pgfonlayer}{edgelayer}
\draw[] (0.center) to[out=180, in=63.4] (2.center);
\draw[] (1.center) to[out=180, in=-63.3] (2.center);
\draw[] (2.center) to[out=180, in=  0] (3.center);
\draw[] (3.center) to[out=117, in=  0] (4.center);
\draw[] (3.center) to[out=-116, in=  0] (5.center);
\draw[] (7.center) to[out=180, in=  0] (8.center);
\draw[] (9.center) to[out=180, in=  0] (10.center);
\draw[] (8.center) to[out=135, in=  0] (11.center);
\draw[] (10.center) to[out=-135, in=  0] (13.center);
\draw[] (11.center) to[out=180, in= 45] (14.center);
\draw[] (10.center) to[out=135, in=-44.9] (12.center) to[out=135, in=-44.9] (14.center);
\draw[] (14.center) to[out=180, in=  0] (15.center);
\draw[] (8.center) to[out=-135, in= 45] (12.center) to[out=-135, in= 45] (16.center);
\draw[] (13.center) to[out=180, in=-44.9] (16.center);
\draw[] (16.center) to[out=180, in=  0] (17.center);
\end{pgfonlayer}
\end{tikzpicture}
 } \\
\end{array}
\]

\end{definition}
The first three equations are the monoid equations we saw before, the next three equations are comonoid equations (dual of monoid equations). The last four equations are called the bigebra/bialgebra/bimonoid equations.

With the first equations, we may define as previously white diamonds of type $k\to 1$ and black diamonds of type $1 \to k$.
The last four equations are essentially commuting equations, saying that when we want to do a black and a white operation, the order of the operation doesn't matter.

With these ten equations, it is easy to see that any diagram can be put in the following form: first white diamonds, then a permutation, then black diamonds, as in the following example:

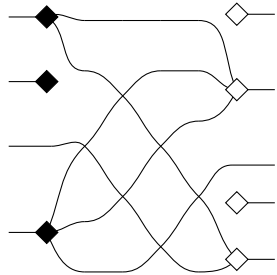
\begin{center}
\begin{tikzpicture}[baseline=(current bounding box.center)]
\begin{pgfonlayer}{nodelayer}
\node [] (0) at (-0.000,-0.250) {};
\node [style=monoid] (1) at (-0.500,-0.250) {};
\node [] (2) at (-0.000,-1.250) {};
\node [style=monoid] (3) at (-0.500,-1.250) {};
\node [] (4) at (-0.000,-2.250) {};
\node [] (5) at (-0.563,-2.250) {};
\node [] (6) at (-0.000,-2.750) {};
\node [style=monoid] (7) at (-0.500,-2.750) {};
\node [] (8) at (-0.000,-3.500) {};
\node [style=monoid] (9) at (-0.500,-3.500) {};
\node [] (10) at (-1.000,-0.333) {};
\node [] (11) at (-1.000,-1.000) {};
\node [] (12) at (-1.000,-1.667) {};
\node [] (13) at (-1.000,-2.667) {};
\node [] (14) at (-1.000,-3.667) {};
\node [] (15) at (-1.500,-0.333) {};
\node [] (16) at (-1.500,-1.000) {};
\node [] (17) at (-1.500,-2.000) {};
\node [] (18) at (-1.500,-3.333) {};
\node [] (19) at (-2.000,-0.333) {};
\node [] (20) at (-2.000,-1.333) {};
\node [] (21) at (-2.000,-2.667) {};
\node [] (22) at (-2.000,-3.667) {};
\node [] (23) at (-2.500,-0.333) {};
\node [] (24) at (-2.500,-1.000) {};
\node [] (25) at (-2.500,-2.000) {};
\node [] (26) at (-2.500,-3.000) {};
\node [] (27) at (-2.500,-3.667) {};
\node [style=bmonoid] (28) at (-3.000,-0.286) {};
\node [] (29) at (-3.500,-0.286) {};
\node [style=bmonoid] (30) at (-3.000,-1.143) {};
\node [] (31) at (-3.500,-1.143) {};
\node [] (32) at (-2.938,-2.000) {};
\node [] (33) at (-3.500,-2.000) {};
\node [style=bmonoid] (34) at (-3.000,-3.143) {};
\node [] (35) at (-3.500,-3.143) {};
\end{pgfonlayer}
\begin{pgfonlayer}{edgelayer}
\draw[] (0.center) to[out=180, in=  0] (1.center);
\draw[] (2.center) to[out=180, in=  0] (3.center);
\draw[] (4.center) to[out=180, in=  0] (5.center);
\draw[] (6.center) to[out=180, in=  0] (7.center);
\draw[] (8.center) to[out=180, in=  0] (9.center);
\draw[] (3.center) to[out=117, in=  0] (10.center);
\draw[] (3.center) to[out=180, in=  0] (11.center);
\draw[] (3.center) to[out=-116, in=  0] (12.center);
\draw[] (9.center) to[out=-135, in=  0] (14.center);
\draw[] (10.center) to[out=180, in=  0] (15.center);
\draw[] (11.center) to[out=180, in=  0] (16.center);
\draw[] (15.center) to[out=180, in=  0] (19.center);
\draw[] (5.center) to[out=180, in= 45] (13.center) to[out=-135, in= 45] (18.center) to[out=-135, in=  0] (22.center);
\draw[] (19.center) to[out=180, in=  0] (23.center);
\draw[] (9.center) to[out=135, in=-44.9] (13.center) to[out=135, in=-44.9] (17.center) to[out=135, in=-44.9] (20.center) to[out=135, in=  0] (24.center);
\draw[] (12.center) to[out=180, in= 45] (17.center) to[out=-135, in= 45] (21.center) to[out=-135, in=  0] (26.center);
\draw[] (22.center) to[out=180, in=  0] (27.center);
\draw[] (23.center) to[out=180, in=29.7] (28.center);
\draw[] (24.center) to[out=180, in=-29.6] (28.center);
\draw[] (28.center) to[out=180, in=  0] (29.center);
\draw[] (30.center) to[out=180, in=  0] (31.center);
\draw[] (14.center) to[out=180, in=-44.9] (18.center) to[out=135, in=-44.9] (21.center) to[out=135, in=-44.9] (25.center) to[out=135, in=  0] (32.center);
\draw[] (32.center) to[out=180, in=  0] (33.center);
\draw[] (16.center) to[out=180, in= 45] (20.center) to[out=-135, in= 45] (25.center) to[out=-135, in=66.4] (34.center);
\draw[] (26.center) to[out=180, in=  0] (34.center);
\draw[] (27.center) to[out=180, in=-66.3] (34.center);
\draw[] (34.center) to[out=180, in=  0] (35.center);
\end{pgfonlayer}
\end{tikzpicture}

\end{center}
\newcommand\matrices[1]{M_{\star,\star}(#1)}

One can then prove, using the bialgebra rules, the following result:
\begin{proposition}
  The prop MAT is (isomorphic to) the prop $\matrices{\mathbb{Z}_+}$, the prop of matrices with nonnegative coefficients.
  Specifically $MAT[n,m] = M_{m,n}(\mathbb{Z}_+)$ is the set of matrices over $\mathbb{Z}_+$ of size $m \times n$, with $\circ$ the product of matrices and $\otimes$ the direct sum : $A \otimes B = \begin{pmatrix} A & 0 \\ 0 & B \end{pmatrix}$  
\end{proposition}
This proposition was probably proven first in \cite{pirash}, but is somewhat considered folklore.
The preceding example corresponds to the matrix
$\begin{pmatrix}
  0 & 1 & 0 & 0 & 1 \\
  0 & 0 & 0 & 0 & 0 \\
  0 & 0 & 0 & 0 & 1 \\
  0 & 2 & 1 & 0 & 0 \\
\end{pmatrix}  $.

This correspondence makes it easy to represent linear algebra equations by diagrams. The interested reader is urged to read the thesis by Zanasi \cite{Zanasi}, and the  \href{https://graphicallinearalgebra.net/}{blog of Pawel Sobocinski} on this particular subject.

There is another way to see this prop: morphisms are \emph{bipartite graphs} (with inputs on the left, and outputs on the right), composition plugs the inputs of the first graph into  the outputs of the second graph, deleting the  internal vertices by bursting them into edges.

This structure of monoid/comonoid with a bigebra law can be found in many domains of mathematics. One particularly important example is in $Fun_X$: one can take for $\mu$ and $\eta$ a monoid on $X$, and for $\Delta/\epsilon$ the \emph{copy comonoid}: $\Delta(x) = (x,x)$ (so the black diamond copies its input) and $\epsilon$ is the only function with codomain $1$ ($\epsilon$ forgets its input). The bigebra law then states the following: Supposons you want to do the product (for the monoid) of two inputs $x,y$ then duplicate the result. Then you can alternatively duplicate the outputs (to obtain two copies of $x$, and two copies of $y$), then do the initial product twice.

The reader fluent in prop and category theory will probably expect the article to proceed to \emph{Hopf algebras}, the canonical way to be able to matrices in $\mathbb{Z}$. 
This article goes however in a different direction.
First, we will deal with matrices in $\mathbb{Z}_+[t]$. For this we need to add a generator to code our new coefficient $t$.
This can be done in the following way:

\begin{proposition}
\label{def:matricespolynomial}  Consider the prop that contains the same generators and equations as before, with an additional generator $h: 1 \rightarrow 1$ that satifies the following equations:  
\[\begin{array}{lll}    
    {h\circ \mu=\mu\circ (h\otimes h)
 } &    {h\circ \eta=\eta
 }\\
    {(h\otimes h)\circ \Delta=\Delta\circ h
 } &    {\epsilon\circ h=\epsilon
 } \\
\end{array}
\]
In pictures:
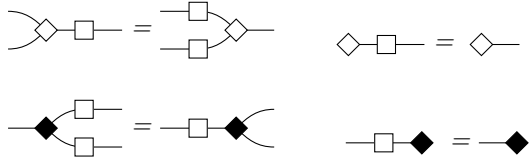
\[\begin{array}{lll}    
    {\begin{tikzpicture}
\begin{pgfonlayer}{nodelayer}
\node [] (0) at (-0.000,-0.500) {};
\node [style=monoid] (1) at (-0.500,-0.500) {};
\node [shape=rectangle,draw,fill=white] (2) at (-1.000,-0.250) {};
\node [] (3) at (-1.500,-0.250) {};
\node [shape=rectangle,draw,fill=white] (4) at (-1.000,-0.750) {};
\node [] (5) at (-1.500,-0.750) {};
\node [label={[label position=center]=}] (6) at (-1.750,-0.500) {};
\node [] (7) at (-2.000,-0.500) {};
\node [shape=rectangle,draw,fill=white] (8) at (-2.500,-0.500) {};
\node [style=monoid] (9) at (-3.000,-0.500) {};
\node [] (10) at (-3.500,-0.250) {};
\node [] (11) at (-3.500,-0.750) {};
\end{pgfonlayer}
\begin{pgfonlayer}{edgelayer}
\draw[] (0.center) to[out=180, in=  0] (1.center);
\draw[] (1.center) to[out=135, in=  0] (2.center);
\draw[] (2.center) to[out=180, in=  0] (3.center);
\draw[] (1.center) to[out=-135, in=  0] (4.center);
\draw[] (4.center) to[out=180, in=  0] (5.center);
\draw[] (7.center) to[out=180, in=  0] (8.center);
\draw[] (8.center) to[out=180, in=  0] (9.center);
\draw[] (9.center) to[out=135, in=  0] (10.center);
\draw[] (9.center) to[out=-135, in=  0] (11.center);
\end{pgfonlayer}
\end{tikzpicture}
 } & &   {\begin{tikzpicture}
\begin{pgfonlayer}{nodelayer}
\node [] (0) at (-0.000,-0.250) {};
\node [style=monoid] (1) at (-0.500,-0.250) {};
\node [label={[label position=center]=}] (2) at (-1.000,-0.250) {};
\node [] (3) at (-1.250,-0.250) {};
\node [shape=rectangle,draw,fill=white] (4) at (-1.750,-0.250) {};
\node [style=monoid] (5) at (-2.250,-0.250) {};
\end{pgfonlayer}
\begin{pgfonlayer}{edgelayer}
\draw[] (0.center) to[out=180, in=  0] (1.center);
\draw[] (3.center) to[out=180, in=  0] (4.center);
\draw[] (4.center) to[out=180, in=  0] (5.center);
\end{pgfonlayer}
\end{tikzpicture}
 }\\
    \\
      {\begin{tikzpicture}
\begin{pgfonlayer}{nodelayer}
\node [] (0) at (-0.000,-0.250) {};
\node [] (1) at (-0.000,-0.750) {};
\node [style=bmonoid] (2) at (-0.500,-0.500) {};
\node [shape=rectangle,draw,fill=white] (3) at (-1.000,-0.500) {};
\node [] (4) at (-1.500,-0.500) {};
\node [label={[label position=center]=}] (5) at (-1.750,-0.500) {};
\node [] (6) at (-2.000,-0.250) {};
\node [shape=rectangle,draw,fill=white] (7) at (-2.500,-0.250) {};
\node [] (8) at (-2.000,-0.750) {};
\node [shape=rectangle,draw,fill=white] (9) at (-2.500,-0.750) {};
\node [style=bmonoid] (10) at (-3.000,-0.500) {};
\node [] (11) at (-3.500,-0.500) {};
\end{pgfonlayer}
\begin{pgfonlayer}{edgelayer}
\draw[] (0.center) to[out=180, in= 45] (2.center);
\draw[] (1.center) to[out=180, in=-44.9] (2.center);
\draw[] (2.center) to[out=180, in=  0] (3.center);
\draw[] (3.center) to[out=180, in=  0] (4.center);
\draw[] (6.center) to[out=180, in=  0] (7.center);
\draw[] (8.center) to[out=180, in=  0] (9.center);
\draw[] (7.center) to[out=180, in= 45] (10.center);
\draw[] (9.center) to[out=180, in=-44.9] (10.center);
\draw[] (10.center) to[out=180, in=  0] (11.center);
\end{pgfonlayer}
\end{tikzpicture}
 } & &   {\begin{tikzpicture}
\begin{pgfonlayer}{nodelayer}
\node [style=bmonoid] (0) at (-0.250,-0.250) {};
\node [] (1) at (-0.750,-0.250) {};
\node [label={[label position=center]=}] (2) at (-1.000,-0.250) {};
\node [style=bmonoid] (3) at (-1.500,-0.250) {};
\node [shape=rectangle,draw,fill=white] (4) at (-2.000,-0.250) {};
\node [] (5) at (-2.500,-0.250) {};
\end{pgfonlayer}
\begin{pgfonlayer}{edgelayer}
\draw[] (0.center) to[out=180, in=  0] (1.center);
\draw[] (3.center) to[out=180, in=  0] (4.center);
\draw[] (4.center) to[out=180, in=  0] (5.center);
\end{pgfonlayer}
\end{tikzpicture}
 } \\
\end{array}
\]
This prop is exactly the prop $\matrices{\mathbb{Z}_+[t]}$, the prop of matrices with coefficients in $\mathbb{Z}_+[t]$
\end{proposition}
As an example, the following diagram:
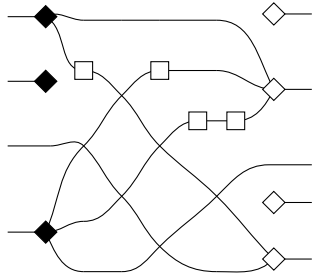
\begin{center}
\begin{tikzpicture}[baseline=(current bounding box.center)]
\begin{pgfonlayer}{nodelayer}
\node [] (0) at (-0.000,-0.250) {};
\node [style=monoid] (1) at (-0.500,-0.250) {};
\node [] (2) at (-0.000,-1.250) {};
\node [style=monoid] (3) at (-0.500,-1.250) {};
\node [] (4) at (-0.000,-2.250) {};
\node [] (5) at (-0.563,-2.250) {};
\node [] (6) at (-0.000,-2.750) {};
\node [style=monoid] (7) at (-0.500,-2.750) {};
\node [] (8) at (-0.000,-3.500) {};
\node [style=monoid] (9) at (-0.500,-3.500) {};
\node [] (10) at (-1.250,-0.333) {};
\node [] (11) at (-1.250,-1.000) {};
\node [shape=rectangle,draw,fill=white] (12) at (-1.000,-1.667) {};
\node [shape=rectangle,draw,fill=white] (13) at (-1.500,-1.667) {};
\node [] (14) at (-1.250,-2.667) {};
\node [] (15) at (-1.250,-3.667) {};
\node [] (16) at (-2.000,-0.333) {};
\node [shape=rectangle,draw,fill=white] (17) at (-2.000,-1.000) {};
\node [] (18) at (-2.000,-2.000) {};
\node [] (19) at (-2.000,-3.333) {};
\node [] (20) at (-2.500,-0.333) {};
\node [] (21) at (-2.500,-1.333) {};
\node [] (22) at (-2.500,-2.667) {};
\node [] (23) at (-2.500,-3.667) {};
\node [] (24) at (-3.000,-0.333) {};
\node [shape=rectangle,draw,fill=white] (25) at (-3.000,-1.000) {};
\node [] (26) at (-3.000,-2.000) {};
\node [] (27) at (-3.000,-3.000) {};
\node [] (28) at (-3.000,-3.667) {};
\node [style=bmonoid] (29) at (-3.500,-0.286) {};
\node [] (30) at (-4.000,-0.286) {};
\node [style=bmonoid] (31) at (-3.500,-1.143) {};
\node [] (32) at (-4.000,-1.143) {};
\node [] (33) at (-3.438,-2.000) {};
\node [] (34) at (-4.000,-2.000) {};
\node [style=bmonoid] (35) at (-3.500,-3.143) {};
\node [] (36) at (-4.000,-3.143) {};
\end{pgfonlayer}
\begin{pgfonlayer}{edgelayer}
\draw[] (0.center) to[out=180, in=  0] (1.center);
\draw[] (2.center) to[out=180, in=  0] (3.center);
\draw[] (4.center) to[out=180, in=  0] (5.center);
\draw[] (6.center) to[out=180, in=  0] (7.center);
\draw[] (8.center) to[out=180, in=  0] (9.center);
\draw[] (3.center) to[out=117, in=  0] (10.center);
\draw[] (3.center) to[out=180, in=  0] (11.center);
\draw[] (3.center) to[out=-116, in=  0] (12.center);
\draw[] (12.center) to[out=180, in=  0] (13.center);
\draw[] (9.center) to[out=-135, in=  0] (15.center);
\draw[] (10.center) to[out=180, in=  0] (16.center);
\draw[] (11.center) to[out=180, in=  0] (17.center);
\draw[] (16.center) to[out=180, in=  0] (20.center);
\draw[] (5.center) to[out=180, in= 45] (14.center) to[out=-135, in= 45] (19.center) to[out=-135, in=  0] (23.center);
\draw[] (20.center) to[out=180, in=  0] (24.center);
\draw[] (9.center) to[out=135, in=-44.9] (14.center) to[out=135, in=-44.9] (18.center) to[out=135, in=-44.9] (21.center) to[out=135, in=  0] (25.center);
\draw[] (13.center) to[out=180, in= 45] (18.center) to[out=-135, in= 45] (22.center) to[out=-135, in=  0] (27.center);
\draw[] (23.center) to[out=180, in=  0] (28.center);
\draw[] (24.center) to[out=180, in=29.7] (29.center);
\draw[] (25.center) to[out=180, in=-29.6] (29.center);
\draw[] (29.center) to[out=180, in=  0] (30.center);
\draw[] (31.center) to[out=180, in=  0] (32.center);
\draw[] (15.center) to[out=180, in=-44.9] (19.center) to[out=135, in=-44.9] (22.center) to[out=135, in=-44.9] (26.center) to[out=135, in=  0] (33.center);
\draw[] (33.center) to[out=180, in=  0] (34.center);
\draw[] (17.center) to[out=180, in= 45] (21.center) to[out=-135, in= 45] (26.center) to[out=-135, in=66.4] (35.center);
\draw[] (27.center) to[out=180, in=  0] (35.center);
\draw[] (28.center) to[out=180, in=-66.3] (35.center);
\draw[] (35.center) to[out=180, in=  0] (36.center);
\end{pgfonlayer}
\end{tikzpicture}

\end{center}
correspond to the matrix
$\begin{pmatrix}
  0 & 1 & 0 & 0 & t \\
  0 & 0 & 0 & 0 & 0 \\
  0 & 0 & 0 & 0 & 1 \\
  0 & t^2 + t & 1 & 0 & 0 \\
\end{pmatrix}  $.
We do not prove this result, as the proof is essentially similar to the previous proposition.

We finish this section by a few examples and questions that are actually not used in the following, but might interest the reader coming from symbolic dynamics and/or computability theory

\begin{example}
  Consider the prop with two generators linked by the following equations:
  \begin{center}
{\begin{tikzpicture}[baseline=(current bounding box.center)]
\begin{pgfonlayer}{nodelayer}
\node [] (0) at (-0.000,-0.250) {};
\node [] (1) at (-0.375,-0.250) {};
\node [] (2) at (-0.750,-0.250) {};
\node [] (3) at (-0.000,-0.750) {};
\node [] (4) at (-0.375,-0.750) {};
\node [] (5) at (-0.750,-0.750) {};
\node [label={[label position=center]=}] (6) at (-1.000,-0.500) {};
\node [] (7) at (-1.250,-0.250) {};
\node [] (8) at (-1.250,-0.750) {};
\node [style=monoid] (9) at (-1.750,-0.500) {};
\node [style=monoid] (10) at (-2.250,-0.500) {};
\node [] (11) at (-2.750,-0.250) {};
\node [] (12) at (-2.750,-0.750) {};
\end{pgfonlayer}
\begin{pgfonlayer}{edgelayer}
\draw[] (0.center) to[out=180, in=  0] (1.center);
\draw[] (1.center) to[out=180, in=  0] (2.center);
\draw[] (3.center) to[out=180, in=  0] (4.center);
\draw[] (4.center) to[out=180, in=  0] (5.center);
\draw[] (7.center) to[out=180, in= 45] (9.center);
\draw[] (8.center) to[out=180, in=-44.9] (9.center);
\draw[] (9.center) to[out=180, in=  0] (10.center);
\draw[] (10.center) to[out=135, in=  0] (11.center);
\draw[] (10.center) to[out=-135, in=  0] (12.center);
\end{pgfonlayer}
\end{tikzpicture}
 }\hspace{5cm}   {\begin{tikzpicture}[baseline=(current bounding box.center)]
\begin{pgfonlayer}{nodelayer}
\node [] (0) at (-0.000,-0.250) {};
\node [] (1) at (-0.375,-0.250) {};
\node [] (2) at (-0.750,-0.250) {};
\node [label={[label position=center]=}] (3) at (-1.000,-0.250) {};
\node [] (4) at (-1.250,-0.250) {};
\node [style=monoid] (5) at (-1.750,-0.250) {};
\node [style=monoid] (6) at (-2.250,-0.250) {};
\node [] (7) at (-2.750,-0.250) {};
\end{pgfonlayer}
\begin{pgfonlayer}{edgelayer}
\draw[] (0.center) to[out=180, in=  0] (1.center);
\draw[] (1.center) to[out=180, in=  0] (2.center);
\draw[] (4.center) to[out=180, in=  0] (5.center);
\draw[] (5.center) to[out=153, in=26.6] (6.center);
\draw[] (5.center) to[out=-153, in=-26.5] (6.center);
\draw[] (6.center) to[out=180, in=  0] (7.center);
\end{pgfonlayer}
\end{tikzpicture}
 } 
\end{center}
  Then one can prove easily that the morphisms $1\rightarrow 1$ of this prop corresponds exactly to Thompson's group $V$ \cite{CannonFloydParry}. This will be immediate for the reader familiar with this family of groups as the pictures we obtained here are very similar to the diagrams used to express elements of the group, see \cite{Belk}.

  There is also a notion of a pro, which is a prop without the symmetry $\sigma$. If we delete the symmetry $\sigma$ from the definition, then one can see that the morphisms $1 \rightarrow 1$ of this pro now corresponds exactly to Thompson's group $F$.
\end{example}  

\begin{example}
  Consider the prop with two generators linked by the following equations:
  \hfill    
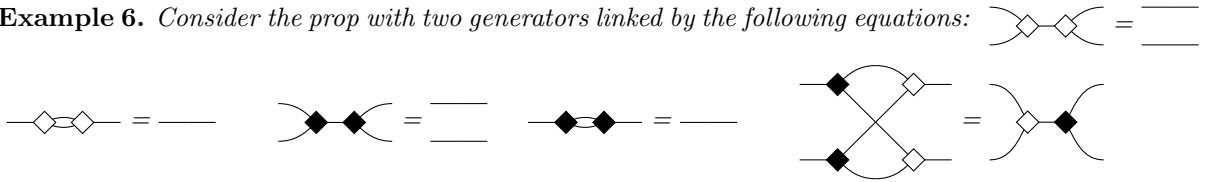
{\begin{tikzpicture}[baseline=(current bounding box.center)]
\begin{pgfonlayer}{nodelayer}
\node [] (0) at (-0.000,-0.250) {};
\node [] (1) at (-0.375,-0.250) {};
\node [] (2) at (-0.750,-0.250) {};
\node [] (3) at (-0.000,-0.750) {};
\node [] (4) at (-0.375,-0.750) {};
\node [] (5) at (-0.750,-0.750) {};
\node [label={[label position=center]=}] (6) at (-1.000,-0.500) {};
\node [] (7) at (-1.250,-0.250) {};
\node [] (8) at (-1.250,-0.750) {};
\node [style=monoid] (9) at (-1.750,-0.500) {};
\node [style=monoid] (10) at (-2.250,-0.500) {};
\node [] (11) at (-2.750,-0.250) {};
\node [] (12) at (-2.750,-0.750) {};
\end{pgfonlayer}
\begin{pgfonlayer}{edgelayer}
\draw[] (0.center) to[out=180, in=  0] (1.center);
\draw[] (1.center) to[out=180, in=  0] (2.center);
\draw[] (3.center) to[out=180, in=  0] (4.center);
\draw[] (4.center) to[out=180, in=  0] (5.center);
\draw[] (7.center) to[out=180, in= 45] (9.center);
\draw[] (8.center) to[out=180, in=-44.9] (9.center);
\draw[] (9.center) to[out=180, in=  0] (10.center);
\draw[] (10.center) to[out=135, in=  0] (11.center);
\draw[] (10.center) to[out=-135, in=  0] (12.center);
\end{pgfonlayer}
\end{tikzpicture}
 }\hfill   {\begin{tikzpicture}[baseline=(current bounding box.center)]
\begin{pgfonlayer}{nodelayer}
\node [] (0) at (-0.000,-0.250) {};
\node [] (1) at (-0.375,-0.250) {};
\node [] (2) at (-0.750,-0.250) {};
\node [label={[label position=center]=}] (3) at (-1.000,-0.250) {};
\node [] (4) at (-1.250,-0.250) {};
\node [style=monoid] (5) at (-1.750,-0.250) {};
\node [style=monoid] (6) at (-2.250,-0.250) {};
\node [] (7) at (-2.750,-0.250) {};
\end{pgfonlayer}
\begin{pgfonlayer}{edgelayer}
\draw[] (0.center) to[out=180, in=  0] (1.center);
\draw[] (1.center) to[out=180, in=  0] (2.center);
\draw[] (4.center) to[out=180, in=  0] (5.center);
\draw[] (5.center) to[out=153, in=26.6] (6.center);
\draw[] (5.center) to[out=-153, in=-26.5] (6.center);
\draw[] (6.center) to[out=180, in=  0] (7.center);
\end{pgfonlayer}
\end{tikzpicture}
 } \hfill
  \hfill    {\begin{tikzpicture}[baseline=(current bounding box.center)]
\begin{pgfonlayer}{nodelayer}
\node [] (0) at (-0.000,-0.250) {};
\node [] (1) at (-0.375,-0.250) {};
\node [] (2) at (-0.750,-0.250) {};
\node [] (3) at (-0.000,-0.750) {};
\node [] (4) at (-0.375,-0.750) {};
\node [] (5) at (-0.750,-0.750) {};
\node [label={[label position=center]=}] (6) at (-1.000,-0.500) {};
\node [] (7) at (-1.250,-0.250) {};
\node [] (8) at (-1.250,-0.750) {};
\node [style=bmonoid] (9) at (-1.750,-0.500) {};
\node [style=bmonoid] (10) at (-2.250,-0.500) {};
\node [] (11) at (-2.750,-0.250) {};
\node [] (12) at (-2.750,-0.750) {};
\end{pgfonlayer}
\begin{pgfonlayer}{edgelayer}
\draw[] (0.center) to[out=180, in=  0] (1.center);
\draw[] (1.center) to[out=180, in=  0] (2.center);
\draw[] (3.center) to[out=180, in=  0] (4.center);
\draw[] (4.center) to[out=180, in=  0] (5.center);
\draw[] (7.center) to[out=180, in= 45] (9.center);
\draw[] (8.center) to[out=180, in=-44.9] (9.center);
\draw[] (9.center) to[out=180, in=  0] (10.center);
\draw[] (10.center) to[out=135, in=  0] (11.center);
\draw[] (10.center) to[out=-135, in=  0] (12.center);
\end{pgfonlayer}
\end{tikzpicture}
 }\hfill   {\begin{tikzpicture}[baseline=(current bounding box.center)]
\begin{pgfonlayer}{nodelayer}
\node [] (0) at (-0.000,-0.250) {};
\node [] (1) at (-0.375,-0.250) {};
\node [] (2) at (-0.750,-0.250) {};
\node [label={[label position=center]=}] (3) at (-1.000,-0.250) {};
\node [] (4) at (-1.250,-0.250) {};
\node [style=bmonoid] (5) at (-1.750,-0.250) {};
\node [style=bmonoid] (6) at (-2.250,-0.250) {};
\node [] (7) at (-2.750,-0.250) {};
\end{pgfonlayer}
\begin{pgfonlayer}{edgelayer}
\draw[] (0.center) to[out=180, in=  0] (1.center);
\draw[] (1.center) to[out=180, in=  0] (2.center);
\draw[] (4.center) to[out=180, in=  0] (5.center);
\draw[] (5.center) to[out=153, in=26.6] (6.center);
\draw[] (5.center) to[out=-153, in=-26.5] (6.center);
\draw[] (6.center) to[out=180, in=  0] (7.center);
\end{pgfonlayer}
\end{tikzpicture}
 } \hfill
  \hfill    {\begin{tikzpicture}[baseline=(current bounding box.center)]
\begin{pgfonlayer}{nodelayer}
\node [] (0) at (-0.000,-0.500) {};
\node [] (1) at (-0.000,-1.500) {};
\node [style=bmonoid] (2) at (-0.500,-1.000) {};
\node [style=monoid] (3) at (-1.000,-1.000) {};
\node [] (4) at (-1.500,-0.500) {};
\node [] (5) at (-1.500,-1.500) {};
\node [label={[label position=center]=}] (6) at (-1.750,-1.000) {};
\node [] (7) at (-2.000,-0.500) {};
\node [style=monoid] (8) at (-2.500,-0.500) {};
\node [] (9) at (-2.000,-1.500) {};
\node [style=monoid] (10) at (-2.500,-1.500) {};
\node [] (11) at (-3.000,-0.250) {};
\node [] (12) at (-3.000,-1.000) {};
\node [] (13) at (-3.000,-1.750) {};
\node [style=bmonoid] (14) at (-3.500,-0.500) {};
\node [] (15) at (-4.000,-0.500) {};
\node [style=bmonoid] (16) at (-3.500,-1.500) {};
\node [] (17) at (-4.000,-1.500) {};
\end{pgfonlayer}
\begin{pgfonlayer}{edgelayer}
\draw[] (0.center) to[out=180, in=63.4] (2.center);
\draw[] (1.center) to[out=180, in=-63.3] (2.center);
\draw[] (2.center) to[out=180, in=  0] (3.center);
\draw[] (3.center) to[out=117, in=  0] (4.center);
\draw[] (3.center) to[out=-116, in=  0] (5.center);
\draw[] (7.center) to[out=180, in=  0] (8.center);
\draw[] (9.center) to[out=180, in=  0] (10.center);
\draw[] (8.center) to[out=135, in=  0] (11.center);
\draw[] (10.center) to[out=-135, in=  0] (13.center);
\draw[] (11.center) to[out=180, in= 45] (14.center);
\draw[] (10.center) to[out=135, in=-44.9] (12.center) to[out=135, in=-44.9] (14.center);
\draw[] (14.center) to[out=180, in=  0] (15.center);
\draw[] (8.center) to[out=-135, in= 45] (12.center) to[out=-135, in= 45] (16.center);
\draw[] (13.center) to[out=180, in=-44.9] (16.center);
\draw[] (16.center) to[out=180, in=  0] (17.center);
\end{pgfonlayer}
\end{tikzpicture}
 } \hfill

  Then one can prove that the morphisms $1\rightarrow 1$ of this prop corresponds exactly to the invertible generalized shift maps of Moore \cite{Moore} over an alphabet of size $2$ (or of Thompson's group $2V$). Equivalently, one can think of morphism $n \rightarrow m$ as an atomic movement of a reversible one-tape Turing machine on an alphabet of size $2$. The $n$ different inputs correspond to the $n$ possible initial state of the Turing machine, and the $m$ output of the state at the next step (usually we take $m = n$).

  The idea is to consider the generators as acting on two binfinite stacks of symbols 0 and 1. The white symbols correspond respectively to the push and pop operations on the first stack.
  The pop operation splits the input into two different futures: the future where the symbol that was popped is a 0 (say, taking the wire on top), and the future where the symbol that was popped is a 1 (taking the wire on the bottom).

  Using two stacks, one can simulate a Turing machine. It is then an exercise to code  the ``shift'' operation by using one pop on one stack and one push on the other stack.
\end{example}

\begin{remark}
  Let $M$ be a finitely presented monoid. Then is it easy to find a prop s.t. the morphisms $1 \rightarrow 1$ of the prop corresponds with the monoid $M$: just put one generator of the prop per generator of the monoid, and one equation in the prop per equation of the monoid.

  Conversely, is it always true that the $1\rightarrow 1$ morphisms of a finitely presented prop form a finitely presented monoid ?
  It is obvious that they form a  recursively presented monoid.

  More generally, consider a prop with a finite number of generators, and a recursive set of equations. Can we always reduce the set of equations to a finite one, at the price of adding some generators and only considering the ``subprop'' generated by the original generators ? This is the equivalent of the Higman embedding theorems for groups, and an open question asked by Lawvere.  
\end{remark}  

\section{Traced props}

\subsection{Introduction and statements of the theorems}
From the time they were given birth, every computer scientist that sees circuits has only one natural urge: to  plug some output back to an input, which would lead to a feedback loop, or to fixed point, depending on the context. From the point of view of category theory, this action is called \emph{tracing}.

We will now consider traced props, which are props with an additional operator(s): the trace. Intuitively, the operator $tr_k$ (trace $k$) is plugging the $k$ first outputs into the $k$ first inputs.
This operator will be represented graphically as follows:

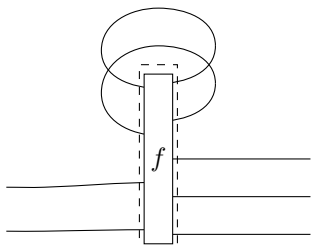
\begin{tikzpicture}[baseline=(current bounding box.center)]
\begin{pgfonlayer}{nodelayer}
\node [] (0) at (-0.750,-1.000) {};
\node [] (1) at (-0.000,-2.250) {};
\node [] (2) at (-0.500,-2.250) {};
\node [] (3) at (-0.000,-2.750) {};
\node [] (4) at (-0.500,-2.750) {};
\node [] (5) at (-0.000,-3.250) {};
\node [] (6) at (-0.500,-3.250) {};
\node [] (7) at (-1.250,-0.750) {};
\node [] (8) at (-1.250,-1.250) {};
\node [] (9) at (-2.750,-0.812) {};
\node [] (10) at (-2.750,-1.375) {};
\path[dashed=None,draw=,color=black] (-1.750,-1.000) rectangle (-2.250,-3.500);
\node at (-2.000,-2.250) {};
\node [] (11) at (-2.000,-0.250) {};
\node [] (12) at (-2.000,-0.750) {};
\node [shape=rectangle,fill=white,draw,label={[label position=center]$f$},inner xsep=0.1875cm,inner ysep=1.125cm] (13) at (-2.000,-2.250) {};
\node [] (14) at (-2.000,-1.312) {};
\node [] (15) at (-2.000,-1.937) {};
\node [] (16) at (-2.000,-2.562) {};
\node [] (17) at (-2.000,-3.188) {};
\node [] (18) at (-2.000,-1.250) {};
\node [] (19) at (-2.000,-1.750) {};
\node [] (20) at (-2.000,-2.250) {};
\node [] (21) at (-2.000,-2.750) {};
\node [] (22) at (-2.000,-3.250) {};
\node [] (23) at (-3.250,-1.167) {};
\node [] (24) at (-3.500,-2.625) {};
\node [] (25) at (-4.000,-2.625) {};
\node [] (26) at (-3.500,-3.208) {};
\node [] (27) at (-4.000,-3.208) {};
\end{pgfonlayer}
\begin{pgfonlayer}{edgelayer}
\draw[] (7.center) to[out=-89.9, in=  0] (18.center);
\draw[] (8.center) to[out=-89.9, in=  0] (19.center);
\draw[] (1.center) to[out=180, in=  0] (2.center) to[out=180, in=  0] (20.center);
\draw[] (3.center) to[out=180, in=  0] (4.center) to[out=180, in=  0] (21.center);
\draw[] (5.center) to[out=180, in=  0] (6.center) to[out=180, in=  0] (22.center);
\draw[] (14.center) to[out=180, in=-89.9] (9.center) to[out= 90, in=180] (11.center) to[out=  0, in= 90] (7.center);
\draw[] (15.center) to[out=180, in=-89.9] (10.center) to[out= 90, in=180] (12.center) to[out=  0, in= 90] (8.center);
\draw[] (16.center) to[out=180, in=  2] (24.center) to[out=-178, in=  0] (25.center);
\draw[] (17.center) to[out=180, in=  1] (26.center) to[out=-179, in=  0] (27.center);
\end{pgfonlayer}
\end{tikzpicture}

The dashed rectangle represents on which morphism the trace is applied. Once the equations of a traced prop are given, it will no longer be essential to represent this dashed rectangle, as all ways of understanding the diagram will give the same morphism:

\begin{definition}[traced prop]
\label{def:tracedprop}  
  
  A traced prop is a prop that contains an operator $tr$ that transforms a morphism $m+1 \rightarrow n+1$ into a morphism $m\rightarrow n$ and that satifies the following five equations, where $tr_p$ represents the $p$-th iterate of the operator (with $tr_0(f) = f)$. 
  \begin{itemize}
  \item Tightening: $tr \left((\mathrm{id}\otimes g)\circ f\circ (\mathrm{id}\otimes h)\right)=g\circ tr \left(f\right)\circ h
    $
  \item Yanking:$tr \left(\sigma\right)=\mathrm{id}
    $
  \item Sliding:$tr_m \left(f\circ (g\otimes id^{\otimes p})\right)=tr_n \left((g\otimes id^{\otimes q})\circ f\right)
    $ for $f: p+m \to q+n$ and $g: n\to m$
 \item Strength:$tr \left(f\otimes g\right)=tr \left(f\right)\otimes g
    $     
  \end{itemize}
  In pictures:
  
  
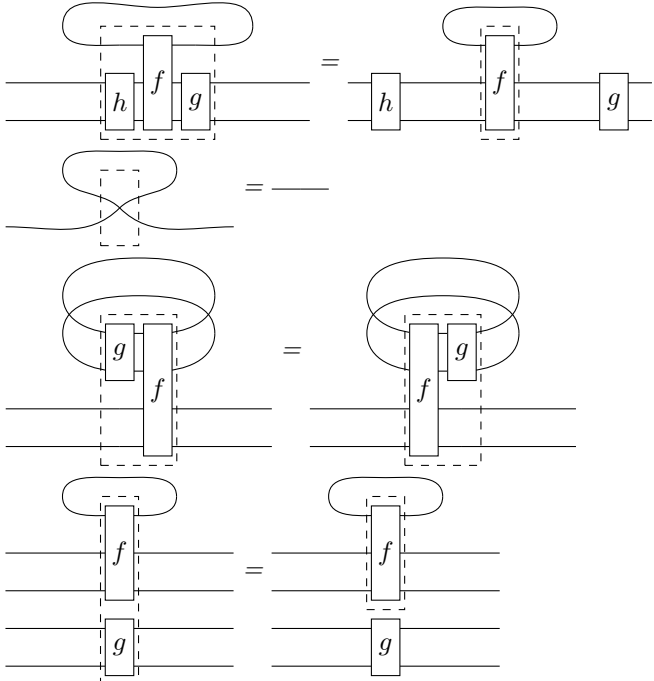
\begin{tikzpicture}[baseline=(current bounding box.center)]
\begin{pgfonlayer}{nodelayer}
\node [] (0) at (-0.750,-0.500) {};
\node [] (1) at (-0.000,-1.250) {};
\node [] (2) at (-0.000,-1.750) {};
\node [shape=rectangle,fill=white,draw,label={[label position=center]$g$},inner xsep=0.1875cm,inner ysep=0.375cm] (3) at (-0.500,-1.500) {};
\node [] (4) at (-0.500,-1.250) {};
\node [] (5) at (-0.500,-1.750) {};
\node [] (6) at (-0.500,-1.250) {};
\node [] (7) at (-0.500,-1.750) {};
\node [] (8) at (-1.250,-0.500) {};
\node [] (9) at (-2.750,-0.500) {};
\path[dashed=None,draw=,color=black] (-1.750,-0.500) rectangle (-2.250,-2.000);
\node at (-2.000,-1.250) {};
\node [] (10) at (-2.000,-0.250) {};
\node [shape=rectangle,fill=white,draw,label={[label position=center]$f$},inner xsep=0.1875cm,inner ysep=0.625cm] (11) at (-2.000,-1.250) {};
\node [] (12) at (-2.000,-0.750) {};
\node [] (13) at (-2.000,-1.250) {};
\node [] (14) at (-2.000,-1.750) {};
\node [] (15) at (-2.000,-0.750) {};
\node [] (16) at (-2.000,-1.250) {};
\node [] (17) at (-2.000,-1.750) {};
\node [] (18) at (-3.250,-0.500) {};
\node [shape=rectangle,fill=white,draw,label={[label position=center]$h$},inner xsep=0.1875cm,inner ysep=0.375cm] (19) at (-3.500,-1.500) {};
\node [] (20) at (-3.500,-1.250) {};
\node [] (21) at (-3.500,-1.750) {};
\node [] (22) at (-3.500,-1.250) {};
\node [] (23) at (-3.500,-1.750) {};
\node [] (24) at (-4.000,-1.250) {};
\node [] (25) at (-4.000,-1.750) {};
\node [label={[label position=center]=}] (26) at (-4.250,-1.000) {};
\node [] (27) at (-4.750,-0.500) {};
\node [] (28) at (-4.500,-1.250) {};
\node [] (29) at (-4.500,-1.750) {};
\node [] (30) at (-5.250,-0.500) {};
\node [] (31) at (-7.750,-0.500) {};
\path[dashed=None,draw=,color=black] (-5.750,-0.500) rectangle (-7.250,-2.000);
\node at (-6.500,-1.250) {};
\node [] (32) at (-6.500,-0.250) {};
\node [] (33) at (-6.000,-0.750) {};
\node [shape=rectangle,fill=white,draw,label={[label position=center]$g$},inner xsep=0.1875cm,inner ysep=0.375cm] (34) at (-6.000,-1.500) {};
\node [] (35) at (-6.000,-1.250) {};
\node [] (36) at (-6.000,-1.750) {};
\node [] (37) at (-6.000,-1.250) {};
\node [] (38) at (-6.000,-1.750) {};
\node [shape=rectangle,fill=white,draw,label={[label position=center]$f$},inner xsep=0.1875cm,inner ysep=0.625cm] (39) at (-6.500,-1.250) {};
\node [] (40) at (-6.500,-0.750) {};
\node [] (41) at (-6.500,-1.250) {};
\node [] (42) at (-6.500,-1.750) {};
\node [] (43) at (-6.500,-0.750) {};
\node [] (44) at (-6.500,-1.250) {};
\node [] (45) at (-6.500,-1.750) {};
\node [] (46) at (-7.000,-0.750) {};
\node [shape=rectangle,fill=white,draw,label={[label position=center]$h$},inner xsep=0.1875cm,inner ysep=0.375cm] (47) at (-7.000,-1.500) {};
\node [] (48) at (-7.000,-1.250) {};
\node [] (49) at (-7.000,-1.750) {};
\node [] (50) at (-7.000,-1.250) {};
\node [] (51) at (-7.000,-1.750) {};
\node [] (52) at (-8.250,-0.500) {};
\node [] (53) at (-8.500,-1.250) {};
\node [] (54) at (-8.500,-1.750) {};
\end{pgfonlayer}
\begin{pgfonlayer}{edgelayer}
\draw[] (1.center) to[out=180, in=  0] (6.center);
\draw[] (2.center) to[out=180, in=  0] (7.center);
\draw[] (8.center) to[out=-89.9, in=  0] (15.center);
\draw[] (4.center) to[out=180, in=  0] (16.center);
\draw[] (5.center) to[out=180, in=  0] (17.center);
\draw[] (12.center) to[out=180, in=-89.9] (9.center) to[out= 90, in=180] (10.center) to[out=  0, in= 90] (8.center);
\draw[] (13.center) to[out=180, in=  0] (22.center);
\draw[] (14.center) to[out=180, in=  0] (23.center);
\draw[] (20.center) to[out=180, in=  0] (24.center);
\draw[] (21.center) to[out=180, in=  0] (25.center);
\draw[] (30.center) to[out=-89.9, in=  0] (33.center);
\draw[] (28.center) to[out=180, in=  0] (37.center);
\draw[] (29.center) to[out=180, in=  0] (38.center);
\draw[] (33.center) to[out=180, in=  0] (43.center);
\draw[] (35.center) to[out=180, in=  0] (44.center);
\draw[] (36.center) to[out=180, in=  0] (45.center);
\draw[] (40.center) to[out=180, in=  0] (46.center);
\draw[] (41.center) to[out=180, in=  0] (50.center);
\draw[] (42.center) to[out=180, in=  0] (51.center);
\draw[] (46.center) to[out=180, in=-89.9] (31.center) to[out= 90, in=180] (32.center) to[out=  0, in= 90] (30.center);
\draw[] (48.center) to[out=180, in=  0] (53.center);
\draw[] (49.center) to[out=180, in=  0] (54.center);
\end{pgfonlayer}
\end{tikzpicture}

  \begin{tikzpicture}[baseline=(current bounding box.center)]
\begin{pgfonlayer}{nodelayer}
\node [] (0) at (-0.000,-0.750) {};
\node [] (1) at (-0.375,-0.750) {};
\node [] (2) at (-0.750,-0.750) {};
\node [label={[label position=center]=}] (3) at (-1.000,-0.750) {};
\node [] (4) at (-1.500,-0.500) {};
\node [] (5) at (-1.250,-1.250) {};
\node [] (6) at (-2.000,-0.500) {};
\node [] (7) at (-3.500,-0.500) {};
\path[dashed=None,draw=,color=black] (-2.500,-0.500) rectangle (-3.000,-1.500);
\node at (-2.750,-1.000) {};
\node [] (8) at (-2.750,-0.250) {};
\node [] (9) at (-2.750,-1.000) {};
\node [] (10) at (-4.000,-0.500) {};
\node [] (11) at (-4.250,-1.250) {};
\end{pgfonlayer}
\begin{pgfonlayer}{edgelayer}
\draw[] (0.center) to[out=180, in=  0] (1.center);
\draw[] (1.center) to[out=180, in=  0] (2.center);
\draw[] (5.center) to[out=180, in=-44.9] (9.center) to[out=135, in=-89.9] (7.center) to[out= 90, in=180] (8.center) to[out=  0, in= 90] (6.center);
\draw[] (6.center) to[out=-89.9, in= 45] (9.center) to[out=-135, in=  0] (11.center);
\end{pgfonlayer}
\end{tikzpicture}

  \begin{tikzpicture}[baseline=(current bounding box.center)]
\begin{pgfonlayer}{nodelayer}
\node [] (0) at (-0.250,-1.000) {};
\node [] (1) at (-0.000,-2.250) {};
\node [] (2) at (-0.000,-2.750) {};
\node [] (3) at (-0.750,-0.750) {};
\node [] (4) at (-0.750,-1.250) {};
\node [] (5) at (-2.750,-0.750) {};
\node [] (6) at (-2.750,-1.250) {};
\path[dashed=None,draw=,color=black] (-1.250,-1.000) rectangle (-2.250,-3.000);
\node at (-1.750,-2.000) {};
\node [] (7) at (-1.750,-0.250) {};
\node [] (8) at (-1.750,-0.750) {};
\node [shape=rectangle,fill=white,draw,label={[label position=center]$g$},inner xsep=0.1875cm,inner ysep=0.375cm] (9) at (-1.500,-1.500) {};
\node [] (10) at (-1.500,-1.250) {};
\node [] (11) at (-1.500,-1.750) {};
\node [] (12) at (-1.500,-1.250) {};
\node [] (13) at (-1.500,-1.750) {};
\node [] (14) at (-1.500,-2.250) {};
\node [] (15) at (-1.500,-2.750) {};
\node [shape=rectangle,fill=white,draw,label={[label position=center]$f$},inner xsep=0.1875cm,inner ysep=0.875cm] (16) at (-2.000,-2.000) {};
\node [] (17) at (-2.000,-1.250) {};
\node [] (18) at (-2.000,-1.750) {};
\node [] (19) at (-2.000,-2.250) {};
\node [] (20) at (-2.000,-2.750) {};
\node [] (21) at (-2.000,-1.250) {};
\node [] (22) at (-2.000,-1.750) {};
\node [] (23) at (-2.000,-2.250) {};
\node [] (24) at (-2.000,-2.750) {};
\node [] (25) at (-3.250,-1.000) {};
\node [] (26) at (-3.500,-2.250) {};
\node [] (27) at (-3.500,-2.750) {};
\node [label={[label position=center]=}] (28) at (-3.750,-1.500) {};
\node [] (29) at (-4.250,-1.000) {};
\node [] (30) at (-4.000,-2.250) {};
\node [] (31) at (-4.000,-2.750) {};
\node [] (32) at (-4.750,-0.750) {};
\node [] (33) at (-4.750,-1.250) {};
\node [] (34) at (-6.750,-0.750) {};
\node [] (35) at (-6.750,-1.250) {};
\path[dashed=None,draw=,color=black] (-5.250,-1.000) rectangle (-6.250,-3.000);
\node at (-5.750,-2.000) {};
\node [] (36) at (-5.750,-0.250) {};
\node [] (37) at (-5.750,-0.750) {};
\node [shape=rectangle,fill=white,draw,label={[label position=center]$f$},inner xsep=0.1875cm,inner ysep=0.875cm] (38) at (-5.500,-2.000) {};
\node [] (39) at (-5.500,-1.250) {};
\node [] (40) at (-5.500,-1.750) {};
\node [] (41) at (-5.500,-2.250) {};
\node [] (42) at (-5.500,-2.750) {};
\node [] (43) at (-5.500,-1.250) {};
\node [] (44) at (-5.500,-1.750) {};
\node [] (45) at (-5.500,-2.250) {};
\node [] (46) at (-5.500,-2.750) {};
\node [shape=rectangle,fill=white,draw,label={[label position=center]$g$},inner xsep=0.1875cm,inner ysep=0.375cm] (47) at (-6.000,-1.500) {};
\node [] (48) at (-6.000,-1.250) {};
\node [] (49) at (-6.000,-1.750) {};
\node [] (50) at (-6.000,-1.250) {};
\node [] (51) at (-6.000,-1.750) {};
\node [] (52) at (-6.000,-2.250) {};
\node [] (53) at (-6.000,-2.750) {};
\node [] (54) at (-7.250,-1.000) {};
\node [] (55) at (-7.500,-2.250) {};
\node [] (56) at (-7.500,-2.750) {};
\end{pgfonlayer}
\begin{pgfonlayer}{edgelayer}
\draw[] (3.center) to[out=-89.9, in=  0] (12.center);
\draw[] (4.center) to[out=-89.9, in=  0] (13.center);
\draw[] (1.center) to[out=180, in=  0] (14.center);
\draw[] (2.center) to[out=180, in=  0] (15.center);
\draw[] (10.center) to[out=180, in=  0] (21.center);
\draw[] (11.center) to[out=180, in=  0] (22.center);
\draw[] (14.center) to[out=180, in=  0] (23.center);
\draw[] (15.center) to[out=180, in=  0] (24.center);
\draw[] (17.center) to[out=180, in=-89.9] (5.center) to[out= 90, in=180] (7.center) to[out=  0, in= 90] (3.center);
\draw[] (18.center) to[out=180, in=-89.9] (6.center) to[out= 90, in=180] (8.center) to[out=  0, in= 90] (4.center);
\draw[] (19.center) to[out=180, in=  0] (26.center);
\draw[] (20.center) to[out=180, in=  0] (27.center);
\draw[] (32.center) to[out=-89.9, in=  0] (43.center);
\draw[] (33.center) to[out=-89.9, in=  0] (44.center);
\draw[] (30.center) to[out=180, in=  0] (45.center);
\draw[] (31.center) to[out=180, in=  0] (46.center);
\draw[] (39.center) to[out=180, in=  0] (50.center);
\draw[] (40.center) to[out=180, in=  0] (51.center);
\draw[] (41.center) to[out=180, in=  0] (52.center);
\draw[] (42.center) to[out=180, in=  0] (53.center);
\draw[] (48.center) to[out=180, in=-89.9] (34.center) to[out= 90, in=180] (36.center) to[out=  0, in= 90] (32.center);
\draw[] (49.center) to[out=180, in=-89.9] (35.center) to[out= 90, in=180] (37.center) to[out=  0, in= 90] (33.center);
\draw[] (52.center) to[out=180, in=  0] (55.center);
\draw[] (53.center) to[out=180, in=  0] (56.center);
\end{pgfonlayer}
\end{tikzpicture}

 \begin{tikzpicture}[baseline=(current bounding box.center)]
\begin{pgfonlayer}{nodelayer}
\node [] (0) at (-0.250,-0.500) {};
\node [] (1) at (-0.000,-1.250) {};
\node [] (2) at (-0.000,-1.750) {};
\node [] (3) at (-0.000,-2.250) {};
\node [] (4) at (-0.000,-2.750) {};
\node [] (5) at (-0.750,-0.500) {};
\node [] (6) at (-2.250,-0.500) {};
\path[dashed=None,draw=,color=black] (-1.250,-0.500) rectangle (-1.750,-2.000);
\node at (-1.500,-1.250) {};
\node [] (7) at (-1.500,-0.250) {};
\node [shape=rectangle,fill=white,draw,label={[label position=center]$f$},inner xsep=0.1875cm,inner ysep=0.625cm] (8) at (-1.500,-1.250) {};
\node [] (9) at (-1.500,-0.750) {};
\node [] (10) at (-1.500,-1.250) {};
\node [] (11) at (-1.500,-1.750) {};
\node [] (12) at (-1.500,-0.750) {};
\node [] (13) at (-1.500,-1.250) {};
\node [] (14) at (-1.500,-1.750) {};
\node [shape=rectangle,fill=white,draw,label={[label position=center]$g$},inner xsep=0.1875cm,inner ysep=0.375cm] (15) at (-1.500,-2.500) {};
\node [] (16) at (-1.500,-2.250) {};
\node [] (17) at (-1.500,-2.750) {};
\node [] (18) at (-1.500,-2.250) {};
\node [] (19) at (-1.500,-2.750) {};
\node [] (20) at (-2.750,-0.500) {};
\node [] (21) at (-3.000,-1.250) {};
\node [] (22) at (-3.000,-1.750) {};
\node [] (23) at (-3.000,-2.250) {};
\node [] (24) at (-3.000,-2.750) {};
\node [label={[label position=center]=}] (25) at (-3.250,-1.500) {};
\node [] (26) at (-3.750,-0.500) {};
\node [] (27) at (-3.500,-1.250) {};
\node [] (28) at (-3.500,-1.750) {};
\node [] (29) at (-3.500,-2.250) {};
\node [] (30) at (-3.500,-2.750) {};
\node [] (31) at (-4.250,-0.500) {};
\node [] (32) at (-5.750,-0.500) {};
\path[dashed=None,draw=,color=black] (-4.750,-0.500) rectangle (-5.250,-3.000);
\node at (-5.000,-1.750) {};
\node [] (33) at (-5.000,-0.250) {};
\node [shape=rectangle,fill=white,draw,label={[label position=center]$f$},inner xsep=0.1875cm,inner ysep=0.625cm] (34) at (-5.000,-1.250) {};
\node [] (35) at (-5.000,-0.750) {};
\node [] (36) at (-5.000,-1.250) {};
\node [] (37) at (-5.000,-1.750) {};
\node [] (38) at (-5.000,-0.750) {};
\node [] (39) at (-5.000,-1.250) {};
\node [] (40) at (-5.000,-1.750) {};
\node [shape=rectangle,fill=white,draw,label={[label position=center]$g$},inner xsep=0.1875cm,inner ysep=0.375cm] (41) at (-5.000,-2.500) {};
\node [] (42) at (-5.000,-2.250) {};
\node [] (43) at (-5.000,-2.750) {};
\node [] (44) at (-5.000,-2.250) {};
\node [] (45) at (-5.000,-2.750) {};
\node [] (46) at (-6.250,-0.500) {};
\node [] (47) at (-6.500,-1.250) {};
\node [] (48) at (-6.500,-1.750) {};
\node [] (49) at (-6.500,-2.250) {};
\node [] (50) at (-6.500,-2.750) {};
\end{pgfonlayer}
\begin{pgfonlayer}{edgelayer}
\draw[] (5.center) to[out=-89.9, in=  0] (12.center);
\draw[] (1.center) to[out=180, in=  0] (13.center);
\draw[] (2.center) to[out=180, in=  0] (14.center);
\draw[] (9.center) to[out=180, in=-89.9] (6.center) to[out= 90, in=180] (7.center) to[out=  0, in= 90] (5.center);
\draw[] (3.center) to[out=180, in=  0] (18.center);
\draw[] (4.center) to[out=180, in=  0] (19.center);
\draw[] (10.center) to[out=180, in=  0] (21.center);
\draw[] (11.center) to[out=180, in=  0] (22.center);
\draw[] (16.center) to[out=180, in=  0] (23.center);
\draw[] (17.center) to[out=180, in=  0] (24.center);
\draw[] (31.center) to[out=-89.9, in=  0] (38.center);
\draw[] (27.center) to[out=180, in=  0] (39.center);
\draw[] (28.center) to[out=180, in=  0] (40.center);
\draw[] (29.center) to[out=180, in=  0] (44.center);
\draw[] (30.center) to[out=180, in=  0] (45.center);
\draw[] (35.center) to[out=180, in=-89.9] (32.center) to[out= 90, in=180] (33.center) to[out=  0, in= 90] (31.center);
\draw[] (36.center) to[out=180, in=  0] (47.center);
\draw[] (37.center) to[out=180, in=  0] (48.center);
\draw[] (42.center) to[out=180, in=  0] (49.center);
\draw[] (43.center) to[out=180, in=  0] (50.center);
\end{pgfonlayer}
\end{tikzpicture}

A traced prop functor $F$ is a prop functor that preserves the trace: $tr F(f) = F(tr f)$
\end{definition}

We now will be interested in presentations of traced props. The definition of a traced prop given by generators and equations is similar to Definition~\ref{defn:universalprop}, by replacing props and prop functors by traced prop and prop functors.

Another way of defining it is by first looking at the prop given by generators and equations, and then adding the trace:
\begin{definition}
  Let $P$ be a prop. The traced completion $\widehat P$ of $P$ is a prop s.t.
  \begin{itemize}
  \item $\widehat P$ is traced.
  \item There is a prop functor $\iota: P  \to \widehat P$.
  \item For any other prop ($Q,\eta$) with the same properties, there is a unique traced prop functor $G$ from $\widehat P$ to $Q$ s.t. $G\iota = \eta$
  \end{itemize}    
\end{definition}
By the universal property of the definition, it is easy to see that the traced completion is unique up to isomorphism. It is less clear that it exists, and we will construct explicitely the traced completion of a prop in the next section.

It is important to note that $\iota$ is in general not an embedding of $P$ into $\widehat P$, i.e., it is not always injective.

\vspace{5mm}

Our goal now is to add traces to matrices. 
Consider the prop  $\matrices{\mathbb{Z}_+}$ of definition~\ref{def:matricesintegers}, i.e. the prop of matrices with noninteger coefficients. As explained, the morphisms of this prop can also be seen as bipartite graphs. Now consider the \emph{traced prop} given by the same axioms, and in particular the morphisms $0 \to 0$ of this prop, i.e. morphisms with no input and no output.
They are clearly obtained by  plugging the outputs of some bipartite graph back into its inputs, so that  all vertices are not ``internal vertices'' and there are no inputs/output vertices.

An example is illustrated in Figure~\ref{fig:matrixgraph}.
We start from a matrix $M$ (\ref{fig:matrixgraph}.(a)), and represent it in the prop using a diagram (\ref{fig:matrixgraph}.(b)).
Now we trace the diagram by plugging inputs into outputs (\ref{fig:matrixgraph}.(c)). By identifying both extremities that are joined together, one can see we just obtained a fancy way of representing the multigraph (\ref{fig:matrixgraph}.(d)),i.e. \emph{the graph whose adjacency matrix is exactly} $M$.

Therefore, to each matrix $M$ of nonnegative integers, one can associate a morphism (called $\tilde{M}$ in the next definition) with no inputs and outputs in $\widehat{\matrices{\mathbb{Z}}}$, which looks visually like the graph whose adjacency matrix is $M$.

However our new prop has relations, which means that some of these square matrices actually represent the same morphism in this prop. When are two matrices/graphs equal exactly?

\begin{theorem}
  \label{thm:z}
Let $\widehat{\matrices{\mathbb{Z}}}$ be a \emph{traced} completion of the prop $\matrices{\mathbb{Z}}$ of Definition~\ref{def:matricesintegers}.

To any square matrix $M$ of $\matrices{\mathbb{Z}}$ of size $n$, we associate a $0 \to 0$ morphism $\tilde{M}$ of $\widehat{\matrices{\mathbb{Z}}}$ by tracing everything, i.e. $\tilde{M} = tr_n(\iota(M))$.

Let $M, N$ two square matrices, possibly of different size. Then ${\tilde M}$ and ${\tilde N}$  represent the same morphism exactly when $M$ and $N$ are flow-equivalent.
\end{theorem}  
Flow equivalence \cite{parry75} is a relation on graphs (equivalently on square matrices of nonnegative integers) used in symbolic dynamics. The definition will be given later on.
\begin{figure}[t]
  \begin{tabular}{cccc}
    $(a)$ & $(b)$ & $(c)$ & $(d)$ \\
    $M = \begin{pmatrix}
  0 & 2 & 1 \\
  0 & 1 & 1 \\
  0 & 0 & 1 \\
    \end{pmatrix}$ &
    \begin{tikzpicture}[baseline=(current bounding box.center)]
\begin{pgfonlayer}{nodelayer}
\node [] (0) at (-0.000,-0.500) {};
\node [style=monoid] (1) at (-0.500,-0.500) {};
\node [] (2) at (-0.000,-1.500) {};
\node [style=monoid] (3) at (-0.500,-1.500) {};
\node [] (4) at (-0.000,-2.500) {};
\node [style=monoid] (5) at (-0.500,-2.500) {};
\node [] (6) at (-1.000,-0.250) {};
\node [] (7) at (-1.000,-0.750) {};
\node [] (8) at (-1.000,-1.500) {};
\node [] (9) at (-1.000,-2.250) {};
\node [] (10) at (-1.000,-2.750) {};
\node [style=bmonoid] (11) at (-1.500,-0.500) {};
\node [] (12) at (-2.000,-0.500) {};
\node [style=bmonoid] (13) at (-1.500,-1.500) {};
\node [] (14) at (-2.000,-1.500) {};
\node [] (15) at (-1.438,-2.500) {};
\node [] (16) at (-2.000,-2.500) {};
\end{pgfonlayer}
\begin{pgfonlayer}{edgelayer}
\draw[] (0.center) to[out=180, in=  0] (1.center);
\draw[] (2.center) to[out=180, in=  0] (3.center);
\draw[] (4.center) to[out=180, in=  0] (5.center);
\draw[] (3.center) to[out=127, in=  0] (6.center);
\draw[] (3.center) to[out=180, in=  0] (7.center);
\draw[] (5.center) to[out=180, in=  0] (9.center);
\draw[] (5.center) to[out=-127, in=  0] (10.center);
\draw[] (6.center) to[out=180, in=53.1] (11.center);
\draw[] (7.center) to[out=180, in=  0] (11.center);
\draw[] (5.center) to[out=127, in=-44.9] (8.center) to[out=135, in=-53] (11.center);
\draw[] (11.center) to[out=180, in=  0] (12.center);
\draw[] (3.center) to[out=-127, in= 45] (8.center) to[out=-135, in= 45] (13.center);
\draw[] (9.center) to[out=180, in=-44.9] (13.center);
\draw[] (13.center) to[out=180, in=  0] (14.center);
\draw[] (10.center) to[out=180, in=  0] (15.center);
\draw[] (15.center) to[out=180, in=  0] (16.center);
\end{pgfonlayer}
\end{tikzpicture}
 &
    \begin{tikzpicture}[baseline=(current bounding box.center)]
\begin{pgfonlayer}{nodelayer}
\node [] (0) at (-0.500,-1.250) {};
\node [] (1) at (-0.500,-2.000) {};
\node [] (2) at (-0.500,-2.750) {};
\node [] (3) at (-3.000,-1.250) {};
\node [] (4) at (-3.000,-2.000) {};
\node [] (5) at (-3.000,-2.750) {};
\path[dashed=None,draw=,color=black] (-1.000,-1.500) rectangle (-2.500,-4.500);
\node at (-1.750,-3.000) {};
\node [] (6) at (-1.750,-0.250) {};
\node [] (7) at (-1.750,-0.750) {};
\node [] (8) at (-1.750,-1.250) {};
\node [style=monoid] (9) at (-1.250,-2.000) {};
\node [style=monoid] (10) at (-1.250,-3.000) {};
\node [style=monoid] (11) at (-1.250,-4.000) {};
\node [] (12) at (-1.750,-1.750) {};
\node [] (13) at (-1.750,-2.250) {};
\node [] (14) at (-1.750,-3.000) {};
\node [] (15) at (-1.750,-3.750) {};
\node [] (16) at (-1.750,-4.250) {};
\node [style=bmonoid] (17) at (-2.250,-2.000) {};
\node [style=bmonoid] (18) at (-2.250,-3.000) {};
\node [] (19) at (-2.250,-4.000) {};
\end{pgfonlayer}
\begin{pgfonlayer}{edgelayer}
\draw[] (0.center) to[out=-89.9, in=  0] (9.center);
\draw[] (1.center) to[out=-89.9, in=  0] (10.center);
\draw[] (2.center) to[out=-89.9, in=  0] (11.center);
\draw[] (10.center) to[out=127, in=  0] (12.center);
\draw[] (10.center) to[out=180, in=  0] (13.center);
\draw[] (11.center) to[out=180, in=  0] (15.center);
\draw[] (11.center) to[out=-127, in=  0] (16.center);
\draw[] (12.center) to[out=180, in=53.1] (17.center);
\draw[] (13.center) to[out=180, in=  0] (17.center);
\draw[] (11.center) to[out=127, in=-44.9] (14.center) to[out=135, in=-53] (17.center);
\draw[] (10.center) to[out=-127, in= 45] (14.center) to[out=-135, in= 45] (18.center);
\draw[] (15.center) to[out=180, in=-44.9] (18.center);
\draw[] (16.center) to[out=180, in=  0] (19.center);
\draw[] (17.center) to[out=180, in=-89.9] (3.center) to[out= 90, in=180] (6.center) to[out=  0, in= 90] (0.center);
\draw[] (18.center) to[out=180, in=-89.9] (4.center) to[out= 90, in=180] (7.center) to[out=  0, in= 90] (1.center);
\draw[] (19.center) to[out=180, in=-89.9] (5.center) to[out= 90, in=180] (8.center) to[out=  0, in= 90] (2.center);
\end{pgfonlayer}
\end{tikzpicture}
 &
\begin{tikzpicture}[baseline=(current bounding box.center)]
  \node[circle,draw] (A) at (0,2) {};
  \node[circle,draw] (B) at (0,1) {};
  \node[circle,draw] (C) at (0,0) {};

  \draw[->, bend left] (A) edge (B);
  \draw[->, bend right] (A) edge (B);
  \draw[->, bend left=40] (A) edge (C);
  \draw[->, loop left] (B) edge (B);
  \draw[->, bend left] (B) edge (C);
  \draw[->, loop left] (C) edge (C);
\end{tikzpicture}  
\\            
\end{tabular}    
  \caption{From a matrix to a graph, using diagrams.}
  \label{fig:matrixgraph}
  
\end{figure}
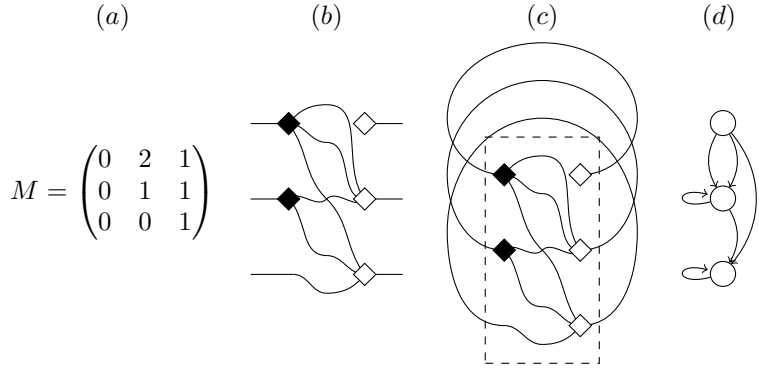  

This theorem is already present in the work of Hillman~\cite{Hillman}, although its use has never been investigated thoroughly. Here we make it more precise.

Before proving the theorem and giving applications, we will do the same work with the prop of Proposition~\ref{def:matricespolynomial}. Here we are able to represent matrices in $\mathbb{Z}_+[t]$. There are harder to represents as graphs, unless one starts labelling the edges with elements of $\mathbb{Z}_+[t]$. However, if all coefficients of the matrix are integer multiples of $t$, then we can also see them as graphs.

\begin{theorem}
  \label{thm:zt}
Let $\widehat {\matrices{\mathbb{Z}_+[t]}}$ be a traced completion of the prop $\matrices{\mathbb{Z}_+[t]}$ of Proposition~\ref{def:matricespolynomial}.
We use the same notation as the previous theorem.

Let $M$ and $N$ two square matrices with nonnegative integral coefficients. Then $\widetilde{tM}$ and $\widetilde{tN}$  represent the same morphism exactly when $M$ and $N$ are strong-shift equivalent.
\end{theorem}  

Two matrices are strong-shift equivalent exactly when the graphs represents isomorphic subshifts~\cite{williams2}. As a consequence, this theorem essentially states that (a) deciding the theory of strong-shift equivalence, equivalently  deciding whether two graphs represents isomorphic subshifts of finite type is exactly the same as (b) deciding which equalities are true in all models of traced bialgebras.

The remainder of this section is devoted to the proof of the two statements, and of the definitions of flow equivalence and strong-shift equivalence.

\subsection{Understanding traced props}

Suppose we already know quite well a prop given by generators and relations. How do we obtain the traced prop given by the same relations (i.e. the traced completion of the original prop) ?
This is what we  will do below.

We denote $\sigma_{a,b} : a+b \to b+a$ the symmetry that on diagrams exchanges the first $a$ wires and the last $b$ wires.
\begin{proposition}
  \label{prop:tracedcompletion}
  Let $P$ be a prop.  The traced completion $\widehat{P}$ of $P$ is isomorphic to the following prop $Q$:

  \begin{itemize}
  \item Morphisms $n \to m$ are equivalence classes for $\sim$ of pairs $(M,k)$ where $M \in P[k+n,k+m]$  and $k$ is an integer, where $\sim$ is the smallest equivalence relation s.t.
    \begin{itemize}
    \item $((g \otimes id^{\otimes m} ) \circ M,q) \sim (M \circ (g \otimes id^{\otimes n} ),p)$ where $M: n+q \to m+p$ and $g: p \to q$
    \item $(M,k) \sim  ((id \otimes M) \circ (\sigma_{n,1} \otimes id^{\otimes m}) ,k+1)$ 
    \end{itemize}
    To preverse ink, we will use  $[M,k]$ to denote the equivalence class that contains $(M,k)$.
  \item The composition of  $[M,p]$ and $[N,q]$, where $M: p+m\to p+n$ and $N: p+r \to p+m$ is 
    $\left[(N \otimes M) \circ (id^{\otimes q} \otimes \sigma_{m+q,r}), p+q+m\right]$
  \item The tensor product of  $[M,p]$ and $[N,q]$, where $M: m+p\to n+p$ and $N: m'+q \to n'+q$ is
    $\left[(id^{\otimes p} \otimes \sigma_{n,q} \otimes id^{\otimes n'}) \circ (N \otimes M) \circ (id^{\otimes p} \otimes \sigma_{m,q} \otimes id^{\otimes m'}) , p+q\right]$
  \item The trace of $[M,k]$ is  $[M,k+1]$
\end{itemize}    
\end{proposition}

The definition is a bit hard to digest, and will be easier to see graphically.
We will picture pairs $[M,k]$ as the diagram $M$ with the first $k$ wires dashed:
\begin{tikzpicture}[baseline=(current bounding box.center)]
\begin{pgfonlayer}{nodelayer}
\node [] (0) at (-0.000,-0.250) {};
\node [] (1) at (-0.000,-0.750) {};
\node [shape=rectangle,fill=white,draw,label={[label position=center]$M$},inner xsep=0.1875cm,inner ysep=0.375cm] (2) at (-0.500,-0.500) {};
\node [] (3) at (-0.500,-0.250) {};
\node [] (4) at (-0.500,-0.750) {};
\node [] (5) at (-0.500,-0.250) {};
\node [] (6) at (-0.500,-0.750) {};
\node [] (7) at (-1.000,-0.250) {};
\node [] (8) at (-1.000,-0.750) {};
\end{pgfonlayer}
\begin{pgfonlayer}{edgelayer}
\draw[dashed] (0.center) to[out=180, in=  0] (5.center);
\draw[] (1.center) to[out=180, in=  0] (6.center);
\draw[dashed] (3.center) to[out=180, in=  0] (7.center);
\draw[] (4.center) to[out=180, in=  0] (8.center);
\end{pgfonlayer}
\end{tikzpicture}

For simplicity, wires in the following diagrams may represent an arbitrary (possibly zero) number of wires

  \begin{itemize}
  \item The equivalence class on pairs is defined by the following rules (the second equation gives two different diagrammatic interpretations, depending on whether the wire to be permuted with the first wire is dashed or nondashed):
    
    
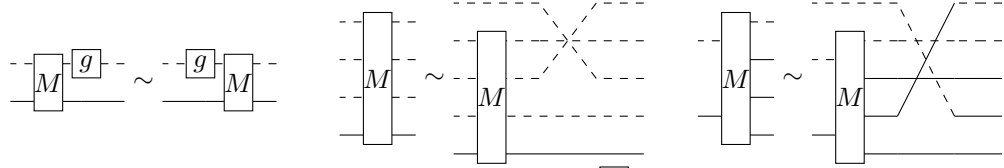
{ \begin{tikzpicture}[baseline=(current bounding box.center)]
\begin{pgfonlayer}{nodelayer}
\node [] (0) at (-0.000,-0.250) {};
\node [] (1) at (-0.000,-0.750) {};
\node [shape=rectangle,fill=white,draw,label={[label position=center]$M$},inner xsep=0.1875cm,inner ysep=0.375cm] (2) at (-0.500,-0.500) {};
\node [] (3) at (-0.500,-0.250) {};
\node [] (4) at (-0.500,-0.750) {};
\node [] (5) at (-0.500,-0.250) {};
\node [] (6) at (-0.500,-0.750) {};
\node [shape=rectangle,fill=white,draw,label={[label position=center]$g$},inner xsep=0.1875cm,inner ysep=0.1875cm] (7) at (-1.000,-0.250) {};
\node [] (8) at (-1.000,-0.250) {};
\node [] (9) at (-1.000,-0.250) {};
\node [] (10) at (-1.500,-0.250) {};
\node [] (11) at (-0.938,-0.750) {};
\node [] (12) at (-1.500,-0.750) {};
\node [label={[label position=center]$\sim$}] (13) at (-1.750,-0.500) {};
\node [] (14) at (-2.000,-0.250) {};
\node [shape=rectangle,fill=white,draw,label={[label position=center]$g$},inner xsep=0.1875cm,inner ysep=0.1875cm] (15) at (-2.500,-0.250) {};
\node [] (16) at (-2.500,-0.250) {};
\node [] (17) at (-2.500,-0.250) {};
\node [] (18) at (-2.000,-0.750) {};
\node [] (19) at (-2.562,-0.750) {};
\node [shape=rectangle,fill=white,draw,label={[label position=center]$M$},inner xsep=0.1875cm,inner ysep=0.375cm] (20) at (-3.000,-0.500) {};
\node [] (21) at (-3.000,-0.250) {};
\node [] (22) at (-3.000,-0.750) {};
\node [] (23) at (-3.000,-0.250) {};
\node [] (24) at (-3.000,-0.750) {};
\node [] (25) at (-3.500,-0.250) {};
\node [] (26) at (-3.500,-0.750) {};
\end{pgfonlayer}
\begin{pgfonlayer}{edgelayer}
\draw[dashed] (0.center) to[out=180, in=  0] (5.center);
\draw[] (1.center) to[out=180, in=  0] (6.center);
\draw[dashed] (3.center) to[out=180, in=  0] (9.center);
\draw[dashed] (8.center) to[out=180, in=  0] (10.center);
\draw[] (4.center) to[out=180, in=  0] (11.center);
\draw[] (11.center) to[out=180, in=  0] (12.center);
\draw[dashed] (14.center) to[out=180, in=  0] (17.center);
\draw[] (18.center) to[out=180, in=  0] (19.center);
\draw[dashed] (16.center) to[out=180, in=  0] (23.center);
\draw[] (19.center) to[out=180, in=  0] (24.center);
\draw[dashed] (21.center) to[out=180, in=  0] (25.center);
\draw[] (22.center) to[out=180, in=  0] (26.center);
\end{pgfonlayer}
\end{tikzpicture}
 } \hfill    { \begin{tikzpicture}[baseline=(current bounding box.center)]
\begin{pgfonlayer}{nodelayer}
\node [] (0) at (-0.000,-0.250) {};
\node [] (1) at (-0.000,-0.750) {};
\node [] (2) at (-0.000,-1.250) {};
\node [] (3) at (-0.625,-0.250) {};
\node [] (4) at (-1.375,-1.250) {};
\node [] (5) at (-0.625,-0.750) {};
\node [] (6) at (-1.375,-0.750) {};
\node [] (7) at (-0.625,-1.250) {};
\node [] (8) at (-1.375,-0.250) {};
\node [] (9) at (-0.000,-1.750) {};
\node [] (10) at (-1.167,-1.750) {};
\node [] (11) at (-0.000,-2.250) {};
\node [] (12) at (-1.167,-2.250) {};
\node [] (13) at (-2.000,-0.250) {};
\node [] (14) at (-2.500,-0.250) {};
\node [shape=rectangle,fill=white,draw,label={[label position=center]$M$},inner xsep=0.1875cm,inner ysep=0.875cm] (15) at (-2.000,-1.500) {};
\node [] (16) at (-2.000,-0.750) {};
\node [] (17) at (-2.000,-1.250) {};
\node [] (18) at (-2.000,-1.750) {};
\node [] (19) at (-2.000,-2.250) {};
\node [] (20) at (-2.000,-0.750) {};
\node [] (21) at (-2.000,-1.250) {};
\node [] (22) at (-2.000,-1.750) {};
\node [] (23) at (-2.000,-2.250) {};
\node [] (24) at (-2.500,-0.750) {};
\node [] (25) at (-2.500,-1.250) {};
\node [] (26) at (-2.500,-1.750) {};
\node [] (27) at (-2.500,-2.250) {};
\node [label={[label position=center]$\sim$}] (28) at (-2.750,-1.250) {};
\node [] (29) at (-3.000,-0.500) {};
\node [] (30) at (-3.000,-1.000) {};
\node [] (31) at (-3.000,-1.500) {};
\node [] (32) at (-3.000,-2.000) {};
\node [shape=rectangle,fill=white,draw,label={[label position=center]$M$},inner xsep=0.1875cm,inner ysep=0.875cm] (33) at (-3.500,-1.250) {};
\node [] (34) at (-3.500,-0.500) {};
\node [] (35) at (-3.500,-1.000) {};
\node [] (36) at (-3.500,-1.500) {};
\node [] (37) at (-3.500,-2.000) {};
\node [] (38) at (-3.500,-0.500) {};
\node [] (39) at (-3.500,-1.000) {};
\node [] (40) at (-3.500,-1.500) {};
\node [] (41) at (-3.500,-2.000) {};
\node [] (42) at (-4.000,-0.500) {};
\node [] (43) at (-4.000,-1.000) {};
\node [] (44) at (-4.000,-1.500) {};
\node [] (45) at (-4.000,-2.000) {};
\end{pgfonlayer}
\begin{pgfonlayer}{edgelayer}
\draw[dashed] (0.center) to[out=180, in=  0] (3.center);
\draw[dashed] (3.center) to[out=-127, in= 53] (4.center);
\draw[dashed] (1.center) to[out=180, in=  0] (5.center);
\draw[dashed] (5.center) to[out=180, in=  0] (6.center);
\draw[dashed] (2.center) to[out=180, in=  0] (7.center);
\draw[dashed] (7.center) to[out=127, in=-52.9] (8.center);
\draw[dashed] (8.center) to[out=180, in=  0] (13.center) to[out=180, in=  0] (14.center);
\draw[dashed] (6.center) to[out=180, in=  0] (20.center);
\draw[dashed] (4.center) to[out=180, in=  0] (21.center);
\draw[dashed] (9.center) to[out=180, in=  0] (10.center) to[out=180, in=  0] (22.center);
\draw[] (11.center) to[out=180, in=  0] (12.center) to[out=180, in=  0] (23.center);
\draw[dashed] (16.center) to[out=180, in=  0] (24.center);
\draw[dashed] (17.center) to[out=180, in=  0] (25.center);
\draw[dashed] (18.center) to[out=180, in=  0] (26.center);
\draw[] (19.center) to[out=180, in=  0] (27.center);
\draw[dashed] (29.center) to[out=180, in=  0] (38.center);
\draw[dashed] (30.center) to[out=180, in=  0] (39.center);
\draw[dashed] (31.center) to[out=180, in=  0] (40.center);
\draw[] (32.center) to[out=180, in=  0] (41.center);
\draw[dashed] (34.center) to[out=180, in=  0] (42.center);
\draw[dashed] (35.center) to[out=180, in=  0] (43.center);
\draw[dashed] (36.center) to[out=180, in=  0] (44.center);
\draw[] (37.center) to[out=180, in=  0] (45.center);
\end{pgfonlayer}
\end{tikzpicture}
 }\hfill    { \begin{tikzpicture}[baseline=(current bounding box.center)]
\begin{pgfonlayer}{nodelayer}
\node [] (0) at (-0.000,-0.250) {};
\node [] (1) at (-0.000,-0.750) {};
\node [] (2) at (-0.000,-1.250) {};
\node [] (3) at (-0.000,-1.750) {};
\node [] (4) at (-0.000,-2.250) {};
\node [] (5) at (-0.250,-1.250) {};
\node [] (6) at (-0.250,-0.250) {};
\node [] (7) at (-0.250,-0.750) {};
\node [] (8) at (-0.250,-1.250) {};
\node [] (9) at (-0.250,-1.750) {};
\node [] (10) at (-0.250,-2.250) {};
\node [] (11) at (-0.250,-0.250) {};
\node [] (12) at (-0.250,-0.750) {};
\node [] (13) at (-0.250,-1.250) {};
\node [] (14) at (-0.250,-1.750) {};
\node [] (15) at (-0.250,-2.250) {};
\node [] (16) at (-0.625,-0.250) {};
\node [] (17) at (-1.375,-1.750) {};
\node [] (18) at (-0.625,-0.750) {};
\node [] (19) at (-1.375,-0.750) {};
\node [] (20) at (-0.625,-1.250) {};
\node [] (21) at (-1.375,-1.250) {};
\node [] (22) at (-0.625,-1.750) {};
\node [] (23) at (-1.375,-0.250) {};
\node [] (24) at (-0.625,-2.250) {};
\node [] (25) at (-1.375,-2.250) {};
\node [] (26) at (-2.000,-0.250) {};
\node [] (27) at (-2.500,-0.250) {};
\node [shape=rectangle,fill=white,draw,label={[label position=center]$M$},inner xsep=0.1875cm,inner ysep=0.875cm] (28) at (-2.000,-1.500) {};
\node [] (29) at (-2.000,-1.000) {};
\node [] (30) at (-2.000,-2.000) {};
\node [] (31) at (-2.000,-0.750) {};
\node [] (32) at (-2.000,-1.250) {};
\node [] (33) at (-2.000,-1.750) {};
\node [] (34) at (-2.000,-2.250) {};
\node [] (35) at (-2.500,-1.000) {};
\node [] (36) at (-2.500,-2.000) {};
\node [label={[label position=center]$\sim$}] (37) at (-2.750,-1.250) {};
\node [] (38) at (-3.000,-0.500) {};
\node [] (39) at (-3.000,-1.000) {};
\node [] (40) at (-3.000,-1.500) {};
\node [] (41) at (-3.000,-2.000) {};
\node [shape=rectangle,fill=white,draw,label={[label position=center]$M$},inner xsep=0.1875cm,inner ysep=0.875cm] (42) at (-3.500,-1.250) {};
\node [] (43) at (-3.500,-0.750) {};
\node [] (44) at (-3.500,-1.750) {};
\node [] (45) at (-3.500,-0.500) {};
\node [] (46) at (-3.500,-1.000) {};
\node [] (47) at (-3.500,-1.500) {};
\node [] (48) at (-3.500,-2.000) {};
\node [] (49) at (-4.000,-0.750) {};
\node [] (50) at (-4.000,-1.750) {};
\end{pgfonlayer}
\begin{pgfonlayer}{edgelayer}
\draw[dashed] (0.center) to[out=180, in=  0] (11.center);
\draw[dashed] (1.center) to[out=180, in=  0] (12.center);
\draw[] (2.center) to[out=180, in=  0] (13.center);
\draw[] (3.center) to[out=180, in=  0] (14.center);
\draw[] (4.center) to[out=180, in=  0] (15.center);
\draw[dashed] (6.center) to[out=180, in=  0] (16.center);
\draw[] (16.center) to[out=-117, in= 63] (17.center);
\draw[dashed] (7.center) to[out=180, in=  0] (18.center);
\draw[dashed] (18.center) to[out=180, in=  0] (19.center);
\draw[] (8.center) to[out=180, in=  0] (20.center);
\draw[] (20.center) to[out=180, in=  0] (21.center);
\draw[] (9.center) to[out=180, in=  0] (22.center);
\draw[dashed] (22.center) to[out=117, in=-62.9] (23.center);
\draw[] (10.center) to[out=180, in=  0] (24.center);
\draw[] (24.center) to[out=180, in=  0] (25.center);
\draw[dashed] (23.center) to[out=180, in=  0] (26.center) to[out=180, in=  0] (27.center);
\draw[dashed] (19.center) to[out=180, in=  0] (31.center);
\draw[] (21.center) to[out=180, in=  0] (32.center);
\draw[] (17.center) to[out=180, in=  0] (33.center);
\draw[] (25.center) to[out=180, in=  0] (34.center);
\draw[dashed] (29.center) to[out=180, in=  0] (35.center);
\draw[] (30.center) to[out=180, in=  0] (36.center);
\draw[dashed] (38.center) to[out=180, in=  0] (45.center);
\draw[] (39.center) to[out=180, in=  0] (46.center);
\draw[] (40.center) to[out=180, in=  0] (47.center);
\draw[] (41.center) to[out=180, in=  0] (48.center);
\draw[dashed] (43.center) to[out=180, in=  0] (49.center);
\draw[] (44.center) to[out=180, in=  0] (50.center);
\end{pgfonlayer}
\end{tikzpicture}
 }
  \item The composition of $[M,p]$ and $[N,q]$ is:
        
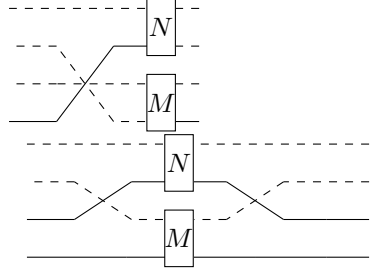
{ \begin{tikzpicture}[baseline=(current bounding box.center)]
\begin{pgfonlayer}{nodelayer}
\node [] (0) at (-0.000,-0.250) {};
\node [] (1) at (-0.000,-0.750) {};
\node [shape=rectangle,fill=white,draw,label={[label position=center]$N$},inner xsep=0.1875cm,inner ysep=0.375cm] (2) at (-0.500,-0.500) {};
\node [] (3) at (-0.500,-0.250) {};
\node [] (4) at (-0.500,-0.750) {};
\node [] (5) at (-0.500,-0.250) {};
\node [] (6) at (-0.500,-0.750) {};
\node [] (7) at (-0.000,-1.250) {};
\node [] (8) at (-0.000,-1.750) {};
\node [shape=rectangle,fill=white,draw,label={[label position=center]$M$},inner xsep=0.1875cm,inner ysep=0.375cm] (9) at (-0.500,-1.500) {};
\node [] (10) at (-0.500,-1.250) {};
\node [] (11) at (-0.500,-1.750) {};
\node [] (12) at (-0.500,-1.250) {};
\node [] (13) at (-0.500,-1.750) {};
\node [] (14) at (-1.333,-0.250) {};
\node [] (15) at (-2.500,-0.250) {};
\node [] (16) at (-1.125,-0.750) {};
\node [] (17) at (-1.875,-1.750) {};
\node [] (18) at (-1.125,-1.250) {};
\node [] (19) at (-1.875,-1.250) {};
\node [] (20) at (-1.125,-1.750) {};
\node [] (21) at (-1.875,-0.750) {};
\node [] (22) at (-2.500,-0.750) {};
\node [] (23) at (-2.500,-1.250) {};
\node [] (24) at (-2.500,-1.750) {};
\end{pgfonlayer}
\begin{pgfonlayer}{edgelayer}
\draw[dashed] (0.center) to[out=180, in=  0] (5.center);
\draw[dashed] (1.center) to[out=180, in=  0] (6.center);
\draw[dashed] (7.center) to[out=180, in=  0] (12.center);
\draw[] (8.center) to[out=180, in=  0] (13.center);
\draw[dashed] (3.center) to[out=180, in=  0] (14.center) to[out=180, in=  0] (15.center);
\draw[] (4.center) to[out=180, in=  0] (16.center);
\draw[] (16.center) to[out=-127, in= 53] (17.center);
\draw[dashed] (10.center) to[out=180, in=  0] (18.center);
\draw[dashed] (18.center) to[out=180, in=  0] (19.center);
\draw[dashed] (11.center) to[out=180, in=  0] (20.center);
\draw[dashed] (20.center) to[out=127, in=-52.9] (21.center);
\draw[dashed] (21.center) to[out=180, in=  0] (22.center);
\draw[dashed] (19.center) to[out=180, in=  0] (23.center);
\draw[] (17.center) to[out=180, in=  0] (24.center);
\end{pgfonlayer}
\end{tikzpicture}
 }
  \item The tensor product of $[M,p]$ and $[N,q]$ is
        
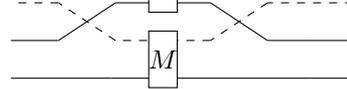
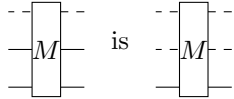
{ \begin{tikzpicture}[baseline=(current bounding box.center)]
\begin{pgfonlayer}{nodelayer}
\node [] (0) at (-0.000,-0.250) {};
\node [] (1) at (-0.500,-0.250) {};
\node [] (2) at (-0.000,-0.750) {};
\node [] (3) at (-0.500,-0.750) {};
\node [] (4) at (-0.000,-1.250) {};
\node [] (5) at (-0.563,-1.250) {};
\node [] (6) at (-0.000,-1.750) {};
\node [] (7) at (-0.563,-1.750) {};
\node [] (8) at (-1.500,-0.250) {};
\node [] (9) at (-1.125,-0.750) {};
\node [] (10) at (-1.875,-1.250) {};
\node [] (11) at (-1.125,-1.250) {};
\node [] (12) at (-1.875,-0.750) {};
\node [] (13) at (-1.500,-1.750) {};
\node [shape=rectangle,fill=white,draw,label={[label position=center]$N$},inner xsep=0.1875cm,inner ysep=0.375cm] (14) at (-2.500,-0.500) {};
\node [] (15) at (-2.500,-0.250) {};
\node [] (16) at (-2.500,-0.750) {};
\node [] (17) at (-2.500,-0.250) {};
\node [] (18) at (-2.500,-0.750) {};
\node [shape=rectangle,fill=white,draw,label={[label position=center]$M$},inner xsep=0.1875cm,inner ysep=0.375cm] (19) at (-2.500,-1.500) {};
\node [] (20) at (-2.500,-1.250) {};
\node [] (21) at (-2.500,-1.750) {};
\node [] (22) at (-2.500,-1.250) {};
\node [] (23) at (-2.500,-1.750) {};
\node [] (24) at (-3.333,-0.250) {};
\node [] (25) at (-4.500,-0.250) {};
\node [] (26) at (-3.125,-0.750) {};
\node [] (27) at (-3.875,-1.250) {};
\node [] (28) at (-3.125,-1.250) {};
\node [] (29) at (-3.875,-0.750) {};
\node [] (30) at (-4.500,-0.750) {};
\node [] (31) at (-4.500,-1.250) {};
\node [] (32) at (-3.188,-1.750) {};
\node [] (33) at (-4.500,-1.750) {};
\end{pgfonlayer}
\begin{pgfonlayer}{edgelayer}
\draw[] (4.center) to[out=180, in=  0] (5.center);
\draw[] (6.center) to[out=180, in=  0] (7.center);
\draw[dashed] (2.center) to[out=180, in=  0] (3.center) to[out=180, in=  0] (9.center);
\draw[dashed] (9.center) to[out=-146, in= 34] (10.center);
\draw[] (5.center) to[out=180, in=  0] (11.center);
\draw[] (11.center) to[out=146, in=-33.9] (12.center);
\draw[] (7.center) to[out=180, in=  0] (13.center);
\draw[dashed] (0.center) to[out=180, in=  0] (1.center) to[out=180, in=  0] (8.center) to[out=180, in=  0] (17.center);
\draw[] (12.center) to[out=180, in=  0] (18.center);
\draw[dashed] (10.center) to[out=180, in=  0] (22.center);
\draw[] (13.center) to[out=180, in=  0] (23.center);
\draw[dashed] (15.center) to[out=180, in=  0] (24.center) to[out=180, in=  0] (25.center);
\draw[] (16.center) to[out=180, in=  0] (26.center);
\draw[] (26.center) to[out=-146, in= 34] (27.center);
\draw[dashed] (20.center) to[out=180, in=  0] (28.center);
\draw[dashed] (28.center) to[out=146, in=-33.9] (29.center);
\draw[dashed] (29.center) to[out=180, in=  0] (30.center);
\draw[] (27.center) to[out=180, in=  0] (31.center);
\draw[] (21.center) to[out=180, in=  0] (32.center);
\draw[] (32.center) to[out=180, in=  0] (33.center);
\end{pgfonlayer}
\end{tikzpicture}
 }
  \item The trace of         { \begin{tikzpicture}[baseline=(current bounding box.center)]
\begin{pgfonlayer}{nodelayer}
\node [] (0) at (-0.000,-0.250) {};
\node [] (1) at (-0.000,-0.750) {};
\node [] (2) at (-0.000,-1.250) {};
\node [shape=rectangle,fill=white,draw,label={[label position=center]$M$},inner xsep=0.1875cm,inner ysep=0.625cm] (3) at (-0.500,-0.750) {};
\node [] (4) at (-0.500,-0.250) {};
\node [] (5) at (-0.500,-0.750) {};
\node [] (6) at (-0.500,-1.250) {};
\node [] (7) at (-0.500,-0.250) {};
\node [] (8) at (-0.500,-0.750) {};
\node [] (9) at (-0.500,-1.250) {};
\node [] (10) at (-1.000,-0.250) {};
\node [] (11) at (-1.000,-0.750) {};
\node [] (12) at (-1.000,-1.250) {};
\end{pgfonlayer}
\begin{pgfonlayer}{edgelayer}
\draw[dashed] (0.center) to[out=180, in=  0] (7.center);
\draw[] (1.center) to[out=180, in=  0] (8.center);
\draw[] (2.center) to[out=180, in=  0] (9.center);
\draw[dashed] (4.center) to[out=180, in=  0] (10.center);
\draw[] (5.center) to[out=180, in=  0] (11.center);
\draw[] (6.center) to[out=180, in=  0] (12.center);
\end{pgfonlayer}
\end{tikzpicture}
 } is         { \begin{tikzpicture}[baseline=(current bounding box.center)]
\begin{pgfonlayer}{nodelayer}
\node [] (0) at (-0.000,-0.250) {};
\node [] (1) at (-0.000,-0.750) {};
\node [] (2) at (-0.000,-1.250) {};
\node [shape=rectangle,fill=white,draw,label={[label position=center]$M$},inner xsep=0.1875cm,inner ysep=0.625cm] (3) at (-0.500,-0.750) {};
\node [] (4) at (-0.500,-0.250) {};
\node [] (5) at (-0.500,-0.750) {};
\node [] (6) at (-0.500,-1.250) {};
\node [] (7) at (-0.500,-0.250) {};
\node [] (8) at (-0.500,-0.750) {};
\node [] (9) at (-0.500,-1.250) {};
\node [] (10) at (-1.000,-0.250) {};
\node [] (11) at (-1.000,-0.750) {};
\node [] (12) at (-1.000,-1.250) {};
\end{pgfonlayer}
\begin{pgfonlayer}{edgelayer}
\draw[dashed] (0.center) to[out=180, in=  0] (7.center);
\draw[dashed] (1.center) to[out=180, in=  0] (8.center);
\draw[] (2.center) to[out=180, in=  0] (9.center);
\draw[dashed] (4.center) to[out=180, in=  0] (10.center);
\draw[dashed] (5.center) to[out=180, in=  0] (11.center);
\draw[] (6.center) to[out=180, in=  0] (12.center);
\end{pgfonlayer}
\end{tikzpicture}
 }

\end{itemize}    

  \begin{proof}
    The main idea is that:
    \begin{itemize}
    \item $Q$ is indeed a traced prop. In particular the two
      compositions are well defined. The proof is not interesting and given in appendix \ref{app:proofalacon}
    \item The equations we took for defining the equivalence class, and the two products are the good one, as illustrated by the diagrammatic equations below.
    \end{itemize}
    Let $\widehat{P}$ be the traced completion of $P$. 
    First, consider the subcategory $R$ of $\widehat{P}$ consisting of all morphisms of the form $Tr_k \iota(M)$ for $M$ a morphism of $P$.
    It is easy to see that $R$ is actually a traced prop. By the universal property, $R = \widehat{P}$, therefore all morphisms of $\widehat{P}$ can be written in the form 
    $Tr_k\ \iota(M)$ for $M$ a morphism of $P$.
    
    Let $M$ be any morphism in a traced prop. All the following equations are true in any traced prop

        
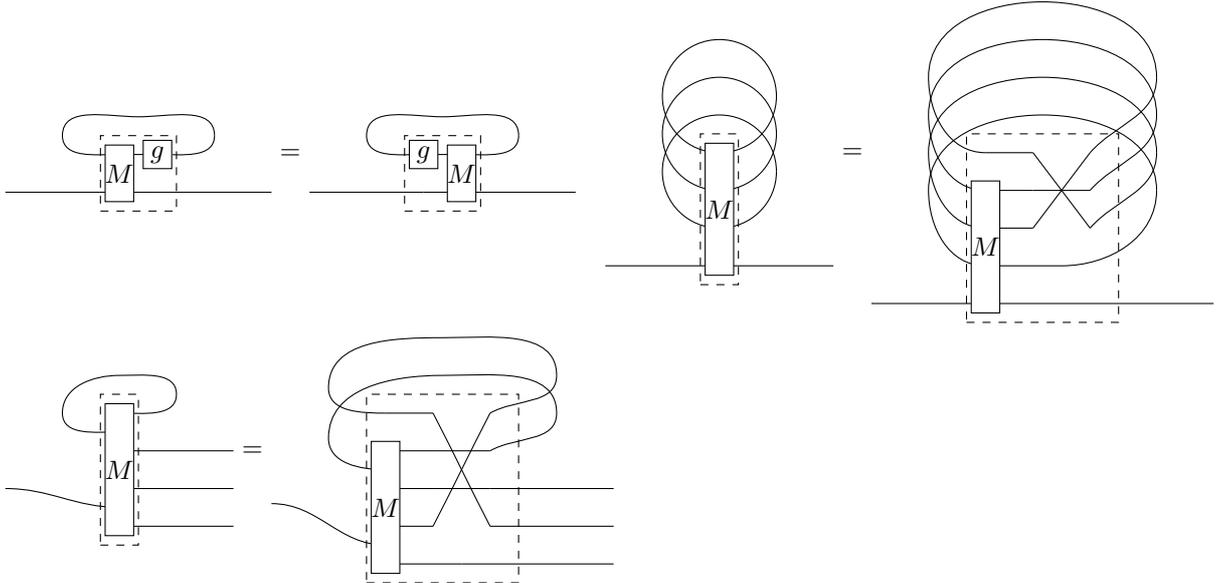
{ \begin{tikzpicture}[baseline=(current bounding box.center)]
\begin{pgfonlayer}{nodelayer}
\node [] (0) at (-0.250,-0.500) {};
\node [] (1) at (-0.000,-1.250) {};
\node [] (2) at (-0.750,-0.500) {};
\node [] (3) at (-2.750,-0.500) {};
\path[dashed=None,draw=,color=black] (-1.250,-0.500) rectangle (-2.250,-1.500);
\node at (-1.750,-1.000) {};
\node [] (4) at (-1.750,-0.250) {};
\node [shape=rectangle,fill=white,draw,label={[label position=center]$M$},inner xsep=0.1875cm,inner ysep=0.375cm] (5) at (-1.500,-1.000) {};
\node [] (6) at (-1.500,-0.750) {};
\node [] (7) at (-1.500,-1.250) {};
\node [] (8) at (-1.500,-0.750) {};
\node [] (9) at (-1.500,-1.250) {};
\node [shape=rectangle,fill=white,draw,label={[label position=center]$g$},inner xsep=0.1875cm,inner ysep=0.1875cm] (10) at (-2.000,-0.750) {};
\node [] (11) at (-2.000,-0.750) {};
\node [] (12) at (-2.000,-0.750) {};
\node [] (13) at (-2.000,-1.250) {};
\node [] (14) at (-3.250,-0.500) {};
\node [] (15) at (-3.500,-1.250) {};
\node [label={[label position=center]=}] (16) at (-3.750,-0.750) {};
\node [] (17) at (-4.250,-0.500) {};
\node [] (18) at (-4.000,-1.250) {};
\node [] (19) at (-4.750,-0.500) {};
\node [] (20) at (-6.750,-0.500) {};
\path[dashed=None,draw=,color=black] (-5.250,-0.500) rectangle (-6.250,-1.500);
\node at (-5.750,-1.000) {};
\node [] (21) at (-5.750,-0.250) {};
\node [shape=rectangle,fill=white,draw,label={[label position=center]$g$},inner xsep=0.1875cm,inner ysep=0.1875cm] (22) at (-5.500,-0.750) {};
\node [] (23) at (-5.500,-0.750) {};
\node [] (24) at (-5.500,-0.750) {};
\node [] (25) at (-5.500,-1.250) {};
\node [shape=rectangle,fill=white,draw,label={[label position=center]$M$},inner xsep=0.1875cm,inner ysep=0.375cm] (26) at (-6.000,-1.000) {};
\node [] (27) at (-6.000,-0.750) {};
\node [] (28) at (-6.000,-1.250) {};
\node [] (29) at (-6.000,-0.750) {};
\node [] (30) at (-6.000,-1.250) {};
\node [] (31) at (-7.250,-0.500) {};
\node [] (32) at (-7.500,-1.250) {};
\end{pgfonlayer}
\begin{pgfonlayer}{edgelayer}
\draw[] (2.center) to[out=-89.9, in=  0] (8.center);
\draw[] (1.center) to[out=180, in=  0] (9.center);
\draw[] (6.center) to[out=180, in=  0] (12.center);
\draw[] (7.center) to[out=180, in=  0] (13.center);
\draw[] (11.center) to[out=180, in=-89.9] (3.center) to[out= 90, in=180] (4.center) to[out=  0, in= 90] (2.center);
\draw[] (13.center) to[out=180, in=  0] (15.center);
\draw[] (19.center) to[out=-89.9, in=  0] (24.center);
\draw[] (18.center) to[out=180, in=  0] (25.center);
\draw[] (23.center) to[out=180, in=  0] (29.center);
\draw[] (25.center) to[out=180, in=  0] (30.center);
\draw[] (27.center) to[out=180, in=-89.9] (20.center) to[out= 90, in=180] (21.center) to[out=  0, in= 90] (19.center);
\draw[] (28.center) to[out=180, in=  0] (32.center);
\end{pgfonlayer}
\end{tikzpicture}
 }\hfill
        { \begin{tikzpicture}[baseline=(current bounding box.center)]
\begin{pgfonlayer}{nodelayer}
\node [] (0) at (-0.250,-2.000) {};
\node [] (1) at (-0.000,-4.250) {};
\node [] (2) at (-0.750,-1.250) {};
\node [] (3) at (-0.750,-1.750) {};
\node [] (4) at (-0.750,-2.250) {};
\node [] (5) at (-0.750,-2.750) {};
\node [] (6) at (-3.750,-1.250) {};
\node [] (7) at (-3.750,-1.750) {};
\node [] (8) at (-3.750,-2.250) {};
\node [] (9) at (-3.750,-2.750) {};
\path[dashed=None,draw=,color=black] (-1.250,-2.000) rectangle (-3.250,-4.500);
\node at (-2.250,-3.250) {};
\node [] (10) at (-2.250,-0.250) {};
\node [] (11) at (-2.250,-0.750) {};
\node [] (12) at (-2.250,-1.250) {};
\node [] (13) at (-2.250,-1.750) {};
\node [] (14) at (-1.625,-2.250) {};
\node [] (15) at (-2.375,-3.250) {};
\node [] (16) at (-1.625,-2.750) {};
\node [] (17) at (-2.375,-2.750) {};
\node [] (18) at (-1.625,-3.250) {};
\node [] (19) at (-2.375,-2.250) {};
\node [] (20) at (-2.000,-3.750) {};
\node [] (21) at (-2.000,-4.250) {};
\node [] (22) at (-3.000,-2.250) {};
\node [shape=rectangle,fill=white,draw,label={[label position=center]$M$},inner xsep=0.1875cm,inner ysep=0.875cm] (23) at (-3.000,-3.500) {};
\node [] (24) at (-3.000,-2.750) {};
\node [] (25) at (-3.000,-3.250) {};
\node [] (26) at (-3.000,-3.750) {};
\node [] (27) at (-3.000,-4.250) {};
\node [] (28) at (-3.000,-2.750) {};
\node [] (29) at (-3.000,-3.250) {};
\node [] (30) at (-3.000,-3.750) {};
\node [] (31) at (-3.000,-4.250) {};
\node [] (32) at (-4.250,-2.000) {};
\node [] (33) at (-4.500,-4.250) {};
\node [label={[label position=center]=}] (34) at (-4.750,-2.250) {};
\node [] (35) at (-5.250,-2.000) {};
\node [] (36) at (-5.000,-3.750) {};
\node [] (37) at (-5.750,-1.500) {};
\node [] (38) at (-5.750,-2.000) {};
\node [] (39) at (-5.750,-2.500) {};
\node [] (40) at (-7.250,-1.500) {};
\node [] (41) at (-7.250,-2.000) {};
\node [] (42) at (-7.250,-2.500) {};
\path[dashed=None,draw=,color=black] (-6.250,-2.000) rectangle (-6.750,-4.000);
\node at (-6.500,-3.000) {};
\node [] (43) at (-6.500,-0.750) {};
\node [] (44) at (-6.500,-1.250) {};
\node [] (45) at (-6.500,-1.750) {};
\node [shape=rectangle,fill=white,draw,label={[label position=center]$M$},inner xsep=0.1875cm,inner ysep=0.875cm] (46) at (-6.500,-3.000) {};
\node [] (47) at (-6.500,-2.250) {};
\node [] (48) at (-6.500,-2.750) {};
\node [] (49) at (-6.500,-3.250) {};
\node [] (50) at (-6.500,-3.750) {};
\node [] (51) at (-6.500,-2.250) {};
\node [] (52) at (-6.500,-2.750) {};
\node [] (53) at (-6.500,-3.250) {};
\node [] (54) at (-6.500,-3.750) {};
\node [] (55) at (-7.750,-2.000) {};
\node [] (56) at (-8.000,-3.750) {};
\end{pgfonlayer}
\begin{pgfonlayer}{edgelayer}
\draw[] (2.center) to[out=-89.9, in= 49] (14.center);
\draw[] (14.center) to[out=-127, in= 53] (15.center);
\draw[] (3.center) to[out=-89.9, in= 49] (16.center);
\draw[] (16.center) to[out=180, in=  0] (17.center);
\draw[] (4.center) to[out=-89.9, in= 49] (18.center);
\draw[] (18.center) to[out=127, in=-52.9] (19.center);
\draw[] (5.center) to[out=-89.9, in=  0] (20.center);
\draw[] (1.center) to[out=180, in=  0] (21.center);
\draw[] (19.center) to[out=180, in=  0] (22.center);
\draw[] (17.center) to[out=180, in=  0] (28.center);
\draw[] (15.center) to[out=180, in=  0] (29.center);
\draw[] (20.center) to[out=180, in=  0] (30.center);
\draw[] (21.center) to[out=180, in=  0] (31.center);
\draw[] (22.center) to[out=180, in=-89.9] (6.center) to[out= 90, in=180] (10.center) to[out=  0, in= 90] (2.center);
\draw[] (24.center) to[out=180, in=-89.9] (7.center) to[out= 90, in=180] (11.center) to[out=  0, in= 90] (3.center);
\draw[] (25.center) to[out=180, in=-89.9] (8.center) to[out= 90, in=180] (12.center) to[out=  0, in= 90] (4.center);
\draw[] (26.center) to[out=180, in=-89.9] (9.center) to[out= 90, in=180] (13.center) to[out=  0, in= 90] (5.center);
\draw[] (27.center) to[out=180, in=  0] (33.center);
\draw[] (37.center) to[out=-89.9, in=  0] (51.center);
\draw[] (38.center) to[out=-89.9, in=  0] (52.center);
\draw[] (39.center) to[out=-89.9, in=  0] (53.center);
\draw[] (36.center) to[out=180, in=  0] (54.center);
\draw[] (47.center) to[out=180, in=-89.9] (40.center) to[out= 90, in=180] (43.center) to[out=  0, in= 90] (37.center);
\draw[] (48.center) to[out=180, in=-89.9] (41.center) to[out= 90, in=180] (44.center) to[out=  0, in= 90] (38.center);
\draw[] (49.center) to[out=180, in=-89.9] (42.center) to[out= 90, in=180] (45.center) to[out=  0, in= 90] (39.center);
\draw[] (50.center) to[out=180, in=  0] (56.center);
\end{pgfonlayer}
\end{tikzpicture}
 }

        { \begin{tikzpicture}[baseline=(current bounding box.center)]
\begin{pgfonlayer}{nodelayer}
\node [] (0) at (-0.250,-1.000) {};
\node [] (1) at (-0.000,-2.250) {};
\node [] (2) at (-0.000,-2.750) {};
\node [] (3) at (-0.000,-3.250) {};
\node [] (4) at (-0.750,-0.750) {};
\node [] (5) at (-0.750,-1.250) {};
\node [] (6) at (-3.750,-0.917) {};
\node [] (7) at (-3.750,-1.583) {};
\path[dashed=None,draw=,color=black] (-1.250,-1.000) rectangle (-3.250,-3.500);
\node at (-2.250,-2.250) {};
\node [] (8) at (-2.250,-0.250) {};
\node [] (9) at (-2.250,-0.750) {};
\node [] (10) at (-1.625,-1.250) {};
\node [] (11) at (-2.375,-2.750) {};
\node [] (12) at (-1.625,-1.750) {};
\node [] (13) at (-2.375,-1.750) {};
\node [] (14) at (-1.625,-2.250) {};
\node [] (15) at (-2.375,-2.250) {};
\node [] (16) at (-1.625,-2.750) {};
\node [] (17) at (-2.375,-1.250) {};
\node [] (18) at (-2.000,-3.250) {};
\node [] (19) at (-3.000,-1.250) {};
\node [shape=rectangle,fill=white,draw,label={[label position=center]$M$},inner xsep=0.1875cm,inner ysep=0.875cm] (20) at (-3.000,-2.500) {};
\node [] (21) at (-3.000,-2.000) {};
\node [] (22) at (-3.000,-3.000) {};
\node [] (23) at (-3.000,-1.750) {};
\node [] (24) at (-3.000,-2.250) {};
\node [] (25) at (-3.000,-2.750) {};
\node [] (26) at (-3.000,-3.250) {};
\node [] (27) at (-4.250,-1.050) {};
\node [] (28) at (-4.500,-2.450) {};
\node [] (29) at (-4.250,-3.150) {};
\node [label={[label position=center]=}] (30) at (-4.750,-1.750) {};
\node [] (31) at (-5.250,-1.000) {};
\node [] (32) at (-5.000,-1.750) {};
\node [] (33) at (-5.000,-2.250) {};
\node [] (34) at (-5.000,-2.750) {};
\node [] (35) at (-5.750,-1.000) {};
\node [] (36) at (-7.250,-1.250) {};
\path[dashed=None,draw=,color=black] (-6.250,-1.000) rectangle (-6.750,-3.000);
\node at (-6.500,-2.000) {};
\node [] (37) at (-6.500,-0.750) {};
\node [shape=rectangle,fill=white,draw,label={[label position=center]$M$},inner xsep=0.1875cm,inner ysep=0.875cm] (38) at (-6.500,-2.000) {};
\node [] (39) at (-6.500,-1.500) {};
\node [] (40) at (-6.500,-2.500) {};
\node [] (41) at (-6.500,-1.250) {};
\node [] (42) at (-6.500,-1.750) {};
\node [] (43) at (-6.500,-2.250) {};
\node [] (44) at (-6.500,-2.750) {};
\node [] (45) at (-7.750,-1.250) {};
\node [] (46) at (-8.000,-2.250) {};
\node [] (47) at (-7.750,-2.750) {};
\end{pgfonlayer}
\begin{pgfonlayer}{edgelayer}
\draw[] (4.center) to[out=-89.9, in= 30] (10.center);
\draw[] (10.center) to[out=-117, in= 63] (11.center);
\draw[] (5.center) to[out=-89.9, in= 30] (12.center);
\draw[] (12.center) to[out=180, in=  0] (13.center);
\draw[] (1.center) to[out=180, in=  0] (14.center);
\draw[] (14.center) to[out=180, in=  0] (15.center);
\draw[] (2.center) to[out=180, in=  0] (16.center);
\draw[] (16.center) to[out=117, in=-62.9] (17.center);
\draw[] (3.center) to[out=180, in=  0] (18.center);
\draw[] (17.center) to[out=180, in=  0] (19.center);
\draw[] (13.center) to[out=180, in=  0] (23.center);
\draw[] (15.center) to[out=180, in=  0] (24.center);
\draw[] (11.center) to[out=180, in=  0] (25.center);
\draw[] (18.center) to[out=180, in=  0] (26.center);
\draw[] (19.center) to[out=180, in=-89.9] (6.center) to[out= 90, in=180] (8.center) to[out=  0, in= 90] (4.center);
\draw[] (21.center) to[out=180, in=-89.9] (7.center) to[out= 90, in=180] (9.center) to[out=  0, in= 90] (5.center);
\draw[] (22.center) to[out=180, in=  0] (28.center);
\draw[] (35.center) to[out=-89.9, in=  0] (41.center);
\draw[] (32.center) to[out=180, in=  0] (42.center);
\draw[] (33.center) to[out=180, in=  0] (43.center);
\draw[] (34.center) to[out=180, in=  0] (44.center);
\draw[] (39.center) to[out=180, in=-89.9] (36.center) to[out= 90, in=180] (37.center) to[out=  0, in= 90] (35.center);
\draw[] (40.center) to[out=180, in=  0] (46.center);
\end{pgfonlayer}
\end{tikzpicture}
 }

        This implies that the map $F: [M,k] \mapsto Tr_k\ \iota(M)$ from $Q$ to $\widehat{P}$ is well defined, and surjective.

        Now the following equations are true in any traced prop:
        
        
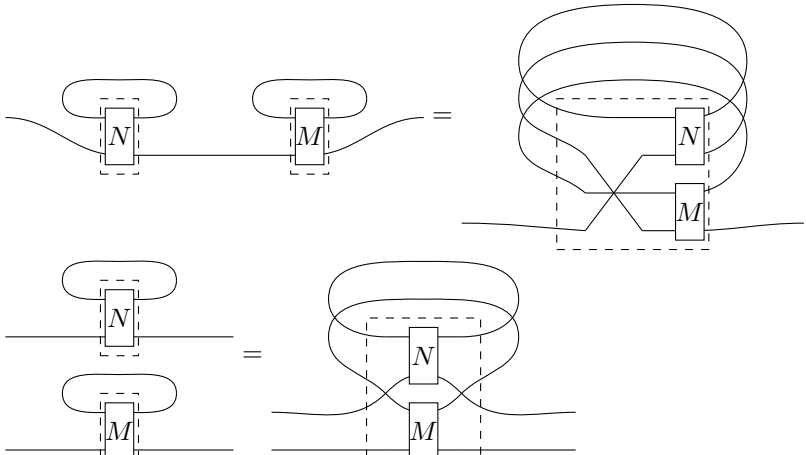
{ \begin{tikzpicture}[baseline=(current bounding box.center)]
\begin{pgfonlayer}{nodelayer}
\node [] (0) at (-0.250,-1.400) {};
\node [] (1) at (-0.000,-3.150) {};
\node [] (2) at (-0.750,-1.000) {};
\node [] (3) at (-0.750,-1.500) {};
\node [] (4) at (-0.750,-2.000) {};
\node [] (5) at (-3.750,-1.000) {};
\node [] (6) at (-3.750,-1.500) {};
\node [] (7) at (-3.750,-2.000) {};
\path[dashed=None,draw=,color=black] (-1.250,-1.500) rectangle (-3.250,-3.500);
\node at (-2.250,-2.500) {};
\node [] (8) at (-2.250,-0.250) {};
\node [] (9) at (-2.250,-0.750) {};
\node [] (10) at (-2.250,-1.250) {};
\node [shape=rectangle,fill=white,draw,label={[label position=center]$N$},inner xsep=0.1875cm,inner ysep=0.375cm] (11) at (-1.500,-2.000) {};
\node [] (12) at (-1.500,-1.750) {};
\node [] (13) at (-1.500,-2.250) {};
\node [] (14) at (-1.500,-1.750) {};
\node [] (15) at (-1.500,-2.250) {};
\node [shape=rectangle,fill=white,draw,label={[label position=center]$M$},inner xsep=0.1875cm,inner ysep=0.375cm] (16) at (-1.500,-3.000) {};
\node [] (17) at (-1.500,-2.750) {};
\node [] (18) at (-1.500,-3.250) {};
\node [] (19) at (-1.500,-2.750) {};
\node [] (20) at (-1.500,-3.250) {};
\node [] (21) at (-2.500,-1.750) {};
\node [] (22) at (-2.125,-2.250) {};
\node [] (23) at (-2.875,-3.250) {};
\node [] (24) at (-2.125,-2.750) {};
\node [] (25) at (-2.875,-2.750) {};
\node [] (26) at (-2.125,-3.250) {};
\node [] (27) at (-2.875,-2.250) {};
\node [] (28) at (-4.250,-1.400) {};
\node [] (29) at (-4.500,-3.150) {};
\node [label={[label position=center]=}] (30) at (-4.750,-1.750) {};
\node [] (31) at (-5.000,-1.750) {};
\node [] (32) at (-5.750,-1.500) {};
\node [] (33) at (-7.250,-1.500) {};
\path[dashed=None,draw=,color=black] (-6.250,-1.500) rectangle (-6.750,-2.500);
\node at (-6.500,-2.000) {};
\node [] (34) at (-6.500,-1.250) {};
\node [shape=rectangle,fill=white,draw,label={[label position=center]$M$},inner xsep=0.1875cm,inner ysep=0.375cm] (35) at (-6.500,-2.000) {};
\node [] (36) at (-6.500,-1.750) {};
\node [] (37) at (-6.500,-2.250) {};
\node [] (38) at (-6.500,-1.750) {};
\node [] (39) at (-6.500,-2.250) {};
\node [] (40) at (-8.250,-1.500) {};
\node [] (41) at (-9.750,-1.500) {};
\path[dashed=None,draw=,color=black] (-8.750,-1.500) rectangle (-9.250,-2.500);
\node at (-9.000,-2.000) {};
\node [] (42) at (-9.000,-1.250) {};
\node [shape=rectangle,fill=white,draw,label={[label position=center]$N$},inner xsep=0.1875cm,inner ysep=0.375cm] (43) at (-9.000,-2.000) {};
\node [] (44) at (-9.000,-1.750) {};
\node [] (45) at (-9.000,-2.250) {};
\node [] (46) at (-9.000,-1.750) {};
\node [] (47) at (-9.000,-2.250) {};
\node [] (48) at (-10.500,-1.750) {};
\end{pgfonlayer}
\begin{pgfonlayer}{edgelayer}
\draw[] (2.center) to[out=-89.9, in=  0] (14.center);
\draw[] (3.center) to[out=-89.9, in=  0] (15.center);
\draw[] (4.center) to[out=-89.9, in=  0] (19.center);
\draw[] (1.center) to[out=180, in=  0] (20.center);
\draw[] (12.center) to[out=180, in=  0] (21.center);
\draw[] (13.center) to[out=180, in=  0] (22.center);
\draw[] (22.center) to[out=-127, in= 53] (23.center);
\draw[] (17.center) to[out=180, in=  0] (24.center);
\draw[] (24.center) to[out=180, in=  0] (25.center);
\draw[] (18.center) to[out=180, in=  0] (26.center);
\draw[] (26.center) to[out=127, in=-52.9] (27.center);
\draw[] (21.center) to[out=180, in=-89.9] (5.center) to[out= 90, in=180] (8.center) to[out=  0, in= 90] (2.center);
\draw[] (27.center) to[out=139, in=-89.9] (6.center) to[out= 90, in=180] (9.center) to[out=  0, in= 90] (3.center);
\draw[] (25.center) to[out=139, in=-89.9] (7.center) to[out= 90, in=180] (10.center) to[out=  0, in= 90] (4.center);
\draw[] (23.center) to[out=176, in=  0] (29.center);
\draw[] (32.center) to[out=-89.9, in=  0] (38.center);
\draw[] (31.center) to[out=180, in=  0] (39.center);
\draw[] (36.center) to[out=180, in=-89.9] (33.center) to[out= 90, in=180] (34.center) to[out=  0, in= 90] (32.center);
\draw[] (40.center) to[out=-89.9, in=  0] (46.center);
\draw[] (37.center) to[out=180, in=  0] (47.center);
\draw[] (44.center) to[out=180, in=-89.9] (41.center) to[out= 90, in=180] (42.center) to[out=  0, in= 90] (40.center);
\draw[] (45.center) to[out=180, in=  0] (48.center);
\end{pgfonlayer}
\end{tikzpicture}
 }

        { \begin{tikzpicture}[baseline=(current bounding box.center)]
\begin{pgfonlayer}{nodelayer}
\node [] (0) at (-0.250,-1.000) {};
\node [] (1) at (-0.000,-2.250) {};
\node [] (2) at (-0.000,-2.750) {};
\node [] (3) at (-0.750,-0.750) {};
\node [] (4) at (-0.750,-1.250) {};
\node [] (5) at (-3.250,-0.750) {};
\node [] (6) at (-3.250,-1.250) {};
\path[dashed=None,draw=,color=black] (-1.250,-1.000) rectangle (-2.750,-3.000);
\node at (-2.000,-2.000) {};
\node [] (7) at (-2.000,-0.250) {};
\node [] (8) at (-2.000,-0.750) {};
\node [] (9) at (-1.500,-1.250) {};
\node [] (10) at (-1.500,-2.000) {};
\node [] (11) at (-1.500,-2.750) {};
\node [shape=rectangle,fill=white,draw,label={[label position=center]$N$},inner xsep=0.1875cm,inner ysep=0.375cm] (12) at (-2.000,-1.500) {};
\node [] (13) at (-2.000,-1.250) {};
\node [] (14) at (-2.000,-1.750) {};
\node [] (15) at (-2.000,-1.250) {};
\node [] (16) at (-2.000,-1.750) {};
\node [shape=rectangle,fill=white,draw,label={[label position=center]$M$},inner xsep=0.1875cm,inner ysep=0.375cm] (17) at (-2.000,-2.500) {};
\node [] (18) at (-2.000,-2.250) {};
\node [] (19) at (-2.000,-2.750) {};
\node [] (20) at (-2.000,-2.250) {};
\node [] (21) at (-2.000,-2.750) {};
\node [] (22) at (-2.500,-1.250) {};
\node [] (23) at (-2.500,-2.000) {};
\node [] (24) at (-2.500,-2.750) {};
\node [] (25) at (-3.750,-1.000) {};
\node [] (26) at (-4.000,-2.250) {};
\node [] (27) at (-4.000,-2.750) {};
\node [label={[label position=center]=}] (28) at (-4.250,-1.500) {};
\node [] (29) at (-4.750,-0.500) {};
\node [] (30) at (-4.500,-1.250) {};
\node [] (31) at (-4.750,-2.000) {};
\node [] (32) at (-4.500,-2.750) {};
\node [] (33) at (-5.250,-0.500) {};
\node [] (34) at (-6.750,-0.500) {};
\path[dashed=None,draw=,color=black] (-5.750,-0.500) rectangle (-6.250,-1.500);
\node at (-6.000,-1.000) {};
\node [] (35) at (-6.000,-0.250) {};
\node [shape=rectangle,fill=white,draw,label={[label position=center]$N$},inner xsep=0.1875cm,inner ysep=0.375cm] (36) at (-6.000,-1.000) {};
\node [] (37) at (-6.000,-0.750) {};
\node [] (38) at (-6.000,-1.250) {};
\node [] (39) at (-6.000,-0.750) {};
\node [] (40) at (-6.000,-1.250) {};
\node [] (41) at (-5.250,-2.000) {};
\node [] (42) at (-6.750,-2.000) {};
\path[dashed=None,draw=,color=black] (-5.750,-2.000) rectangle (-6.250,-3.000);
\node at (-6.000,-2.500) {};
\node [] (43) at (-6.000,-1.750) {};
\node [shape=rectangle,fill=white,draw,label={[label position=center]$M$},inner xsep=0.1875cm,inner ysep=0.375cm] (44) at (-6.000,-2.500) {};
\node [] (45) at (-6.000,-2.250) {};
\node [] (46) at (-6.000,-2.750) {};
\node [] (47) at (-6.000,-2.250) {};
\node [] (48) at (-6.000,-2.750) {};
\node [] (49) at (-7.250,-0.500) {};
\node [] (50) at (-7.500,-1.250) {};
\node [] (51) at (-7.250,-2.000) {};
\node [] (52) at (-7.500,-2.750) {};
\end{pgfonlayer}
\begin{pgfonlayer}{edgelayer}
\draw[] (3.center) to[out=-89.9, in=  0] (9.center);
\draw[] (2.center) to[out=180, in=  0] (11.center);
\draw[] (9.center) to[out=180, in=  0] (15.center);
\draw[] (1.center) to[out=180, in=-44.9] (10.center) to[out=135, in=  0] (16.center);
\draw[] (4.center) to[out=-89.9, in= 45] (10.center) to[out=-135, in=  0] (20.center);
\draw[] (11.center) to[out=180, in=  0] (21.center);
\draw[] (13.center) to[out=180, in=  0] (22.center);
\draw[] (19.center) to[out=180, in=  0] (24.center);
\draw[] (22.center) to[out=180, in=-89.9] (5.center) to[out= 90, in=180] (7.center) to[out=  0, in= 90] (3.center);
\draw[] (18.center) to[out=180, in=-44.9] (23.center) to[out=135, in=-89.9] (6.center) to[out= 90, in=180] (8.center) to[out=  0, in= 90] (4.center);
\draw[] (14.center) to[out=180, in= 45] (23.center) to[out=-135, in=  0] (26.center);
\draw[] (24.center) to[out=180, in=  0] (27.center);
\draw[] (33.center) to[out=-89.9, in=  0] (39.center);
\draw[] (30.center) to[out=180, in=  0] (40.center);
\draw[] (37.center) to[out=180, in=-89.9] (34.center) to[out= 90, in=180] (35.center) to[out=  0, in= 90] (33.center);
\draw[] (41.center) to[out=-89.9, in=  0] (47.center);
\draw[] (32.center) to[out=180, in=  0] (48.center);
\draw[] (45.center) to[out=180, in=-89.9] (42.center) to[out= 90, in=180] (43.center) to[out=  0, in= 90] (41.center);
\draw[] (38.center) to[out=180, in=  0] (50.center);
\draw[] (46.center) to[out=180, in=  0] (52.center);
\end{pgfonlayer}
\end{tikzpicture}
 }

        This implies $(a)$ that the map $ M \mapsto [M,0]$ is a prop functor from $P$ and $Q$ and $(b)$ that $F(A B) = F(A) F(B)$ and $F(A \otimes B) = F(A) \otimes F(B)$. Along with $F(0) = 0$, $F(id) = id$, $F(\sigma) = \sigma$ and $tr F(f) = F(tr f)$, we obtain that $F$ is actually a (surjective) traced prop functor from our prop $Q$ to $\widehat{P}$.

        Now by the universal property, we have another functor $G$ from  $\widehat{P}$ to $Q$, satisfying $G \iota(M) = [M,0]$ for all  morphisms $M$ in $P$.

        This implies that $FG$ is the identity on $\iota(M)$ and therefore that $FG$ is the identity functor by the universal property of $\widehat{P}$. $F$ is therefore surjective and injective, and therefore an isomorphism.
\end{proof}    

  \subsection{Flow equivalence and strong shift equivalence}
  \begin{definition}
    Let $\mathcal{R}$ be a semiring.
    We say that $M$ and $N$ are Parry-Sullivan-equivalent, for short PS-equivalent, in $\mathcal{R}$ if $M \simeq N$ where $\simeq$ is the
    smallest equivalence relation s.t.:

  \begin{itemize}
  \item $RS \simeq RS$ for all (nonnecessarily square) matrices $R,S$.
  \item $\begin{pmatrix}
    a\\
    A\\
  \end{pmatrix} \simeq
    \begin{pmatrix}
      0 & 1 \\
      A & 0  \\
      a & 0 \\
    \end{pmatrix}
    $ for all row vectors $a$.
  \end{itemize}

  In the special case where $\mathcal{R}=\mathbb{Z}_+$ we will say that $M$ and $N$ are \emph{flow-equivalent}
\end{definition}    
  This relation on matrices has only been formulated for the special case of $\mathcal{R} = \mathbb{Z}_+$ by Parry and Sullivan \cite{parry75}.

\begin{corollary}[formal statement of Theorem~\ref{thm:z}]
  Suppose that $P = \matrices{\mathcal{R}}$ is the prop of matrices  with coefficients in a semiring $\mathcal{R}$.

  Let $M,N$ be two matrices of size $n\times n$ and $m\times m$ respectively.
  Then $[M,n] = [N,m]$ in the traced completion of $P$ iff $M$ and $N$ are PS-equivalent
\end{corollary}
It suffices to remark that the equation of $\sim$ in Proposition~\ref{def:matricesintegers} coincides with the equations of $\simeq$  given above in the case where there are no inputs and outputs.

For the next theorem, we need :

\begin{definition}
    Let $\mathcal{R}$ be a semiring.
    We say that $M$ and $N$ are strong shift equivalent in $\mathcal{R}$ if $M \equiv N$ where $\equiv$ is the
    smallest equivalence relation on matrices with coefficients in $\mathcal{R}$ s.t. $RS \equiv SR$ for all (nonnecessarily square) matrices $R,S$ with coefficients in  $\mathcal{R}$ .
\end{definition}

\begin{theorem}
  Let $M, N$ be two matrices with nonnegative integer coefficients.
  
  Then $t M $ and $ t N$ are PS-equivalent in the semiring $\mathcal{R} = \mathbb{Z}_+[t]$ iff  $M$ and $N$ are strong shift equivalent in $\mathbb{Z}_+$.
\end{theorem}

\begin{proof}
First, it is clear that if $M,N$ are strong shift equivalent in $\mathbb{Z}_+$, then $tM$ and $tN$ are strong shift equivalent, and therefore PS-equivalent, in $\mathbb{Z}_+[t]$.

For two matrices $M,N \in \mathbb{Z}_+[t]$, say that $I-M$ and $I-N$ are positive equivalent if 
there is a finite sequence of matrices $M_1 \dots M_n \in \mathbb{Z}_+[t]$ s.t
\begin{itemize}
\item $M_1 = M$
\item $M_n = N$
\item for all $i \leq n$, either   $I-M_i$ is obtained from $I-M_{i-1}$ by a row/column operation\footnote{We include in row/column operations the ability to add one dimension to the matrix. The original article \cite{boylewagoner} does not use this property, as it deals with infinite matrices with finitely many nonzero coefficients. This is of course equivalent.} with coefficients in $\mathbb{Z}_+[t]$, or
  $I-M_{i-1}$ is obtained from $I-M_i$ by a row/column operation with coefficients in $\mathbb{Z}_+[t]$.  
\end{itemize}
(Notice that $I-M_i$ lives  in $\mathbb{Z}[t]$, not in $\mathbb{Z}_+[t]$).

By \cite{boylewagoner}, $M$ and $N$ are strong shift equivalent in $\mathbb{Z}_+$ iff  $I-tM$ and $I-tN$ are positive equivalent.
Therefore it suffices to prove that if $M$ and $N$ are PS-equivalent in $\mathbb{Z}_+[t]$, then $I-M$ and $I-N$ are positive equivalent.

Using the equations (4.5) to (4.8) of \cite{boylewagoner} for $R,S \in \mathbb{Z}_+[t]$ and the parameter $t$ (in their paper) equal to $1$ , we obtain easily  that $I-RS$ and $\begin{pmatrix} I & 0 \\ 0 & I-SR\end{pmatrix} $ are positive equivalent, from which it follows easily that $I-RS$ and $I-SR$ are positive equivalent.

  Second, let $M = \begin{pmatrix}
    a\\
    A\\
  \end{pmatrix} $.
  With a suitable choice of $R$ and $S$, one can see that $I-M$  is positive-equivalent to $I-M'$ with
$M' = \begin{pmatrix}
    a & 0 \\
    A & 0 \\
    a & 0 
  \end{pmatrix} $  and therefore positive-equivalent to $I-N$ where $N = \begin{pmatrix}
    0 & 1 \\
    A & 0 \\
    a & 0 \\
  \end{pmatrix} $, as $I-M'$  is obtained from $I-N$ by adding the last row to the first row.

Therefore PS-equivalence implies positive equivalence\footnote{There is a generalization of positive-equivalence to any semiring included in some ring, and one can see easily that PS-equivalence always imply positive equivalence. In fact, these notions coincide for rings, but in a general semiring,  positive equivalence do not imply PS-equivalence.}.
\end{proof}

\begin{corollary}[formal statement of Theorem~\ref{thm:zt}]
  Suppose that $P=\matrices{\mathbb{Z}_+[t]}$ is the prop of matrices with coefficients in a semiring $\mathbb{Z}_+[t]$.

  Let $M,N$ be two matrices with nonnegative integer coefficients of size $n\times n$ and $m\times m$ respectively.
  Then $[tM,n] = [tN,m]$ in the traced completion of $P$ iff $M$ and $N$ are strong shift equivalent.
\end{corollary}

\section{Consequences and applications}

Deciding strong shift equivalent is the main open problem of symbolic dynamics, and it is not clear that the new reformulation of the problem in terms of category theory would imply decidability or undecidability of the problem.
However it can be used to obtain new invariants, or to recover in a natural way existing invariants.
An \emph{invariant} of strong shift equivalence is a quantity (a real number, a group...) associated to a matrix $M$ s.t. two matrices that are strong shift equivalent have the same invariant. An invariant is never enough to prove that two matrices are strong shift equivalent, but a well chosen invariant might be used to prove that two matrices are not strong shift equivalent.

Here is our method to obtain new invariants.
Consider a category $\mathfrak{C}$ in the wild, which contains some mathematical items which satisfy the equations of Definition~\ref{def:matricesintegers} and Proposition~\ref{def:matricespolynomial}, and suppose the category is traced.
For the ease of the reader, we have recapitulated the most important equations in Fig~\ref{fig:equations}.

\begin{figure}
\[\begin{array}{lll}    
    {\begin{tikzpicture}[baseline=(current bounding box.center)]
\begin{pgfonlayer}{nodelayer}
\node [] (0) at (-0.000,-0.500) {};
\node [style=monoid] (1) at (-0.500,-0.500) {};
\node [] (2) at (-1.000,-0.250) {};
\node [] (3) at (-1.500,-0.250) {};
\node [style=monoid] (4) at (-1.000,-0.750) {};
\node [] (5) at (-1.500,-0.625) {};
\node [] (6) at (-1.500,-0.875) {};
\node [label={[label position=center]=}] (7) at (-1.750,-0.500) {};
\node [] (8) at (-2.000,-0.500) {};
\node [style=monoid] (9) at (-2.500,-0.500) {};
\node [style=monoid] (10) at (-3.000,-0.250) {};
\node [] (11) at (-3.500,-0.125) {};
\node [] (12) at (-3.500,-0.375) {};
\node [] (13) at (-3.000,-0.750) {};
\node [] (14) at (-3.500,-0.750) {};
\end{pgfonlayer}
\begin{pgfonlayer}{edgelayer}
\draw[] (0.center) to[out=180, in=  0] (1.center);
\draw[] (1.center) to[out=135, in=  0] (2.center);
\draw[] (2.center) to[out=180, in=  0] (3.center);
\draw[] (1.center) to[out=-135, in=  0] (4.center);
\draw[] (4.center) to[out=153, in=  0] (5.center);
\draw[] (4.center) to[out=-153, in=  0] (6.center);
\draw[] (8.center) to[out=180, in=  0] (9.center);
\draw[] (9.center) to[out=135, in=  0] (10.center);
\draw[] (10.center) to[out=153, in=  0] (11.center);
\draw[] (10.center) to[out=-153, in=  0] (12.center);
\draw[] (9.center) to[out=-135, in=  0] (13.center);
\draw[] (13.center) to[out=180, in=  0] (14.center);
\end{pgfonlayer}
\end{tikzpicture}
 } &     {\begin{tikzpicture}[baseline=(current bounding box.center)]
\begin{pgfonlayer}{nodelayer}
\node [] (0) at (-0.000,-0.500) {};
\node [style=monoid] (1) at (-0.500,-0.500) {};
\node [] (2) at (-1.000,-0.250) {};
\node [] (3) at (-1.000,-0.750) {};
\node [label={[label position=center]=}] (4) at (-1.250,-0.500) {};
\node [] (5) at (-1.500,-0.500) {};
\node [style=monoid] (6) at (-2.000,-0.500) {};
\node [] (7) at (-2.500,-0.500) {};
\node [] (8) at (-3.000,-0.250) {};
\node [] (9) at (-3.000,-0.750) {};
\end{pgfonlayer}
\begin{pgfonlayer}{edgelayer}
\draw[] (0.center) to[out=180, in=  0] (1.center);
\draw[] (1.center) to[out=135, in=  0] (2.center);
\draw[] (1.center) to[out=-135, in=  0] (3.center);
\draw[] (5.center) to[out=180, in=  0] (6.center);
\draw[] (6.center) to[out=-135, in=-44.9] (7.center) to[out=135, in=  0] (8.center);
\draw[] (6.center) to[out=135, in= 45] (7.center) to[out=-135, in=  0] (9.center);
\end{pgfonlayer}
\end{tikzpicture}
 }\\
    \\
      {\begin{tikzpicture}[baseline=(current bounding box.center)]
\begin{pgfonlayer}{nodelayer}
\node [] (0) at (-0.000,-0.250) {};
\node [] (1) at (-0.500,-0.250) {};
\node [] (2) at (-0.000,-0.625) {};
\node [] (3) at (-0.000,-0.875) {};
\node [style=bmonoid] (4) at (-0.500,-0.750) {};
\node [style=bmonoid] (5) at (-1.000,-0.500) {};
\node [] (6) at (-1.500,-0.500) {};
\node [label={[label position=center]=}] (7) at (-1.750,-0.500) {};
\node [] (8) at (-2.000,-0.125) {};
\node [] (9) at (-2.000,-0.375) {};
\node [style=bmonoid] (10) at (-2.500,-0.250) {};
\node [] (11) at (-2.000,-0.750) {};
\node [] (12) at (-2.500,-0.750) {};
\node [style=bmonoid] (13) at (-3.000,-0.500) {};
\node [] (14) at (-3.500,-0.500) {};
\end{pgfonlayer}
\begin{pgfonlayer}{edgelayer}
\draw[] (0.center) to[out=180, in=  0] (1.center);
\draw[] (2.center) to[out=180, in=26.6] (4.center);
\draw[] (3.center) to[out=180, in=-26.5] (4.center);
\draw[] (1.center) to[out=180, in= 45] (5.center);
\draw[] (4.center) to[out=180, in=-44.9] (5.center);
\draw[] (5.center) to[out=180, in=  0] (6.center);
\draw[] (8.center) to[out=180, in=26.6] (10.center);
\draw[] (9.center) to[out=180, in=-26.5] (10.center);
\draw[] (11.center) to[out=180, in=  0] (12.center);
\draw[] (10.center) to[out=180, in= 45] (13.center);
\draw[] (12.center) to[out=180, in=-44.9] (13.center);
\draw[] (13.center) to[out=180, in=  0] (14.center);
\end{pgfonlayer}
\end{tikzpicture}
 }&    {\begin{tikzpicture}[baseline=(current bounding box.center)]
\begin{pgfonlayer}{nodelayer}
\node [] (0) at (-0.000,-0.250) {};
\node [] (1) at (-0.000,-0.750) {};
\node [style=bmonoid] (2) at (-0.500,-0.500) {};
\node [] (3) at (-1.000,-0.500) {};
\node [label={[label position=center]=}] (4) at (-1.250,-0.500) {};
\node [] (5) at (-1.500,-0.250) {};
\node [] (6) at (-1.500,-0.750) {};
\node [] (7) at (-2.000,-0.500) {};
\node [style=bmonoid] (8) at (-2.500,-0.500) {};
\node [] (9) at (-3.000,-0.500) {};
\end{pgfonlayer}
\begin{pgfonlayer}{edgelayer}
\draw[] (0.center) to[out=180, in= 45] (2.center);
\draw[] (1.center) to[out=180, in=-44.9] (2.center);
\draw[] (2.center) to[out=180, in=  0] (3.center);
\draw[] (6.center) to[out=180, in=-44.9] (7.center) to[out=135, in= 45] (8.center);
\draw[] (5.center) to[out=180, in= 45] (7.center) to[out=-135, in=-44.9] (8.center);
\draw[] (8.center) to[out=180, in=  0] (9.center);
\end{pgfonlayer}
\end{tikzpicture}
 }\\
      \\
\multicolumn{2}{c}{\begin{tikzpicture}
\begin{pgfonlayer}{nodelayer}
\node [] (0) at (-0.000,-0.500) {};
\node [] (1) at (-0.000,-1.500) {};
\node [style=bmonoid] (2) at (-0.500,-1.000) {};
\node [style=monoid] (3) at (-1.000,-1.000) {};
\node [] (4) at (-1.500,-0.500) {};
\node [] (5) at (-1.500,-1.500) {};
\node [label={[label position=center]=}] (6) at (-1.750,-1.000) {};
\node [] (7) at (-2.000,-0.500) {};
\node [style=monoid] (8) at (-2.500,-0.500) {};
\node [] (9) at (-2.000,-1.500) {};
\node [style=monoid] (10) at (-2.500,-1.500) {};
\node [] (11) at (-3.000,-0.250) {};
\node [] (12) at (-3.000,-1.000) {};
\node [] (13) at (-3.000,-1.750) {};
\node [style=bmonoid] (14) at (-3.500,-0.500) {};
\node [] (15) at (-4.000,-0.500) {};
\node [style=bmonoid] (16) at (-3.500,-1.500) {};
\node [] (17) at (-4.000,-1.500) {};
\end{pgfonlayer}
\begin{pgfonlayer}{edgelayer}
\draw[] (0.center) to[out=180, in=63.4] (2.center);
\draw[] (1.center) to[out=180, in=-63.3] (2.center);
\draw[] (2.center) to[out=180, in=  0] (3.center);
\draw[] (3.center) to[out=117, in=  0] (4.center);
\draw[] (3.center) to[out=-116, in=  0] (5.center);
\draw[] (7.center) to[out=180, in=  0] (8.center);
\draw[] (9.center) to[out=180, in=  0] (10.center);
\draw[] (8.center) to[out=135, in=  0] (11.center);
\draw[] (10.center) to[out=-135, in=  0] (13.center);
\draw[] (11.center) to[out=180, in= 45] (14.center);
\draw[] (10.center) to[out=135, in=-44.9] (12.center) to[out=135, in=-44.9] (14.center);
\draw[] (14.center) to[out=180, in=  0] (15.center);
\draw[] (8.center) to[out=-135, in= 45] (12.center) to[out=-135, in= 45] (16.center);
\draw[] (13.center) to[out=180, in=-44.9] (16.center);
\draw[] (16.center) to[out=180, in=  0] (17.center);
\end{pgfonlayer}
\end{tikzpicture}
 } \\
            {\begin{tikzpicture}
\begin{pgfonlayer}{nodelayer}
\node [] (0) at (-0.000,-0.500) {};
\node [style=monoid] (1) at (-0.500,-0.500) {};
\node [shape=rectangle,draw,fill=white] (2) at (-1.000,-0.250) {};
\node [] (3) at (-1.500,-0.250) {};
\node [shape=rectangle,draw,fill=white] (4) at (-1.000,-0.750) {};
\node [] (5) at (-1.500,-0.750) {};
\node [label={[label position=center]=}] (6) at (-1.750,-0.500) {};
\node [] (7) at (-2.000,-0.500) {};
\node [shape=rectangle,draw,fill=white] (8) at (-2.500,-0.500) {};
\node [style=monoid] (9) at (-3.000,-0.500) {};
\node [] (10) at (-3.500,-0.250) {};
\node [] (11) at (-3.500,-0.750) {};
\end{pgfonlayer}
\begin{pgfonlayer}{edgelayer}
\draw[] (0.center) to[out=180, in=  0] (1.center);
\draw[] (1.center) to[out=135, in=  0] (2.center);
\draw[] (2.center) to[out=180, in=  0] (3.center);
\draw[] (1.center) to[out=-135, in=  0] (4.center);
\draw[] (4.center) to[out=180, in=  0] (5.center);
\draw[] (7.center) to[out=180, in=  0] (8.center);
\draw[] (8.center) to[out=180, in=  0] (9.center);
\draw[] (9.center) to[out=135, in=  0] (10.center);
\draw[] (9.center) to[out=-135, in=  0] (11.center);
\end{pgfonlayer}
\end{tikzpicture}
 } &    {\begin{tikzpicture}
\begin{pgfonlayer}{nodelayer}
\node [] (0) at (-0.000,-0.250) {};
\node [] (1) at (-0.000,-0.750) {};
\node [style=bmonoid] (2) at (-0.500,-0.500) {};
\node [shape=rectangle,draw,fill=white] (3) at (-1.000,-0.500) {};
\node [] (4) at (-1.500,-0.500) {};
\node [label={[label position=center]=}] (5) at (-1.750,-0.500) {};
\node [] (6) at (-2.000,-0.250) {};
\node [shape=rectangle,draw,fill=white] (7) at (-2.500,-0.250) {};
\node [] (8) at (-2.000,-0.750) {};
\node [shape=rectangle,draw,fill=white] (9) at (-2.500,-0.750) {};
\node [style=bmonoid] (10) at (-3.000,-0.500) {};
\node [] (11) at (-3.500,-0.500) {};
\end{pgfonlayer}
\begin{pgfonlayer}{edgelayer}
\draw[] (0.center) to[out=180, in= 45] (2.center);
\draw[] (1.center) to[out=180, in=-44.9] (2.center);
\draw[] (2.center) to[out=180, in=  0] (3.center);
\draw[] (3.center) to[out=180, in=  0] (4.center);
\draw[] (6.center) to[out=180, in=  0] (7.center);
\draw[] (8.center) to[out=180, in=  0] (9.center);
\draw[] (7.center) to[out=180, in= 45] (10.center);
\draw[] (9.center) to[out=180, in=-44.9] (10.center);
\draw[] (10.center) to[out=180, in=  0] (11.center);
\end{pgfonlayer}
\end{tikzpicture}
 }\\
\end{array}
\]

\caption{Some of the equations that should satisfy the black, white diamonds and the rectangle square in the category so that we can apply the whole procedure. Equations involving the morphisms $0\to 1$ and $1 \to 0$ are not shown for simplicity. The full list is presented earlier in Definition~\ref{def:matricesintegers} and Proposition~\ref{def:matricespolynomial}}.
\label{fig:equations}
\end{figure}
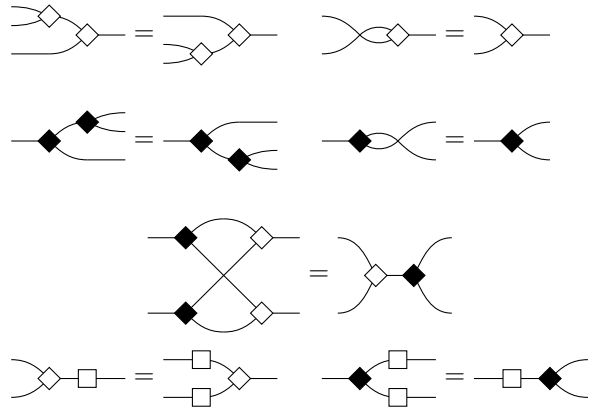

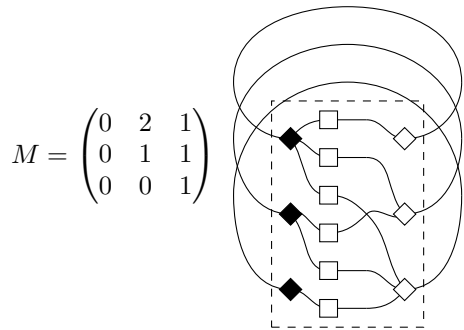
\begin{figure}
  $M = \begin{pmatrix}
  0 & 2 & 1 \\
  0 & 1 & 1 \\
  0 & 0 & 1 \\
\end{pmatrix}$ 
\begin{tikzpicture}[baseline=(current bounding box.center)]
\begin{pgfonlayer}{nodelayer}
\node [] (0) at (-0.500,-1.250) {};
\node [] (1) at (-0.500,-2.000) {};
\node [] (2) at (-0.500,-2.750) {};
\node [] (3) at (-3.500,-1.250) {};
\node [] (4) at (-3.500,-2.000) {};
\node [] (5) at (-3.500,-2.750) {};
\path[dashed=None,draw=,color=black] (-1.000,-1.500) rectangle (-3.000,-4.500);
\node at (-2.000,-3.000) {};
\node [] (6) at (-2.000,-0.250) {};
\node [] (7) at (-2.000,-0.750) {};
\node [] (8) at (-2.000,-1.250) {};
\node [style=monoid] (9) at (-1.250,-2.000) {};
\node [style=monoid] (10) at (-1.250,-3.000) {};
\node [style=monoid] (11) at (-1.250,-4.000) {};
\node [] (12) at (-1.750,-1.750) {};
\node [] (13) at (-1.750,-2.250) {};
\node [] (14) at (-1.750,-3.000) {};
\node [] (15) at (-1.750,-3.750) {};
\node [] (16) at (-1.750,-4.250) {};
\node [shape=rectangle,draw,fill=white] (17) at (-2.250,-1.750) {};
\node [shape=rectangle,draw,fill=white] (18) at (-2.250,-2.250) {};
\node [shape=rectangle,draw,fill=white] (19) at (-2.250,-2.750) {};
\node [shape=rectangle,draw,fill=white] (20) at (-2.250,-3.250) {};
\node [shape=rectangle,draw,fill=white] (21) at (-2.250,-3.750) {};
\node [shape=rectangle,draw,fill=white] (22) at (-2.250,-4.250) {};
\node [style=bmonoid] (23) at (-2.750,-2.000) {};
\node [style=bmonoid] (24) at (-2.750,-3.000) {};
\node [style=bmonoid] (25) at (-2.750,-4.000) {};
\end{pgfonlayer}
\begin{pgfonlayer}{edgelayer}
\draw[] (0.center) to[out=-89.9, in=  0] (9.center);
\draw[] (1.center) to[out=-89.9, in=  0] (10.center);
\draw[] (2.center) to[out=-89.9, in=  0] (11.center);
\draw[] (9.center) to[out=180, in=  0] (12.center);
\draw[] (10.center) to[out=135, in=  0] (13.center);
\draw[] (11.center) to[out=180, in=  0] (15.center);
\draw[] (11.center) to[out=-127, in=  0] (16.center);
\draw[] (12.center) to[out=180, in=  0] (17.center);
\draw[] (13.center) to[out=180, in=  0] (18.center);
\draw[] (11.center) to[out=127, in=-44.9] (14.center) to[out=135, in=  0] (19.center);
\draw[] (10.center) to[out=-135, in= 45] (14.center) to[out=-135, in=  0] (20.center);
\draw[] (15.center) to[out=180, in=  0] (21.center);
\draw[] (16.center) to[out=180, in=  0] (22.center);
\draw[] (17.center) to[out=180, in=53.1] (23.center);
\draw[] (18.center) to[out=180, in=  0] (23.center);
\draw[] (19.center) to[out=180, in=-53] (23.center);
\draw[] (20.center) to[out=180, in= 45] (24.center);
\draw[] (21.center) to[out=180, in=-44.9] (24.center);
\draw[] (22.center) to[out=180, in=  0] (25.center);
\draw[] (23.center) to[out=180, in=-89.9] (3.center) to[out= 90, in=180] (6.center) to[out=  0, in= 90] (0.center);
\draw[] (24.center) to[out=180, in=-89.9] (4.center) to[out= 90, in=180] (7.center) to[out=  0, in= 90] (1.center);
\draw[] (25.center) to[out=180, in=-89.9] (5.center) to[out= 90, in=180] (8.center) to[out=  0, in= 90] (2.center);
\end{pgfonlayer}
\end{tikzpicture}

\caption{How a matrix is transformed into a diagram}
\label{fig:matrixdiag}
\end{figure}

Then, in this case, given a square matrix $M$ with nonnegative integer coefficients, one can associate to the matrix $M$ a morphism $\psi(M): 0 \to 0$ in this category $\mathfrak{C}$. The correspondance is given in Figure~\ref{fig:matrixdiag}, where black diamonds with $n$ outputs corresponds to a cascade of $n-1$ black diamonds with $2$ outputs (same for white diamonds):
\begin{itemize}  
\item Let $M=(a_{ij})_{1 \leq i,j  \leq n}$
\item The diagram has exactly $n$ black diamonds and $n$ white diamonds, numbered from $1$ to $n$
\item There are $a_{ij}$ wires from the $i$-th black diamond to the $j$-th black diamond, each of them support a white rectangle.
\item The output of the $i$-th white diamond is linked back to the input of the $i$-th black diamond.
\item The new diagram is a diagram with no inputs and outputs that we call $\psi(M)$. 
\end{itemize}  

This can be done purely categorically without resorting to diagrams:
As $\mathfrak{C}$ is traced, and contain generators that satisfy all the equations of $\matrices{\mathbb{Z}_+[t]}$, by the universal property of $\widehat{\matrices{\mathbb{Z}_+[t]}}$ there is a functor $F$ from $\widehat{\matrices{\mathbb{Z}_+[t]}}$ to $\mathfrak{C}$.
Then the $0 \to 0$ morphism we are looking at is just $\psi(M) = F(tr_n(\iota(tM)))$.

From Theorem~\ref{thm:zt} we obtain that if $M$ and $N$ are strong shift equivalent, then $\psi(M) = \psi(N)$.
In other words, we have developed a systematic way for obtaining invariants of strong shift equivalence.
As Theorem~\ref{thm:zt} is an equivalence, we also have the guarantee that our approach is complete in the sense that in one specific category (namely $\widehat{\matrices{\mathbb{Z}_+[t]}}$ itself), $M$ and $N$ are strong shift equivalent iff $\psi(M) = \psi(N)$.

Now, categories which contains (bicommutative) bialgebras do exist in the wild, and they are actually very well studied.
The purpose of this section is now clear: we will look at well known examples of bialgebras, and see to which invariant $\psi$ they correspond.

There are however two caveats: most of the known bialgebras are actually Hopf algebras. The additional structure of Hopf algebras intuitively gives the existence of a ``negative wire'' (the antipode) that can cancel a wire. This is bad in our case, as, when we interpret matrices in these categories, one would obtain invariants for strong shift equivalence in $\mathbb{Z}$, i.e. a weaker and more well known notion of equivalence. The second caveat is that most examples of bialgebras are actually not in traced categories.

It turns out that in many examples, one can transform the bialgebra to solve the second problem, with the added benefit that the bialgebra will not be a Hopf algebra.

We will give three examples. The first two examples are possibly the most well known examples of bialgebras. We will explain how these bialgebras give rise to some well known invariants of symbolic dynamics. The third one is not that well known, and was actually conceived with foresight.

\subsection{Monoids}
\label{sec:monoids}
The first example starts from one of the easiest possible bialgebra.

Let $(\mathcal{M}, \times)$ be a commutative monoid, with unit element $1$. Consider the prop where arrows $n \to m$ corresponds to arbitrary functions from $\mathcal{M}^n \to \mathcal{M}^m$, and consider inside this category the following 
functions:

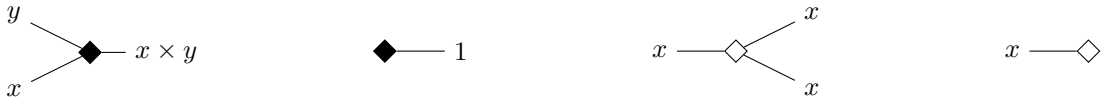
\begin{tikzpicture}[baseline=(current bounding box.center)]
  \node (A) at (0,0) {$x$};
  \node (B) at (0,1) {$y$};
  \node (C) [style=bmonoid] at (1,0.5) {};
  \node (D)  at (2,0.5) {$x\times y$};
  \draw (A) -- (C);
  \draw (B) -- (C);
  \draw (C) -- (D);
\end{tikzpicture}
\hfill
\begin{tikzpicture}[baseline=(current bounding box.center)]
  \node (A) [style=bmonoid] at (0,0.5) {};
  \node (B)  at (1,0.5) {$1$};
  \draw (A) -- (B);
\end{tikzpicture}  
\hfill
\begin{tikzpicture}[baseline=(current bounding box.center)]
  \node (A) at (0,0) {$x$};
  \node (B) at (0,1) {$x$};
  \node (C) [style=monoid] at (-1,0.5) {};
  \node (D)  at (-2,0.5) {$x$};
  \draw (A) -- (C);
  \draw (B) -- (C);
  \draw (C) -- (D);
\end{tikzpicture}
\hfill
\begin{tikzpicture}[baseline=(current bounding box.center)]
  \node (A) [style=monoid] at (0,0.5) {};
  \node (B)  at (-1,0.5) {$x$};
  \draw (A) -- (B);
\end{tikzpicture}

It is easy to see that these functions satisfy the defining equations of \ref{def:matricesintegers}, and in particular the fundamental bigebra equation:

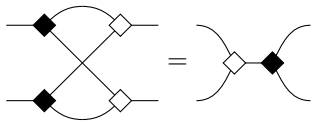
{\begin{tikzpicture}
\begin{pgfonlayer}{nodelayer}
\node [] (0) at (-0.000,-0.500) {};
\node [] (1) at (-0.000,-1.500) {};
\node [style=bmonoid] (2) at (-0.500,-1.000) {};
\node [style=monoid] (3) at (-1.000,-1.000) {};
\node [] (4) at (-1.500,-0.500) {};
\node [] (5) at (-1.500,-1.500) {};
\node [label={[label position=center]=}] (6) at (-1.750,-1.000) {};
\node [] (7) at (-2.000,-0.500) {};
\node [style=monoid] (8) at (-2.500,-0.500) {};
\node [] (9) at (-2.000,-1.500) {};
\node [style=monoid] (10) at (-2.500,-1.500) {};
\node [] (11) at (-3.000,-0.250) {};
\node [] (12) at (-3.000,-1.000) {};
\node [] (13) at (-3.000,-1.750) {};
\node [style=bmonoid] (14) at (-3.500,-0.500) {};
\node [] (15) at (-4.000,-0.500) {};
\node [style=bmonoid] (16) at (-3.500,-1.500) {};
\node [] (17) at (-4.000,-1.500) {};
\end{pgfonlayer}
\begin{pgfonlayer}{edgelayer}
\draw[] (0.center) to[out=180, in=63.4] (2.center);
\draw[] (1.center) to[out=180, in=-63.3] (2.center);
\draw[] (2.center) to[out=180, in=  0] (3.center);
\draw[] (3.center) to[out=117, in=  0] (4.center);
\draw[] (3.center) to[out=-116, in=  0] (5.center);
\draw[] (7.center) to[out=180, in=  0] (8.center);
\draw[] (9.center) to[out=180, in=  0] (10.center);
\draw[] (8.center) to[out=135, in=  0] (11.center);
\draw[] (10.center) to[out=-135, in=  0] (13.center);
\draw[] (11.center) to[out=180, in= 45] (14.center);
\draw[] (10.center) to[out=135, in=-44.9] (12.center) to[out=135, in=-44.9] (14.center);
\draw[] (14.center) to[out=180, in=  0] (15.center);
\draw[] (8.center) to[out=-135, in= 45] (12.center) to[out=-135, in= 45] (16.center);
\draw[] (13.center) to[out=180, in=-44.9] (16.center);
\draw[] (16.center) to[out=180, in=  0] (17.center);
\end{pgfonlayer}
\end{tikzpicture}
 }

However, there is no easy notion of a trace (of a feedback loop, of fixed points) in this category. One can solve the problem in two different ways. The first way is to consider the prop of multivalued partial functions (i.e. relations) instead of functions. This would work in theory but would not be very interesting: Remember that we will trace all inputs/ouputs of the diagrams, so we are only interested in morphisms $0 \to 0$ of your category. This particular category has only 2 such morphisms, so this would partition the whole set of matrices into only two sets, which does not give a good invariant.

The ``good'' way is to go to a weighted setting.
In this setting, one has to consider inputs of size $n$ to be elements of  $\mathcal{M}^n$ with multiplicities \cite{Eilenberg}.
In other words, let $S_n$ be the set of functions from $\mathcal{M}^n$ to $\mathbb{N}_\infty$, the set of nonnegative (possibly infinite) integers.
For a function $f$ in $S_k$ and $x \in \mathcal{M}^k$, $f(x)$  has to be interpreted as the multiplicity of $x$.
By identifying $x \in \mathcal{M}^n$ with the function that is zero everywhere except on $x$, we get an embedding of $\mathcal{M}^n$ into $S_n$.

Then the arrows $n \to m$ of our category correspond to linear (in $\mathbb{N}_\infty$) functions from $S_n$ to $S_m$.

Now this category is traced: Let $f$ be a function from $S_n$ to $S_m$. $Tr(f)$  is obtained in the following way.
Let  $y$ be an input of $Tr(f)$, and $z \in \mathcal{M}^m$. The coefficient of $z$ in the output of $Tr(f)$ is obtained by taking all possibles values of $x \in \mathcal{M}$, computing $f(x,y)$, looking at the coefficient of $(x,z)$ in $f(x,y)$, and adding up all these coefficients, possibly obtaining  $\infty$ as a result.

\begin{example}
Suppose that our monoid $\mathcal{M}$ is the monoid  $\mathbb{R}$ with multiplication and consider the following diagram:
  
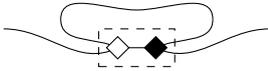
{\begin{tikzpicture}[baseline=(current bounding box.center)]
\begin{pgfonlayer}{nodelayer}
\node [] (0) at (-0.000,-0.500) {};
\node [] (1) at (-0.750,-0.375) {};
\node [] (2) at (-2.750,-0.375) {};
\path[dashed=None,draw=,color=black] (-1.250,-0.500) rectangle (-2.250,-1.000);
\node at (-1.750,-0.750) {};
\node [] (3) at (-1.750,-0.250) {};
\node [style=bmonoid] (4) at (-1.500,-0.750) {};
\node [style=monoid] (5) at (-2.000,-0.750) {};
\node [] (6) at (-3.500,-0.500) {};
\end{pgfonlayer}
\begin{pgfonlayer}{edgelayer}
\draw[] (1.center) to[out=-89.9, in=26.6] (4.center);
\draw[] (0.center) to[out=180, in=-26.5] (4.center);
\draw[] (4.center) to[out=180, in=  0] (5.center);
\draw[] (5.center) to[out=153, in=-89.9] (2.center) to[out= 90, in=180] (3.center) to[out=  0, in= 90] (1.center);
\draw[] (5.center) to[out=-153, in=  0] (6.center);
\end{pgfonlayer}
\end{tikzpicture}
 }

To better understand the diagram, we will look at inputs in $\mathbb{R}^2$, but inputs are really in $\mathbb{R}^2 \to {\mathbb{N}_\infty}$ but they can be extended by linearity.

Without the trace part, we are looking at the function $(x,y) \mapsto (x \times y, x\times y)$
Using the trace, we force the first output ($x \times y$) to  be actually equal to $x$.
Now, for $y \in \mathcal{M}$ there are two cases:

\begin{itemize}
\item If $y = 1$, then $x \times y$ is actually equal to $x$, and we actually obtain an output on the other wire of value $x$, and this for all $x$.
As a consequence, the traced function, on input $1$, gives as input all possible values in $\mathcal{M}$, each with coefficient $1$. 
\item If $y \not= 1$, then $x \times y = x$ only if $x = 0$. As a consequence, the traced function, on input $y \not= 1$, gives as output the value $0$, with coefficient $1$  
\end{itemize}  
\end{example}

\begin{example}
Suppose now that our $\mathcal{M}$ is the monoid  $\mathbb{R}^{> 0}$ with multiplication (or equivalently $\mathbb{R}$ with addition)

The only difference in this case is that on inputs $y \not =1$ the function does not give an output (technically the output is the function identically equal to $0$)
\end{example}

\begin{example}
  Suppose now that our $\mathcal{M}$ is the monoid  of subsets of $\{0,1\}$ with union.
  \begin{itemize}
  \item On input $y = \emptyset$, the result is all 4 values $\emptyset, \{0\}, \{1\}, \{0,1\}$, each with coefficient $1$
  \item On input $y = \{1\}$, the result is $\{1\}$ with coefficient $1$ and $\{0,1\}$ with coefficient $1$
  \item On input $y = \{0,1\}$, the result is $\{0,1\}$ with coefficient $1$    
\end{itemize}    
\end{example}

Now, we use our protocol, and start from a matrix $M$, see it as a diagram in this category using Figure\ref{fig:matrixdiag}. What is the result we obtain ?
It is easy to see that we obtain exactly the number of solutions of the equation $Mx = x$.

\begin{proposition}
  Let $\mathcal{M}$ be a  commutative  monoid written additively.
  The number of solutions to the equation  $Mx = x$ is an invariant of flow equivalence. More precisely, if $M$ and $N$ are two matrices that are flow equivalent, then the equation $Mx = x$ has the same number of solutions as the equation $Nx = x$.    
\end{proposition}

For strong shift equivalence, we also have to explain what the square will stand for. It is easy to see that we need it to be an homomorphism of $\mathcal{M}$, and we get immediately the following:
\begin{proposition}
  Let $\mathcal{M}$ be a commutative  monoid written additively and $h$ be an homomorphism
  The number of solutions to the equation  $h(Mx) = x$ is an invariant of strong shift equivalence.
  More precisely, if $M$ and $N$ are two matrices that are strong shift equivalent, then the equation $h(Mx) = x$ has the same number of solutions as the equation $h(Nx) = x$.    
\end{proposition}
Of course the fact that this is an invariant is quite immediate: the important thing is that we obtained it entirely algebraically.

By taking $\mathcal{M} = \mathbb{C}$ and $h$ the multiplication by
$\lambda$, we get that if $M$ and $N$ are strong shift equivalent,
then $1/\lambda$ is an eigenvalue of $M$ iff $1/\lambda$ is an
eigenvalue of $N$, i.e. $M$ and $N$ have the same eigenvalues, except
possibly the value $0$. This is a well known invariant, called the
``spectrum away from zero''.

By working in finite fields (so that we can count something: with $\mathcal{M} = \mathbb{C}$ the number of solutions is either 0 or $\infty$), we can prove easily that the eigenvalues should also have the same multiplicities.

\subsection{Algebras and tensors}

In general algebra \cite{Underwood,Sweedler}, a bialgebra in some vector space $V$ over some field $\mathbb{K}$ is something that satifies all axioms of Definition~\ref{def:matricesintegers} when arrows from $n \to m$ are linear maps from $V^{\otimes n}$ to $V^{\otimes m}$.

The most well known bialgebras are the monoid algebra, and the binomial algebra, as any good book about bialgebras and Hopf algebras will prove\footnote{These are the first examples in \cite{Underwood} and the first commutative examples in \cite{Sweedler}. In \cite{Sweedler} the binomial algebra is seen as a bialgebra over the symmetric algebra.}.

In the case of the   monoid algebra  $\mathbb{K}[M]$ of a monoid $M$,  $V$ is the vector space of basis $M$, the multiplication is the multiplication on $M$ extended linearly, and the comultiplication map $\Delta: V \mapsto V^{\otimes 2}$ is the copy map $\Delta(x) = x \otimes x$.

In the case of the binomial algebra, we  take  $V = \mathbb{R}[X]$, the polynomials with real coefficients in one variable.
There is a obvious algebra structure on $V$, the multiplication:
\begin{itemize}
\item $\mu( P \otimes Q) = PQ $
\item $\eta = 1$
\end{itemize}  
And there is a compatible comultiplication, which gives a bialgebra
\begin{itemize}
\item $\Delta(X) = 1 \otimes X + X\otimes 1$. More generally $\Delta(X^n) = \sum_k {n \choose k} X^k \otimes X^{n-k} $
\item $\epsilon(X^n) = \delta_{0,n}$
\end{itemize}  

We can try to  exploit these well known (bicommutative) bialgebras in our case.
There is however a fundamental problem: When $V$ is infinite, there is in general no notion of a trace.

The first example for a  finite, commutative $M$, essentially amounts to the work that has been done in the previous section.

To fit the second example in our framework, we need to put it somehow in a category where a trace is defined.
If we think of these examples as matrices with coefficients in $\mathbb{K}$, then the problem of the trace is that we need to do an infinite sum, which is usually impossible.
However, if we think of matrices with coefficients in a \emph{complete } semiring $\mathcal{R}$ instead of a field $\mathbb{K}$, then everything works.

Let $\mathcal{R}$ be a commutative complete semiring. This is a semiring where all sums, even possibly infinite, are defined\footnote{We will actually need only countable sums for the applications, so $\omega$-complete semirings are sufficient}. Let $I$ be a set.

Then we can look at the prop where the morphisms $n \times m$ are matrices with coefficients in $\mathcal{R}$ of size $I^n$ by $I^m$ with the usual product and tensor product. More exactly:
\begin{itemize}
\item Arrows $n \times m$  are functions from $I^n \times I^m$ to $\mathcal{R}$
\item The composition $fg$ of $f: n \times m$ and $g: m \times p$ is defined by $(fg)(x,y) = \sum_{z \in I^m} f(x,z) g(z,y)$. This is well defined as $\mathcal{R}$ is complete.
\item The tensor product $f\otimes g$ of $f$ and $g$ is defined by $(f \otimes g)((a,b), (c,d)) = f(a,c) g(b,d)$
\item The trace of $f$ is $(tr f)(y,z) = \sum_x f((x,y),(x,z))$  
\end{itemize}
If we think in terms of matrices, the full trace (i.e. where all outputs are branched into inputs) of a matrix is exactly what we usually call its trace, i.e. the sum of all diagonal elements (which is always well defined).

The reader should realize that the previous example of section~\ref{sec:monoids} corresponds to matrices with coefficients in the  complete semiring $\mathbb{N}_{\infty}$.

Now we can try to find the analogue of the binomial bialgebra. In this bialgebra, we were interested in $V = \mathbb{R}[X]$. We will now use the complete semigroup  $\mathbb{R}_{\infty}$, the semigroup which contains \emph{nonnegative} numbers and $\infty$, instead of $\mathbb{R}$. The fact we restrict ourself to nonnegative numbers should not be a problem as the equations that define the multiplication and the comultiplication do not involve any negative number\footnote{However, the binomial bialgebra has the additional property of a Hopf algebra that we will \emph{lose}, as this property involves negative coefficients. This is actually a good thing as we do not want an Hopf algebra.}. As general sums are now authorized, we will not obtain the polynomials in  $\mathbb{R}_\infty$, but \emph{formal series} in $\mathbb{R}_\infty$, that is $\mathbb{R}_\infty [[X]]$.

Let's take time to define things correctly :

\begin{itemize}
\item We define $V^{\otimes n} = \mathbb{R}_\infty [[ X_1 \dots X_n ]]$, formal power series in $n$ variables. Elements of $V^{\otimes n}$ will be written $P(X_1 \dots X_n)$
\item Arrows from $n$ to $m$ are linear maps from $V^{\otimes n}$ to $V^{\otimes m}$
\item For two maps $f,g$  the composition $f\circ g$ is the usual composition of power series.
\item If $P(X_1\dots X_n) \in V^{\otimes n}$ and $Q(X_1\dots X_m) \in V^{\otimes m}$, then the tensor product of $P$ and $Q$ is $(P \otimes Q)(X_1 \dots X_{n+m}) = P(X_1 \dots X_n) Q(X_{n+1} \dots X_{n+m})$. We will identify $X$ with $X_1$, so $X^n \otimes X^p$ is a notation for $X_1^n X_2^p$. 
\end{itemize}

We can now look at the previous example in this setting. First, one can see easily that it is still a bialgebra in this setting.
We will look at invariants for strong shift equivalence, so we need to interpret the square symbols in the diagrams. We will take the function $X^n \mapsto (\lambda X)^n$ for some $\lambda \in \mathbb{R}_\infty$ which is easily seen to be indeed a homomorphism for both the multiplication and comultiplication.

Before explaining which invariant we obtain, we will do a computation in one example.

Let's start with the matrix $\begin{pmatrix} 1 & 1\\1 & 0
\end{pmatrix}$
We therefore want to know to which real number  corresponds the diagram:

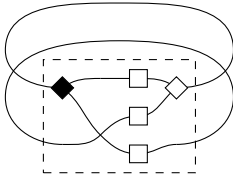
{\begin{tikzpicture}[baseline=(current bounding box.center)]
\begin{pgfonlayer}{nodelayer}
\node [] (0) at (-0.500,-0.875) {};
\node [] (1) at (-0.500,-1.500) {};
\node [] (2) at (-3.500,-0.875) {};
\node [] (3) at (-3.500,-1.500) {};
\path[dashed=None,draw=,color=black] (-1.000,-1.000) rectangle (-3.000,-2.500);
\node at (-2.000,-1.750) {};
\node [] (4) at (-2.000,-0.250) {};
\node [] (5) at (-2.000,-0.750) {};
\node [style=monoid] (6) at (-1.250,-1.375) {};
\node [] (7) at (-1.250,-2.125) {};
\node [shape=rectangle,draw,fill=white] (8) at (-1.750,-1.250) {};
\node [shape=rectangle,draw,fill=white] (9) at (-1.750,-1.750) {};
\node [shape=rectangle,draw,fill=white] (10) at (-1.750,-2.250) {};
\node [] (11) at (-2.250,-1.250) {};
\node [] (12) at (-2.250,-2.000) {};
\node [style=bmonoid] (13) at (-2.750,-1.375) {};
\node [] (14) at (-2.750,-2.125) {};
\end{pgfonlayer}
\begin{pgfonlayer}{edgelayer}
\draw[] (0.center) to[out=-89.9, in=  0] (6.center);
\draw[] (1.center) to[out=-89.9, in=  0] (7.center);
\draw[] (6.center) to[out=143, in=  0] (8.center);
\draw[] (6.center) to[out=-143, in=  0] (9.center);
\draw[] (7.center) to[out=180, in=  0] (10.center);
\draw[] (8.center) to[out=180, in=  0] (11.center);
\draw[] (11.center) to[out=180, in=36.9] (13.center);
\draw[] (10.center) to[out=180, in=-44.9] (12.center) to[out=135, in=-36.8] (13.center);
\draw[] (9.center) to[out=180, in= 45] (12.center) to[out=-135, in=  0] (14.center);
\draw[] (13.center) to[out=180, in=-89.9] (2.center) to[out= 90, in=180] (4.center) to[out=  0, in= 90] (0.center);
\draw[] (14.center) to[out=180, in=-89.9] (3.center) to[out= 90, in=180] (5.center) to[out=  0, in= 90] (1.center);
\end{pgfonlayer}
\end{tikzpicture}
 }

Let's start with the nontraced version\footnote{In this particular example, one can use the equations of the traced prop to remove one trace, obtaining a simpler diagram. This is left as an exercise.}, and an input of the form $X^n \otimes X^m$.

After applying the black diamond, we obtain $\sum_k {n\choose k} X^k \otimes X^{n-k} \otimes  X^m$.
After applying the symmetry, we obtain $\sum_k {n\choose k} X^k \otimes X^m \otimes X^{n-k}$.
After applying the squares, we obtain  $\sum_k {n\choose k} (\lambda X)^k \otimes (\lambda X)^m \otimes (\lambda X)^{n-k}$.
Finally, after applying the white diamond, we obtain
$\sum_k {n\choose k} (\lambda X)^{m+k} \otimes (\lambda X)^{n-k}  = \lambda^{n+m} \sum_k {n\choose k} X^{m+k} \otimes X^{n-k} $.

Now, we are interested in the traced version of this diagram, and therefore, in the diagonal of the map.
Now the coefficient  of $X^n \otimes X^m$ is precisely obtained when $m+k = n$, and is therefore equal to $ {n \choose n-m} \lambda^{n+m}$

The trace of the diagram is therefore, for small values of $\lambda$:

\[ \sum_{n,m} {n \choose n-m} \lambda^{n+m} = \sum_{n,k} {n \choose k} \lambda^{2n-k}  = \frac{1}{1 - \lambda^2 - \lambda} \]

The reader familiar with invariants with symbolic dynamics has recognized this expression and know what our next theorem will be:

\begin{theorem}
  Let $M$ be a nonnegative matrix. If we interpret $M$ in the previous bialgebra, with $X^n \mapsto (\lambda X)^n$ as the homomorphism, the result we obtain is exactly $\zeta_M(\lambda)$ the value of the Zeta function\cite{BowenLanford} of the shift computed at $t = \lambda$, with $\zeta_M(t) = \frac{1}{\det(I-tM)}$.
  
  In particular the Zeta function is an invariant of strong shift equivalence.
\end{theorem}
This theorem is a direct consequence of MacMahon's Master Theorem, see \cite{MacMahon}.

\subsection{Groups and Cospans}

In our first example, we proved that, for all monoid $\mathcal{M}$, the number of solutions in $\mathcal{M}$ of the equation  $Mx = x$ is an invariant of flow equivalence.

One can do a bit better, and actually associate a commutative monoid itself to a matrix $M$. We will explain how this can be explained categorically.
Technically, what we will be looking at is a subcategory of $CoSpan(\mathbf{AbGrp})$ where $\mathbf{AbGrp}$ is the category of abelian groups. However we will explain which prop we will look at in layman terms.

We assume familiarity with classical notations for finitely presented groups. By $\left< X \middle| R \right>_{ab}$, we denote the \emph{commutative} group with generators $X$ and presentation $R$. It is important to note that all our groups are commutative. One can relax this condition at some point (even to obtain a bicommutative bialgebra at the end), but this will be easier to 
understand.

Morphisms $n \to m$ in our categories are triples $(G,n,m)$ where $G$ is a group with at least $n+m$ generators, $G = \left< x_1 \dots x_n, y_1 \dots y_m, z_1 \dots z_p \middle| R \right>_{ab}$. $x_1 \dots x_n$ and $y_1 \dots y_m$ are to be understood as ``input'' and ``output'' generators respectively, and $z_1 \dots z_p$ as ``inside'' generators of the group.
We will usually write the morphism $(G,n,m) = \left< x_1 \dots x_n \middle| y_1 \dots y_m\middle| z_1 \dots z_p \middle| R \right>_{ab}$  to further separate the three kind of generators.

We will consider two such groups to be equal if there is an isomorphism that can only change the ``inner'' part of the group, i.e. 
$(G,n,m) = \left< x_1 \dots x_n\middle| y_1 \dots y_m\middle| z_1 \dots z_p \middle| R \right>_{ab}$ and $(G',n,m) = \left< x_1 \dots x_n\middle| y_1 \dots y_m\middle| z_1 \dots z_q \middle| R' \right>_{ab}$ are equal if there is an isomorphism $\phi$ from $G$ to $G'$  s.t.  $\phi(x_i) = x_i, \phi(y_j) = y_j$.

There are two way to compose these morphisms:

\begin{itemize}
\item The composition of the morphism $(G,n,m) = \left< x_1 \dots x_n\middle| y_1 \dots y_m\middle| z_1 \dots z_p \middle| R \right>_{ab}$ and the morphism  $(G',m,k) = \left< r_1 \dots r_m\middle| s_1 \dots s_k\middle| t_1 \dots t_q \middle| R' \right>_{ab}$ consists in putting the two groups together and identifying the outputs of the first one with the inputs of the second one. It is therefore the group:
  $$(H,n,k) = \left< x_1 \dots x_n\middle| s_1 \dots s_k\middle| r_1 \dots r_m, y_1 \dots y_m, z_1 \dots z_p, t_1 \dots t_q \middle| R , R', r_i = y_i \right>_{ab}$$
\item The tensor product of the morphism $(G^1,n^1,m^1) = \left< x^1_1 \dots x^1_{n^1}\middle| y^1_1 \dots y^1_{m^1}\middle| z^1_1 \dots z^1_{p^1} \middle| R^1 \right>_{ab}$ and 
the morphism  $(G^2,n^2,m^2) = \left< x^2_1 \dots x^2_{n^2}\middle| y^2_1 \dots y^2_{m^2}\middle| z^2_1 \dots z^2_{p^2} \middle| R^2 \right>_{ab}$ consists in putting them side by side, i.e. it is the group:
  $$(H, n^1 + n^2, m^1 + m ^2) = \left< x^1_1 \dots x^1_{n^1}, x^2_1 \dots x^2_{n^2}\middle| y^1_1 \dots y^1_{m^1},y^2_1 \dots y^2_{m^2}\middle| z^1_1 \dots z^1_{p^1},z^2_1 \dots z^2_{p^2} \middle| R^1, R^2 \right>_{ab}$$

  As our groups are abelian, $H$ is just the direct sum of $G^1$ and $G^2$ (or cartesian product if the reader prefer, but categorically it is clearly more like a direct sum).
\end{itemize}

This category is traced (it is even compact closed): the trace of the morphism $(G,n,m) = \left< x_1 \dots x_n\middle| y_1 \dots y_m\middle| z_1 \dots z_p \middle| R \right>_{ab}$ consists in equating the first input and the first output and internalizing them, to obtain $(H, n-1, m-1) = \left< x_2 \dots x_n\middle| y_2 \dots y_m\middle| x_1, y_1, z_1 \dots z_p \middle| R, x_1 = y_1 \right>_{ab}$.

The $0$, identity, and symmetry morphisms in this prop have the following definitions (the definitions are kind of obvious as we choose as relations exactly what we want to happen)

\begin{itemize}
  \item $0 = \left<\middle|\middle|\middle|\right>_{ab}$ (the trivial group with no generators)
  \item $id = \left<x\middle|y\middle|\middle|x = y\right>_{ab}$ (this is the group $\mathbb{Z}$ presented with two generators)
  \item $\sigma = \left<x_1,x_2\middle|y_1,y_2\middle|\middle| x_1 = y_2, x_2 = y_1\right>_{ab}$
\end{itemize}

Now we need to find a bialgebra in this category. Again, the natural definition is quite obvious:
\begin{itemize}
\item $\mu = \left<x_1, x_2\middle|y\middle|\middle|x_1 + x_2 = y\right>_{ab}  $ (the group is $\mathbb{Z}^2$)
\item $\eta = \left<\middle|y\middle|\middle|y = 0\right>_{ab} $ (the trivial group)
\item $\Delta = \left<x\middle|y_1,y_2\middle|\middle|x=y_1=y_2\right>_{ab}$ (the group is $\mathbb{Z}$)
\item $\epsilon = \left<x\middle|\middle|\middle|\right>_{ab} $ (the group is $\mathbb{Z}$)
\end{itemize}

One could represent them diagrammatically like this:

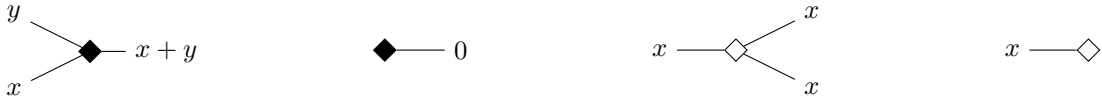
\begin{tikzpicture}[baseline=(current bounding box.center)]
  \node (A) at (0,0) {$x$};
  \node (B) at (0,1) {$y$};
  \node (C) [style=bmonoid] at (1,0.5) {};
  \node (D)  at (2,0.5) {$x + y$};
  \draw (A) -- (C);
  \draw (B) -- (C);
  \draw (C) -- (D);
\end{tikzpicture}
\hfill
\begin{tikzpicture}[baseline=(current bounding box.center)]
  \node (A) [style=bmonoid] at (0,0.5) {};
  \node (B)  at (1,0.5) {$0$};
  \draw (A) -- (B);
\end{tikzpicture}  
\hfill
\begin{tikzpicture}[baseline=(current bounding box.center)]
  \node (A) at (0,0) {$x$};
  \node (B) at (0,1) {$x$};
  \node (C) [style=monoid] at (-1,0.5) {};
  \node (D)  at (-2,0.5) {$x$};
  \draw (A) -- (C);
  \draw (B) -- (C);
  \draw (C) -- (D);
\end{tikzpicture}
\hfill
\begin{tikzpicture}[baseline=(current bounding box.center)]
  \node (A) [style=monoid] at (0,0.5) {};
  \node (B)  at (-1,0.5) {$x$};
  \draw (A) -- (B);
\end{tikzpicture}

These are the same diagrams as in subsection~\ref{sec:monoids}, except now their meaning are different: every wire now bears a \emph{generator}. The first diagrammatic expression means that we have three generators, and there is an equation stating that the output generator is the sum of the input generators.

We can now look at flow equivalence using this bialgebra. Morphisms $0 \to 0$ are groups with no input/output generators and only internal generators, so they are just plain groups, upto isomorphism. We obtain immediately:

\begin{proposition}
  Let $M$ be a $n \times n$ matrix. The group $\left<x_1 \dots x_n \middle| Mx = x\right>_{ab}$ (upto isomorphism) is an invariant of flow equivalence.
In other words, the group $\mathbb{Z}^n / (M - I)\mathbb{Z}^n$ is an invariant of flow equivalence.
\end{proposition}  
This is the celebrated Bowen-Franks group \cite{BowenFranks}.

Now there is no reason to do this construction starting with a group, we can do the same with commutative monoids with no change, to obtain:
\begin{proposition}
  Let $M$ be a $n \times n$ matrix. The commutative monoid $\left<x_1 \dots x_n \middle| Mx = x\right>_{ab}$ is an invariant of flow equivalence.
\end{proposition}  
Everything we did in this subsection was done informally by Hillman \cite{Hillman}, who also showed that the monoid typically does not contain more information than the group in relevant cases (Theorem 29 in \cite{Hillman}). For irreducible matrices, the commutative \emph{semigroup} $\left<x_1 \dots x_n \middle| Mx = x\right>_{ab}$ is actually a group, and therefore equal to the Bowen-Franks group.
The commutative \emph{monoid} is therefore just a group with an extraneous unit element.

Now the Bowen-Franks group is an invariant of flow equivalence, but if we want to find a finer invariant for strong shift equivalence, we need to find a nontrivial representation of the rectangle of Proposition~\ref{def:matricespolynomial}, i.e. a nontrivial morphism for $\mu$ (adding two generators) and $\Delta$ (duplicating a generator)

If we stay in the realm of groups and monoids, we probably won't obtain anything stronger than the previous group/monoid.
One way to go is using $\mathbb{Z}_+[t]$-semimodules \cite[chapter 14]{Golan}.
In this new setting, morphisms will not be monoids, but semimodules over $\mathbb{Z}_+[t]$. This is just a standard commutative monoid $M$ , equipped with an action of $\mathbb{Z}_+[t]$, denoted multiplicatively, which satifies the obvious axioms: $a(m+n) = am+an, (a+b)m = am+bm, (ab)m = a(bm), a0=0, 0m=m, 1m=m$ for $a,b \in \mathbb{Z}_+[t]$ and $m \in M$.

We can do the same theory as before for presentations of $\mathbb{Z}_+[t]$-semimodules. Now $\left< X | R \right>$ will represent the $\mathbb{Z}_+[t]$-semimodule with generators $X$ and equations $R$. Such a semimodule is for example $\left< x,y | t^2 x = 2ty + (t^2 + 3t)x, x + y = t \right>$.

Everything works exactly the same before, with the added benefit that we know have a nontrivial morphism for our bialgebra, as in Proposition~\ref{def:matricespolynomial}: multiplication by $t$:
\[ \left<x\middle|y\middle|\middle| y = tx\right>\]

Doing this  we obtain immediately:

\begin{proposition}
  Let $M$ be a $n \times n$ matrix. The $\mathbb{Z}_+[t]$-semimodule  $\left<x_1 \dots x_n \middle| tMx = x\right>$ is an invariant of strong shift equivalence.
\end{proposition}

In the specific case of $M = \begin{pmatrix} 2 \end{pmatrix}$, this semimodule is $\left<t,x \middle| 2tx = x\right>$, which is the monoid $\mathbb{Z}_+[1/2]$ where $t$ ``acts'' by multiplication by $1/2$.

we can use the same construction with $\mathbb{Z}[t]$-modules instead of $\mathbb{Z}_+[t]$-semimodules. In that case, what we obtain is the cokernel of $I - tM$, which is well known to be equal to the dimension group~\cite{Krieger}. It suggests that the $\mathbb{Z}_+[t]$-semimodule is related to the positive part of the dimension group (seen as the ordered group), this should be examined more closely.

\section{Conclusion}

Using category theory, we have found a way to redescribe the strong shift equivalence problem in a purely categorical way: find identities satisfied in all traced bialgebras with a distinguished morphism.

As a consequence of this approach, we obtain a systematic way of obtaining new invariants for strong shift equivalence, and the hope to actually decide strong shift equivalence using these new invariants.

Here are a few natural questions:

\begin{itemize}
\item Is is sufficient to consider bialgebras (in the sense of linear algebra) of finite dimension to decide strong shift equivalence ? Or at the very least bialgebras that are finitely generated as bialgebras ? This might open a way to actually prove decidability of strong shift equivalence
\item In practice, one observe that many traced props are actually already compact closed. If we restrict ourselves to compact closed categories, do we obtain the same theory and identities ?
\item Is it possible to define a similar theory for subshifts of finite type in 2D?
\end{itemize}  
\section*{Acknowledgements}
E.J. wants to thank Mike Boyle for discussions about flow equivalence and positive equivalence.

\bibliographystyle{amsplain}

\providecommand{\bysame}{\leavevmode\hbox to3em{\hrulefill}\thinspace}
\providecommand{\MR}{\relax\ifhmode\unskip\space\fi MR }
\providecommand{\MRhref}[2]{%
  \href{http://www.ams.org/mathscinet-getitem?mr=#1}{#2}
}
\providecommand{\href}[2]{#2}

\clearpage
\appendix
\section{Proof of Proposition   \ref{prop:tracedcompletion}}
\label{app:proofalacon}
In this section, we prove that $Q$ is indeed a traced prop.

We will look at all axioms of Definition~\ref{def:prop} and Definition~\ref{def:tracedprop}.
\subsection{Preliminary remarks}

Before doing the proofs, we recall the equations of the equivalence relations, and give them names:

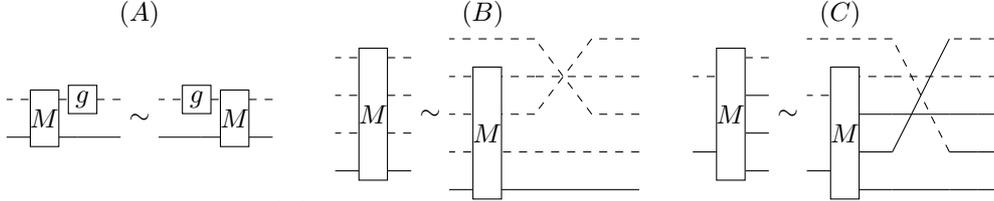
\begin{tabular}{ccc}
  $(A)$ & $(B)$ & $(C)$ \\
\end{tabular}

First remark by equation $(A)$, when $g$ as a permutation, that we can always permute the dashed wires in a diagram.

Second, let $(D)$ be the equation:
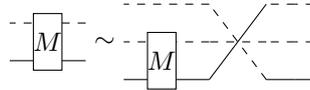
\begin{center}
  { %
 }
\end{center}

$(D)$ is a particular case of equation $(C)$. However we claim that $(C)$ is a consequence of $(A)+(B)+(D)$.
Indeed, by using equation $(D)$, then using naturality of the symmetry,  we obtain:

{ %
 }

Now by using equation $(A)$ to put the fourth wire on top, then equation $(B)$ we obtain

{ %
 }

Now we put back the top wire in third position, and we use equation $(D)$ again:

{ %
 }

Now, the advantage of using equations $(A)+(B)+(D)$ instead of $(A)+(B)+(C)$ is that we can apply equation $(A)$ to $(B)$ and $(D)$ which shows immediately that the mirror equations of $(B)$ and $(D$), i.e. when we apply the construction on the input rather than on the output, is still valid. The same is therefore also true of equation $(C)$. We will call $(B')$, $(C')$, $(D')$ the mirror equations of $(B)$ and $(C)$. Equation $(C')$ for instance is:

{ %
 }

\subsection{The composition is well defined}
For this , it is enough to prove the following equalities about the relation $\sim$, corresponding to equations $(A)$, $(B)$ and $(D$)

${%
}$

The first two are true by a trivial use of $(A)$.
The third one is true by an application of $(B)$ then $(A)$ to permute the wires. The fourth one is true by an application of $(B)$.

The fifth equation is true by an application of $(D)$ and then $(A)$ to permute the wires.
The last equation is true by an application of $(B)$

\subsection{The tensor product is well defined}

Proofs about the tensor product are essentially easy and left to the
reader: The definition we take of the tensor product in the category
$Q$ only consists in placing the two morphisms side by side, and
exchanging the order of the input/outputs so that the dashed wires are
on top, which is an operation as innocuous as possible, as equation
$(A)$ makes us able to permute the dashed wires.

\subsection{The composition is associative}

This amount to show the following equation, which is obvious (the two terms are not only equivalent, but equal):

${

}$

\subsection{The tensor product is associative, and has a neutral element}

This is again trivial.

\clearpage
\subsection{Link between the tensor product and the composition}
We need to prove $([M_1,k_1]\circ [M_2,k_2])\otimes ([M_3,k_3]\circ [M_4,k_4]) \sim ([M_1,k_1]\otimes [M_3,k_3])\circ ([M_2,k_2]\otimes [M_4,k_4])$. Graphically this means:

${%
}$

This is true by first applying naturality of the symmetry to exchange the positions of the boxes representing $M_2$ and $M_3$, and then using equation $(A)$ to permute the wires.

\subsection{Alternate form for the composition}
Before doing the rest of the proofs, we will give  an alternate form
for the composition:

${%
}
$

Indeed, let's start with the definition of the composition and apply
equation $(A)$ with $g = N$, then permute the first three wires:

$
{ %
  }
$

By naturality of the symmetry, this is equal to:

{%
  }
We may then apply equation $(D'$).

\subsection{The composition has an neutral element}
This is obvious from the previous form.

\subsection{The symmetry}. The symmetry is the morphism $(\sigma, 0)$,
where $\sigma$ is the symmetry in the original category. It is easy to
see that it satisfies the axioms necessary for a prop using the
previous form for the composition. 

\subsection{The trace axioms (tightening, yanking, sliding, strength}

All axioms are easy to prove, using the alternate definition of the
composition from above. Yanking in particular is equation $(D)$ in the
particular case where $M$ is $id$. Sliding is essentially equation
$(A)$. Strength and Tightening  are essentially reordering of the
wires using equation $(A)$.

\end{document}